\documentclass[12pt,preprint]{amsart}

\usepackage{amsmath, latexsym, amsfonts, amssymb, amsthm}
\usepackage{graphicx, color,hyperref,dsfont,epsfig,caption,wrapfig,subfig}
\usepackage{mathtools}
\usepackage{movie15}
\hypersetup{
	linktoc=page,
	linkcolor=red,          
	citecolor=blue,        
	filecolor=blue,      
	urlcolor=cyan,
	colorlinks=true           
}

\numberwithin{equation}{section}

\newtheorem{theorem}{Theorem}[section]
\newtheorem{lemma}[theorem]{Lemma}
\newtheorem{proposition}[theorem]{Proposition}
\newtheorem{corollary}[theorem]{Corollary}

\newtheorem{remark}[theorem]{Remark}

\newcounter{thmc}

\newtheorem{thmcite}[thmc]{Theorem}

\theoremstyle{definition}

\renewcommand{\tilde}{\widetilde}          
\DeclareMathSymbol{\leqslant}{\mathalpha}{AMSa}{"36} 
\DeclareMathSymbol{\geqslant}{\mathalpha}{AMSa}{"3E} 
\DeclareMathSymbol{\eset}{\mathalpha}{AMSb}{"3F}     
\renewcommand{\leq}{\;\leqslant\;}                   
\renewcommand{\geq}{\;\geqslant\;}                   

\newcommand{\C}{\mathbb{C}}
\newcommand{\R}{\mathbb{R}}
\newcommand{\Z}{\mathbb{Z}}
\newcommand{\N}{\mathbb{N}}
\newcommand{\Q}{\mathbb{Q}} 
\newcommand{\D}{\mathbb{D}}

\newcommand{\E}{\mathds{E}}
\newcommand{\X}{\bm{\mathrm X}} 
\newcommand{\Y}{\bm{\mathrm Y}} 
\newcommand{\M}{\bm{\mathrm M}} 
\newcommand{\V}{\bm{\mathrm V}} 
\newcommand{\B}{\bm{\mathcal B}}

\renewcommand{\P}{\mathds{P}}

\newcommand{\ps}[1]{\langle #1 \rangle}

\def\FaLi{\mathfrak{C}_\gamma(\alpha_0,\alpha_1,\alpha_\infty)}

\def\eps{\varepsilon}

\def\bi{\begin{itemize}}
	\def\ei{\end{itemize}}

\def\bnum{\begin{enumerate}}
	\def\enum{\end{enumerate}}
\newcommand{\hyper}[3]{ {}_3F_2\left(\begin{matrix}#1 \\ #2\end{matrix} \middle\vert #3 \right)}
\def\<#1{\langle #1 \rangle}

\def\M{\mathbf{M}}

\newcommand{\norm}[1]{\left\lvert#1\right\rvert}
\newcommand{\expect}[1]{\mathbb{E}\left[#1\right]}
\usepackage{bm}
\usepackage{bbm}

\title[Three-point correlation functions in the $\mathfrak{sl}_3$ Toda CFT]{Three-point correlation functions in the $\mathfrak{sl}_3$ Toda theory II: the Fateev-Litvinov formula}

\author{Baptiste Cercl\'e}
\email{baptiste.cercle@universite-paris-saclay.fr}
\address{Laboratoire de Math\'ematiques d'Orsay, B\^atiment 307. Facult\'e des Sciences d'Orsay, Universit\'e Paris-Saclay. F-91405 Orsay Cedex, France}

\begin{document}

	\maketitle
	\begin{abstract}
		Toda Conformal Field Theories (CFTs) form a family of two-dimensional CFTs indexed by semisimple and complex Lie algebras. One of their remarkable features is that they are natural generalizations of Liouville CFT that enjoy an enhanced level of symmetry, prescribed by $W$-algebras. They likewise admit a probabilistic formulation in terms of Gaussian Multiplicative Chaos.
		
		Based on this probabilistic framework, this second article in a two-part series is dedicated to providing a first step towards integrability of these theories. In this perspective we prove the Fateev-Litvinov formula for a family of three-point correlation functions associated to the $\mathfrak{sl}_3$ Toda CFT. This result is the analog of the celebrated DOZZ formula in Liouville CFT. Our method of proof features techniques inspired by the physics literature together with probabilistic ones that naturally arise within the setting of Toda theories. 
	\end{abstract}
	
	
	
	\section{Introduction}
	
	
	\subsection{Toda Conformal Field Theories}
	\subsubsection{Liouville Conformal Field Theory}
	In a recent work~\cite{GKRV_Segal}, Guillarmou, Kupiainen, Rhodes and Vargas proposed a few months ago a rigorous derivation of Segal's axioms ---a functorial definition proposed in 1987 by Segal for two-dimensional Conformal Field Theory (2d CFT hereafter)--- within the framework of Liouville CFT. This achievement is the culmination of a series of work, initiated by David, Kupiainen, Rhodes and Vargas in 2014~\cite{DKRV} and aimed at recovering within a mathematical setting the predictions made in the physics literature for Liouville CFT. 
	
	Introduced by Polyakov in his 1981 seminal work~\cite{Pol81}, Liouville CFT has indeed drawn considerable attention in both the physics and mathematics community. Initially designed as a model of random surfaces in the context of 2d quantum gravity, Liouville CFT has since proved to be key in the understanding of 2d CFT, all the more following the introduction by Belavin-Polyakov-Zamolodchikov~\cite{BPZ} of a general procedure, referred to as \textit{conformal bootstrap}, for solving models of 2d CFT. The importance of Liouville CFT is stressed \textit{e.g.} in~\cite{Teschner_revisited, Nakayama} and the references therein,  to which we refer for additional details.
	
	Following these developments providing a rigorous framework for Liouville CFT has been a challenge for mathematicians and led to major important advances in the understanding of two-dimensional random geometry. One of the program proposed in this direction is the one developed by David-Guillarmou-Kupiainen-Rhodes-Vargas, who successfully addressed this issue by providing a mathematical proof of the DOZZ formula~\cite{KRV_DOZZ}, predicted by physicists~\cite{DO94, ZZ96} in the 90's, as well as a rigorous derivation of the conformal bootstrap procedure~\cite{GKRV}, one of the key inputs in the study of 2d CFT in the physics literature. On a similar perspective and building on the intrinsic connection between Liouville CFT and Conformal Loop Ensembles (CLEs)~\cite{Sh_CLE, SW, MSW}, the imaginary DOZZ formula was shown by Ang, Cai, Wu and Sun~\cite{AS_DOZZ} to describe certain CLE observables. The study of connections between Liouville CFT and random planar maps has also proved to be fundamental and let to numerous breakthroughs~\cite{LeG13, Mie13, MS15a, MS16a, MS16b, HoSu19}, all the more thanks to the construction of the so-called Liouville Quantum Gravity metric~\cite{DDDF,DFGPS, GM20} (we refer to~\cite{DDG} for an overview of the topic).
	
	\subsubsection{From Liouville to Toda Conformal Field Theories}
	Inspired by the Belavin-Polyakov-Zamolodchikov seminal work, Zamolodchikov proposed in 1985~\cite{Za85} a framework designed to extend this machinery to models that enjoy, in addition to conformal invariance, an enhanced level of symmetry. These additional symmetries, called \textit{higher-spin} or \textit{W}-symmetries, are encoded by $W$-algebras, which are Vertex Operator algebras that contain the Virasoro algebra as a subalgebra. 
	Toda CFTs, a family of two-dimensional CFTs indexed by semisimple and complex Lie algebras $\mathfrak{g}$, provide natural extensions of Liouville CFT within this setting. Indeed Liouville CFT is actually the simplest case of a Toda CFT, in that it actually corresponds to the choice of $\mathfrak g=\mathfrak{sl}_2$ for the underlying Lie algebra. However for generic $\mathfrak{g}$ the algebras of symmetry of these 2d CFTs are no longer given by the Virasoro algebra but rather by $W$-algebras, which makes their study particularly interesting from the point of view of representation theory of $W$-algebras (more on this topic can be found \emph{e.g.} in ~\cite{Arakawa_intro}), but also from the perspective of $W$-symmetry (in this respect we refer to the review~\cite{BouSch} and the references therein) and for their links with 2d Quantum Field Theories with Kac-Moody symmetry (see for instance~\cite{BFFOrW1,BFFOrW2} for their interplays with Wess-Zumino-Witten models).  However and unlike Liouville CFT, Toda CFTs are still far from being completely understood, despite having initiated a huge amount of work in the physics literature, all the more thanks to their links with certain four-dimensional gauge theories ---since they are the general setting for the AGT correspondence~\cite{AGT} (see also~\cite{MO,SV_AGT} for a mathematical take on this correspondence)--- and models of statistical physics (for instance the reference~\cite{magnet} provides a journalistic survey on the connection, first unveiled by Zamolodchikov~\cite{Za_E8}, between the Ising Model in a Magnetic Field at criticality and a certain Toda theory associated to the exceptional Lie algebra $E_8$). Their mathematical comprehension is more recent and has been initiated by Rhodes, Vargas and the author in~\cite{Toda_construction}, where a mathematically rigorous interpretation of Toda path integral has been proposed. 
	
	\subsubsection{Path integral for Toda CFTs}
	Indeed one specificity of this family of CFTs is that they can be defined using a path integral approach, thanks to which they admit a probabilistic representation that allows their mathematical study. Namely Toda CFTs provide a way of picking at random a function from a Riemannian surface $\Sigma$ to an Euclidean space $\mathfrak a\simeq \R^r$, the \textit{Toda field} $\varphi$. This Euclidean space comes equipped with a scalar product $\ps{\cdot,\cdot}$ as well as a special basis $(e_1,\cdots,e_r)$, both inherited from the underlying Lie algebra structure. The path integral defines heuristically the law of the Toda field via
	\begin{equation}\label{eq:path_integral}
		\langle F(\varphi) \rangle_{T,g} \coloneqq \int_{\X:\Sigma\to\R^r} F( \X)e^{-S_{T,\mathfrak g}( \X,g)}D \X
	\end{equation}
	where $D \X$ should stand for a \lq\lq uniform measure" over the space of square integrable $\R^r$-valued maps defined on $\Sigma$ and $ S_{T,\mathfrak{g}}$ is the Toda action given by 
	\begin{equation}\label{eq:Toda_action}
		S_{T,\mathfrak{g}}(\X,g)\coloneqq  \frac{1}{4\pi} \int_{\Sigma}  \Big (  \langle\partial_g \X(x), \partial_g \X(x) \rangle_g   +R_g \langle Q, \X(x) \rangle +4\pi \sum_{i=1}^{r} \mu_ie^{\gamma    \langle e_i,\X(x) \rangle}   \Big)\,{\rm v}_{g}(dx).
	\end{equation}
	Here $g$ is a Riemannian metric over $\Sigma$ with associated scalar curvature $R_g$, gradient $\partial_g$ and volume form ${\rm v}_g$, while $\gamma\in(0,\sqrt 2)$\footnote{The range of values $(0,\sqrt 2)$ for the coupling constant $\gamma$ only differs from the one commonly encountered in Liouville theory by a conventional matter. Namely the coupling constant from Liouville and the one of Toda are related by $\gamma\leftrightarrow \sqrt 2\gamma$. This convention accounts for the fact that some elements of the basis have squared norm $2$.} is the coupling constant, $Q\in\R^r$ the background charge and the constants $\mu_i$, $1\leq i\leq r$, are positive and referred to as the cosmological constants. 
	
	Within this framework, Vertex Operators are functionals of the Toda field that depend on an insertion point $z\in\Sigma$ as well as a weight $\alpha\in\R^r$, and are formally defined by setting $V_{\alpha}(z)=e^{\ps{\alpha,\varphi(z)}}$. The correlation functions then take the form
	\begin{equation}\label{eq:formal_correl}
		\langle \prod_{k=1}^NV_{\alpha_k}(z_k) \rangle_{T,g} \coloneqq \int_{\X:\Sigma\to\R^r} \prod_{k=1}^Ne^{\ps{\alpha_k,\X(z_k)}}e^{-S_{T,\mathfrak g}( \X,g)}D \X.
	\end{equation}
	
	\subsubsection{$\mathfrak{sl}_3$ Toda three-point correlation functions}
	It was proved in~\cite{Toda_construction} that it is possible to interpret this path integral using a probabilistic framework based on \emph{Gaussian Free Fields} (GFF in the sequel) and \emph{Gaussian Multiplicative Chaos} (GMC hereafter). Namely, under some assumptions on the weights $(\alpha_1,\cdots,\alpha_N)$ called \emph{Seiberg bounds} the correlation functions admit a probabilistic representation involving correlated GMC measures~\cite[Theorem 3.1]{Toda_construction}. In particular this allows to define probabilistic $\mathfrak{sl}_3$ Toda three-point correlation functions $C_\gamma(\alpha_0,\alpha_1,\alpha_\infty)$ by setting:
	\begin{equation}\label{eq:DOZZ_Toda}
		C_\gamma(\alpha_1,\alpha_2,\alpha_3)\coloneqq \mathbb{E}\left[\prod_{i=1}^2\frac{\Gamma(s_i)\rho_i(\alpha_1,\alpha_2,\alpha_3)^{-s_i}}{\gamma\mu_i^{s_i}}\right]
	\end{equation}
	where the exponents $s_i$ depend on $\gamma$ and the weights $\alpha_k$, $1\leq k\leq 3$, while 
	\begin{equation}\label{eq:rho}
		\rho_i(\alpha_1,\alpha_2,\alpha_3)\coloneqq\int_{\mathbb C} \frac{\norm{x}_+^{\gamma\ps{\alpha_1+\alpha_2+\alpha_3,e_i}}}{\norm{x}^{\gamma\ps{\alpha_1,e_i}}\norm{x-1}^{\gamma\ps{\alpha_2,e_i}}}M^{\gamma e_i}(d^2x)
	\end{equation}
	with $\left(M^{\gamma e_i}(d^2x)\right)_{i=1,2}$ a pair of correlated GMC measures, well-defined for $\gamma\in(0,\sqrt 2)$ ---we defer to Section~\ref{sec:toda} for more details. We are actually able to extend this definition as soon as $(\alpha_0,\alpha_1,\alpha_\infty)$ meets certain requirements, much weaker than the Seiberg bounds, and which we describe in Theorem~\ref{thm:analycity} where the set of such weights is denoted by $\mathcal{A}_3$. 
	This interpretation of the Toda correlation functions was proved to be meaningful in the context of $W$-symmetry by Huang and the author in~\cite{Toda_OPEWV}, where this enhanced level of symmetry was shown to be present within the probabilistic model via the existence of so-called \emph{Ward identities}.
	
	One of the most important steps in the perspective of unifying this probabilistic approach and the one implemented in the physics literature is a rigorous derivation of the three-point correlation functions $C_\gamma(\alpha_0,\alpha_1,\alpha_\infty)$ associated to the $\mathfrak{sl}_n$ Toda CFT, whose expressions have been proposed by Fateev and Litvinov~\cite{FaLi0,FaLi1}. 
	In the case where $\mathfrak g=\mathfrak{sl}_3$, the proposal of Fateev-Litvinov is based on the special functions $\Upsilon$ and $l(x)\coloneqq \frac{\Gamma(x)}{\Gamma(1-x)}$ (see Section~\ref{sec:toda_end} for more background) and takes the form
	\begin{equation}\label{eq:fali}
		\begin{split}
			\mathfrak{C}_\gamma(\alpha_0,\alpha_1,\alpha_\infty)=&\left(\pi\mu l\left(\frac{\gamma^2}{2}\right)\left(\frac{\gamma}{\sqrt 2}\right)^{2-\gamma^2}\right)^{\frac{\ps{2Q-\bar\alpha,\rho}}{\gamma}}\\
			&\quad\frac{\displaystyle\Upsilon'(0)^2\Upsilon(\kappa)\prod_{e\in\Phi^+}\Upsilon(\ps{Q-\alpha_0,e})\Upsilon(\ps{Q-\alpha_\infty,e})}{\prod_{1\leq j,k\leq 3}\displaystyle \Upsilon\left(\frac\kappa 3 + \ps{\alpha_0-Q,h_j})+\ps{\alpha_\infty-Q,h_k}\right)}
		\end{split}
	\end{equation} 
	where the weight $\alpha_1$ has to be taken of the form $\alpha_1=\kappa \omega_2$ for $\kappa$ real, with $\omega_2=\frac{2e_2+e_1}{3}$ and where the $(h_i)_{1\leq i\leq3}$ are special vectors in $\R^2$ (see Figure~\ref{fig:weyl} below). We also denoted by $\Phi^+\coloneqq\{e_1,e_2,e_1+e_2\}$ the set of positive roots.
	
	\begin{center}
		\includegraphics[height=8cm]{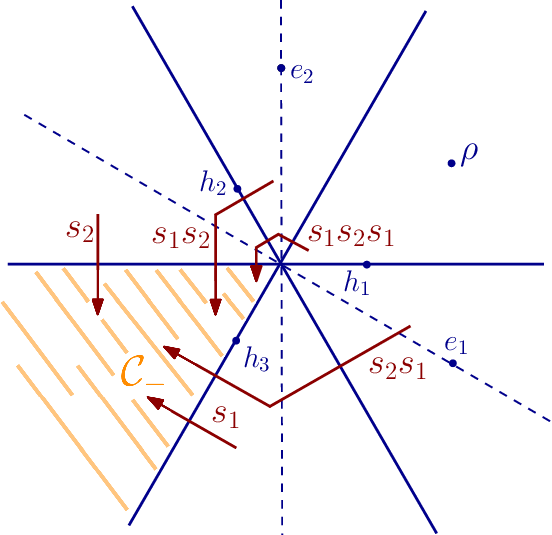}
		\captionof{figure}{The root system associated to $\mathfrak{sl}_3$}
		\label{fig:weyl}
	\end{center}
	
	A first step towards the probabilistic derivation of this formula has been achieved in the first article of this two-part series~\cite{Toda_correl1} where Toda \emph{reflection coefficients}, which are key in the perspective of recovering the formula~\eqref{eq:fali}, were shown to admit a probabilistic representation. This representation is crucial both to make sense of certain correlation functions and to study some of their properties, as well as for extending the range of validity for which the expression $C_\gamma(\alpha_0,\alpha_1,\alpha_\infty)$ makes sense. Indeed one of the specificities of the formula defined by Equation~\eqref{eq:fali} is the existence of transformations $\hat s:\R^2\to\R^2$ and reflection coefficients $R_s:\R^2\to\R$, indexed by elements $s$ of the Weyl group $W$ associated to $\mathfrak{sl}_3$ (see Figure~\ref{fig:weyl} for a simple illustration, but we refer to Subsection~\ref{subsec:lie} for additional details), such that 
	\begin{equation}
		\mathfrak{C}_\gamma(\alpha_0,\alpha_1,\alpha_\infty)=R_s(\alpha_0)\mathfrak{C}_\gamma(\hat s\alpha_0,\alpha_1,\alpha_\infty)
	\end{equation}
	for any such $s$. However from a probabilistic perspective, the probabilistically defined correlation functions can make sense only for at most one single element of $W$, which is the one such that $s(\alpha_0 -Q)\in\mathcal{C}_-$ (see again Figure~\ref{fig:weyl} and Subsection~\ref{subsec:lie}). Nevertheless we are able to prove that the following holds true:
	\begin{theorem}\label{thm:refl}
		For any $\gamma\in(0,\sqrt2)$, assume that $\alpha_1$ is of the form $\alpha_1=\kappa\omega_2$ with $\kappa\in\R$, and set for $\alpha_\infty\in\R^2$
		\begin{equation}
			\mathcal{U}(\alpha_1,\alpha_\infty)\coloneqq\left\{\alpha_0\in\R^2, (\hat s\alpha_0,\alpha_1,\alpha_\infty)\in\mathcal{A}_3\text{ for some }s\in W\right\}.
		\end{equation}
		Then the extension over $\mathcal{U}(\alpha_1,\alpha_\infty)$ of the three-point correlation functions $C_\gamma(\alpha_0,\alpha_1,\alpha_\infty)$ defined by setting
		\begin{equation}\label{eq:ext_3pts}
			C_\gamma(\alpha_0,\alpha_1,\alpha_\infty)\coloneqq R_s(\alpha_0)C_\gamma(\hat s\alpha_0,\alpha_1,\alpha_\infty)\text{ where $s\in W$ is such that }s(\alpha_0-Q)\in\mathcal{C}_-
		\end{equation}
		is analytic in a complex neighborhood of $\mathcal{U}(\alpha_1,\alpha_\infty)$ viewed as a subset of $\C^2$. The reflection coefficients are given, in agreement with~\cite{Toda_correl1}, by
		\begin{equation}\label{equ:expression_toda_refl}
			\begin{split}
				R_s(\alpha)&=\epsilon(s)\frac{A\left(s(\alpha-Q)\right)}{A(\alpha-Q)},\quad\text{where}\\
				A(\alpha)&=\prod_{i=1}^2\left(\mu_i\pi l\left(\frac{\gamma^2}{2}\right)\right)^{\frac{\ps{\alpha,\omega_i}}\gamma}\prod_{e\in\Phi^+}\Gamma\left(1-\frac{\gamma}2\ps{\alpha,e}\right)\Gamma\left(1-\frac{1}\gamma\ps{\alpha,e}\right).
			\end{split}
		\end{equation}
	\end{theorem}
	Recall that due to the convention on the coupling constant $\gamma$ the range of values $(0,\sqrt 2)$ is the optimal one to probabilistically make sense of the correlation functions. 
	
	Based on the above, the main statement of the present document is an integrability result for the $\mathfrak{sl}_3$ Toda CFT. Namely we provide, under the assumption on the coupling constant that $\gamma\in[1,\sqrt2)$, a rigorous derivation of the formula~\eqref{eq:fali} proposed for such three-point correlation functions Toda correlation functions, in agreement with predictions from the physics literature:
	\begin{theorem}\label{thm:main_result}
		For $\gamma\in[1,\sqrt2)$ assume that $\alpha_1$ is of the form $\alpha_1=\kappa\omega_2$ with $\kappa\in\R$, and that $(\alpha_0,\alpha_1,\alpha_\infty)$ is such that $\alpha_0\in\mathcal{U}(\alpha_1,\alpha_\infty)$.
		Then 
		\begin{equation}\label{eq:main_result}
			C_\gamma(\alpha_0,\alpha_1,\alpha_\infty) = \mathfrak{C}_\gamma(\alpha_0,\alpha_1,\alpha_\infty).
		\end{equation}
		
		For the general case where $\gamma\in(0,\sqrt2)$, the following shift equations hold true
		\begin{equation}\label{eq:shift_main}
			\frac{C_\gamma(\alpha_0,\alpha_1,\alpha_\infty)}{\mathfrak{C}_\gamma(\alpha_0,\alpha_1,\alpha_\infty)}=\frac{C_\gamma(\alpha_0+\gamma e,\alpha_1,\alpha_\infty)}{\mathfrak{C}_\gamma(\alpha_0+\gamma e,\alpha_1,\alpha_\infty)}
		\end{equation}
		for every $e\in\Phi^+$.
	\end{theorem}
	The main obstruction at the time being that prevents one from deriving the equality~\eqref{eq:main_result} for the whole range of values for $\gamma$ is to prove that the shift equation dual to Equation~\eqref{eq:shift_main} holds true:
	\begin{equation}\label{eq:shift_demain}
		\frac{C_\gamma(\alpha_0,\alpha_1,\alpha_\infty)}{\mathfrak{C}_\gamma(\alpha_0,\alpha_1,\alpha_\infty)}=\frac{C_\gamma(\alpha_0+\frac2\gamma e,\alpha_1,\alpha_\infty)}{\mathfrak{C}_\gamma(\alpha_0+\frac2\gamma e,\alpha_1,\alpha_\infty)}\cdot
	\end{equation}
	Using the standard conventions for the coupling constant in Liouville theory, the above statements would correspond to a computation of a family of three-point correlation functions for either $\gamma\in(0,2)$ for Theorem~\ref{thm:refl} and Equation~\eqref{eq:shift_main}, and $\gamma\in(\sqrt 2,2)$ for Theorem~\ref{thm:main_result}, which is the most interesting range of values from the point of view of statistical physics.
	
	
	\subsection{Overview of the article and method of proof}
	In order to prove our main statement on the $\mathfrak{sl}_3$ Toda three-point correlation functions, our goal is to show that both shift equations~\eqref{eq:shift_main} and~\eqref{eq:shift_demain} hold true. For this purpose we will actually study certain \emph{four-point} correlation functions and prove that they admit two alternative expressions. The condition that these two expressions coincide will account for both shift equations~\eqref{eq:shift_main} and~\eqref{eq:shift_demain}. However and as opposed to Liouville theory making sense of such correlation functions  is far from obvious and requires some care. We summarize below the different steps of the proof that will allow to infer Theorem~\ref{thm:main_result}.
	
	\subsubsection{Definition and analytic continuation of Toda correlation functions}The first part of this document, Section~\ref{sec:toda}, will be devoted to recalling the probabilistic framework introduced in~\cite{Toda_construction} to provide a rigorous meaning to the Toda correlation functions formally defined by Equation~\eqref{eq:formal_correl}. In this section we also unveil a probabilistic expression that allows to analytically extend the probabilistic definition of the correlation functions beyond the bounds on the weights $\alpha_1,\cdots,\alpha_N$ that are usually assumed to hold so as to define them~\cite[Theorem 3.1]{Toda_construction}. Providing such an extension of the range of values required to make sense of the correlation functions is a key step in the derivation of Theorem~\ref{thm:main_result} and a major technical difficulty compared to the proof of the DOZZ formula~\cite{KRV_DOZZ}, where defining such an extension was actually unnecessary.
	
	The method developed to define such an analytic continuation of the correlation functions is based on a new analytic continuation of moments of GMC measures $M^\gamma(d^2x)$. Namely if we set for $\frac2\gamma<\alpha<Q$ and $p<1$
	\begin{equation}
		\begin{split}
			&F(p)\coloneqq \frac{1}{\Gamma(-p)}\int_{\R}e^{-pc}\expect{e^{-e^ cI(\alpha)}-\mathds 1_{p>0}\left(1+\mathds{1}_{p>\frac{2}{\gamma}(Q-\alpha)}e^{\frac{2}{\gamma}(Q-\alpha)c}R(\alpha)\right)}dc\\
			&\text{with}\quad I(\alpha)\coloneqq \int_\D\norm{x}^{-\gamma\alpha}M^\gamma(d^2x)
		\end{split}
	\end{equation}
	then when $p<\frac{2}{\gamma}(Q-\alpha)$ we have $F(p)=\expect{I(\alpha)^p}$. Moreover we will justify that $F$ can be analytically extended in a complex neighborhood of $(-\infty,1)$ despite $\expect{I(\alpha)^p}$ being infinite for $\frac{2}{\gamma}(Q-\alpha)<p<1$: in that sense $F$ provides an analytic continuation of moments of the GMC measure $I(\alpha)$ beyond the range of values for which they are well-defined.
	An additional consequence of this method that we develop in the present document is the following probabilistic representation of the DOZZ formula~\cite{DO94, ZZ96} for \emph{Liouville CFT} beyond the assumptions made in~\cite{KRV_DOZZ}:
	\begin{corollary}\label{cor:DOZZ}
		Given $\gamma\in(0,2)$, assume that $\alpha_1,\alpha_2,\alpha_3\in\R$ are such that $\alpha_k<Q\coloneqq\frac\gamma2+\frac2\gamma$ for $1\leq k\leq 3$ with the assumption that $s\coloneqq\alpha_1+\alpha_2+\alpha_3-2Q$ satisfies $s>-\gamma$, and set
		\[
		C_\gamma(\alpha_1,\alpha_2,\alpha_3)\coloneqq \int_\R e^{sc}\expect{e^{-\mu e^{\gamma c}\rho(\alpha_1,\alpha_2,\alpha_3)}-\mathcal{R}_{\bm\alpha}(c)}dc
		\]
		where $\rho(\alpha_1,\alpha_2,\alpha_3)$ has been introduced in~\cite[Equation (2.17)]{KRV_DOZZ} and is defined using the GMC measure $M^\gamma(d^2x)$:
		\begin{align*}
			\rho(\alpha_1,\alpha_2,\alpha_3)&=\int_\C \frac{\norm{x}_+^{\gamma(\alpha_1+\alpha_2+\alpha_3)}}{\norm{x}^{\gamma\alpha_1}\norm{x-1}^{\gamma\alpha_2}}M^\gamma(d^2x),\quad\text{while}\\
			\mathcal{R}_{\bm\alpha}(c)&\coloneqq \sum_{\mathcal{U}\subset\{1,2,3\}}\mathds{1}_{s<\sum_{k\in\mathcal{U}}2(\alpha_k-Q)}\prod_{k\in\mathcal{U}}R(\alpha_k)e^{2(Q-\alpha_k)c}\quad\text{with}\\
			R(\alpha)&=-\left(\pi\mu l\left(\frac{\gamma^2}{4}\right)\right)^{\frac{2(Q-\alpha)}\gamma}\frac{\Gamma\left(\frac{2(\alpha-Q)}{\gamma}\right)\Gamma\Big(\frac{\gamma}2(\alpha-Q)\Big)}{\Gamma\left(\frac{2(Q-\alpha)}{\gamma}\right)\Gamma\Big(\frac\gamma2(Q-\alpha)\Big)}
		\end{align*}
		the Liouville reflection coefficient. Then $C_\gamma(\alpha_1,\alpha_2,\alpha_3)$ is analytic in a complex neighborhood of $\{(\alpha_1,\alpha_2,\alpha_3)\in(-\infty,Q),\quad s>-\gamma\}$. In particular
		\begin{equation}
			C_\gamma(\alpha_1,\alpha_2,\alpha_3)=C_\gamma^{DOZZ}(\alpha_1,\alpha_2,\alpha_3)
		\end{equation}
		where $C_\gamma^{DOZZ}(\alpha_1,\alpha_2,\alpha_3)$ is given by the DOZZ formula~\cite{DO94,ZZ96}.
	\end{corollary}
	The proof of this statement is based on the method introduced in the article~\cite{Toda_correl1}, based on a generalized Brownian path decomposition, to study the asymptotic expansion of Toda Vertex Operators. Compared to the bounds from~\cite{KRV_DOZZ}, the main improvement lies in the case where $\max\limits_{1\leq k\leq3}2(\alpha_k-Q)>s>-\gamma$. Under the assumption that $s>\max\limits_{1\leq k\leq3}2(\alpha_k-Q)$ the above probabilistic representation agrees with that from~\cite{KRV_DOZZ}.

	\subsubsection{Four-point correlation functions from BPZ differential equations}Having properly introduced Toda correlation functions and proved that they are analytic in the weights $\alpha_1,\cdots,\alpha_N$, we show in Section~\ref{sec:BPZ} and based on the main statements from ~\cite{Toda_OPEWV}, that certain four-point correlation functions can be explicitly computed up to a prefactor given by a three-point correlation function. Namely, the statement of Theorem~\ref{thm:BPZ} shows that (we defer to Section~\ref{sec:toda} for more details on the notations) if $\alpha=-\chi\omega_1$ with $\chi\in\{\gamma,\frac2\gamma\}$ and $\alpha_1=\kappa\omega_2$ for $\kappa<q$, then as soon as it makes sense
	\begin{equation}\label{eq:BPZ_intro}
		\begin{split}
			&\ps{V_{\alpha}(z)V_{\alpha_0}(0)V_{\alpha_1}(1)V_{\alpha_\infty}(\infty)}=\norm{z}^{\chi\ps{\omega_1,\alpha_0}}\norm{z-1}^{\frac{\chi\kappa}3}\mathcal H(z),\quad\text{where}\\
			&\mathcal H(z)=C_\gamma(\alpha_0+\alpha,\alpha_1,\alpha_\infty)\left(\norm{\mathcal H_0(z)}^2+\sum_{i=1}^2A_\gamma^{(i)}(\alpha,\alpha_0,\alpha_1,\alpha_\infty)\norm{\mathcal H_i(z)}^2\right).
		\end{split}
	\end{equation}
	The constants $A_\gamma^{(i)}(\alpha,\alpha_0,\alpha_1,\alpha_\infty)$ that appear are given by
	\begin{align}
		A_\gamma^{(i)}(\alpha,\alpha_0,\alpha_1,\alpha_\infty)\coloneqq \frac{\prod_{j=1}^3l(A_j)l(B_i-A_j)}{l(B_1)l(B_2)}\frac{l(1+B_1+B_2-2B_i)}{l(B_i-1)},
	\end{align}
	while the functions $\mathcal H_i(z)$ are linearly independent hypergeometric functions of the third order. The proof of this statement is based on the fact that, thanks to~\cite[Theorem 1.3]{Toda_OPEWV}, such four-point correlation functions are solutions of a differential equation in the $z$ variable of the third order, which we refer to as a BPZ-type (for Belavin-Polyakov-Zamolodchikov) differential equation.
	
	\subsubsection{Four-point correlation functions from Operator Product Expansions}
	In Section~\ref{sec:OPE}, we show that an alternative expansion for such four-point correlation functions can be derived based on the probabilistic framework. Namely by studying the behavior of the probabilistically defined expression $\mathcal H(z)$ around $z=0$ (which somehow corresponds to the so-called \emph{Operator Product Expansions} in the physics literature), we show along the proof of Theorem~\ref{thm:OPE} that 
	\begin{equation}\label{eq:OPE_intro}
		\begin{split}
			\mathcal H(z)=\sum_{i=0}^2B_\gamma^{(i)}(\alpha_0,\chi)C_\gamma(\alpha_0-\chi h_{i+1},\alpha_1,\alpha_\infty)\norm{\mathcal H_{i}(z)}^2
		\end{split}
	\end{equation}
	as soon as it makes sense, where the coefficients $B_\gamma^{(i)}(\alpha_0,\chi)$ admit the explicit expression
	\begin{equation}
		B_\gamma^{(i)}(\alpha_0,\chi)=\prod_{j=1}^{i}\left(\pi\mu l\left(\frac{\gamma^2}{2}\right)\right)^{\frac{\chi}{\gamma}}\left(\frac{\chi^2}{2}\right)^2\frac{ l(\frac\chi2\ps{\alpha_0-Q,h_{j}-h_{i+1}})}{l(1+\frac{\chi^2}2+\frac\chi2\ps{\alpha_0-Q,h_j-h_{i+1}})}\cdot
	\end{equation}
	The three-point correlation functions that appear in these expansions may not make sense probabilistically speaking but are rather defined thanks to Theorem~\ref{thm:refl} that we show along the proof of Theorem~\ref{thm:OPE} (and to which we refer for additional precisions). 
	Showing that the above holds is purely probabilistic and strongly relies on the method introduced in the first document~\cite{Toda_correl1} of this two-part series, based on a general path decomposition for diffusion processes. Proving that such an equation does indeed hold is more demanding than its counterpart statement for Liouville theory~\cite{KRV_DOZZ}, not only because we need to rely on the generalized path decomposition unveiled in~\cite{Toda_correl1} but also due to the additional technicality of the computations involved.
	
	\subsubsection{Conclusion} 
	Eventually in Section~\ref{sec:toda_end} we combine the two above expansions to show that our main statement, Theorem~\ref{thm:main_result}, does indeed hold. Namely by linear independence of the $\mathcal H_{i}$ we infer the \emph{shift equations} for the three-point correlation functions:
	\begin{equation}
		\begin{split}
			&\frac{C_\gamma(\alpha_0-\chi h_{i+1},\alpha_1,\alpha_\infty)}{C_\gamma(\alpha_0-\chi h_1,\alpha_1,\alpha_\infty)}=\frac{A_\gamma^{(i)}(-\chi h_1,\alpha_0,\alpha_1,\alpha_\infty)}{B^{(i)}(\alpha_0)}
		\end{split}
	\end{equation}
	which are valid for $i=1,2$ and $\chi\in\{\gamma,\frac2\gamma\}$, and as soon as the four-point correlation functions 
	\[
	\ps{V_{-\chi h_1}(z)V_{\alpha_0}(0)V_{\alpha_1}(1)V_{\alpha_\infty}(\infty)}
	\]
	make sense. The assumption that such correlation functions are probabilistically well-defined when $\chi=\frac2\gamma$ stems from the fact that Equation~\eqref{eq:main_result} holds true only under the assumption that $\gamma\in[1,\sqrt2)$. 
	
	Eventually we will see that this set of shift equations characterizes, up to a multiplicative constant, the expression proposed in~\cite{FaLi1} for Toda three-point correlation functions. One can evaluate the value of this multiplicative constant, which allows us to conclude for the proof of Theorem~\ref{thm:main_result}. 
	
	\subsection{Some outlooks}\label{subsec:perspectives}
	\subsubsection{Integrability of Toda CFTs}
	The key arguments used to derive formula~\eqref{eq:fali} actually extend to the case where the underlying Lie algebra is $\mathfrak{g}=\mathfrak{sl}_n$ for $n\geq 2$. This approach allowed Fateev-Litvinov to propose a general expression for a family of three-point correlation functions in this setting too~\cite{FaLi0, FaLi1}. At the time being, the main obstruction preventing a mathematically rigorous derivation of such formulas for the general case of $\mathfrak{sl}_n$ is merely of technical order and would rely on a derivation of the Ward identities within this framework, which is feasible but requires rather tedious algebraic computations. On a similar perspective, extending the range of validity for $\gamma$ in Theorem~\ref{thm:main_result} would involve a more precise description of the asymptotic expansion of Toda correlation functions, which is mostly prevented by technical issues too.
	
	Likewise, the probabilistic method considered in this document to prove Theorem~\ref{thm:main_result} may be applied to the case of boundary Toda CFTs, which would hopefully allow to obtain formulas analogous to the Fyodorov-Bouchaud formula~\cite{Rem17} to describe the law of correlated GMC measures on the circle.
	
	\subsubsection{Study of the Toda Hamiltonian and conformal bootstrap}
	Related to the program of integrability of the $\mathfrak{sl}_3$ Toda CFT is the implementation of the \emph{conformal bootstrap} procedure within this setting. This would require a refined study of the Toda Hamiltonian that, building on the study of the Liouville Hamiltonian conducted in~\cite{GKRV}, would in turn involve conjugation of scattering theory in symmetric spaces~\cite{MV_sym} together with techniques developed in~\cite{BaD, MV_prod} for tensor products of self-adjoint operators. 
	We would like to stress that from the viewpoint of the physics literature, it does not seem clear at this stage under which form the conformal bootstrap should hold.
	
	\subsection*{Acknowledgments} The author would like to thank C. Guillarmou and V. Vargas for fruitful discussions as well as for comments made on a prior version of this manuscript, and is grateful to the anonymous referee for their very careful reading and useful comments. The author acknowledges that this project has received funding from the European Research Council (ERC) under the European Union’s Horizon 2020 research and innovation programme, grant agreement No 725967.
	
	\section{Correlation functions in the $\mathfrak{sl}_3$ Toda conformal field theory}\label{sec:toda}
	In this first section we provide the necessary background to make sense of Toda correlation functions. Namely we will first recall the probabilistic setting in which Toda CFTs have been introduced in~\cite{Toda_construction} and then show that the definition proposed there can actually be analytically extended beyond the bounds imposed by~\cite[Theorem 3.1]{Toda_construction}.
	
	\subsection{A probabilistic definition of Toda correlation functions}
	\subsubsection{Conformal geometry on the Riemann sphere and reminders on the $\mathfrak{sl}_3$ Lie algebra}\label{subsec:lie}
	The underlying surface for the correlation functions we consider is the two-dimensional sphere $\mathbb{S}^2$ which we identify with the Riemann sphere $\C\cup\lbrace\infty\rbrace$. It is equipped with the Riemannian metric
	\[
	g(z)\coloneqq\norm{z}^{-4}_+|dz|^2
	\]
	where $\norm{z}_+\coloneqq\max(\norm{z},1)$. This metric is the most natural both from the probabilistic viewpoint (via the radial/angular decomposition of the GFF) and in the context of reflection (or Osterwalder-Schrader) positivity. The Green kernel associated to this metric is given by
	\begin{equation}\label{Green_kernel}
		G(x,y)\coloneqq\ln\frac{1}{\norm{x-y}}+\ln\norm{x}_++\ln\norm{y}_+.
	\end{equation}  
	
	Along the proof of Theorem~\ref{thm:main_result}, we will also need additional information on the $\mathfrak{sl}_3$ Lie algebra. We defer for instance to~\cite{humphreys_reflection} for additional details. The Toda field defined via the path integral~\eqref{eq:path_integral} is then a random map $\varphi:\C\cup\lbrace\infty\rbrace\to \R^2$, where $\R^2$ comes equipped with a basis $(e_1,e_2)$ of so-called \emph{simple roots} such that \begin{equation}
		\left(\ps{e_i,e_j}\right)_{i,j}\coloneqq A=\begin{pmatrix}
			2 & -1\\
			-1 & 2
		\end{pmatrix}.
	\end{equation}
	The dual basis is $(\omega_1,\omega_2)$, defined by
	\begin{equation}
		\omega_1=\frac{2e_1+e_2}3\quad\text{and}\quad\omega_2=\frac{e_1+2e_2}{3}\cdot
	\end{equation}
	For $\mathfrak g=\mathfrak{sl}_3$, the Weyl vector takes the form
	\begin{equation}
		\rho= \omega_1+\omega_2=e_1+e_2
	\end{equation}
	while the fundamental weights in the first fundamental representation $\pi_1$ of $\mathfrak{sl}_3$ with the highest weight $\omega_1$ are vectors in $\R^2$ defined by setting:
	\begin{equation}\label{eq:definition_hi}
		h_1\coloneqq \frac{2e_1+e_2}{3},\quad h_2\coloneqq \frac{-e_1+e_2}{3}, \quad h_3\coloneqq  -\frac{e_1+2e_2}{3}\cdot
	\end{equation}
	
	In this context, there is a natural group that acts on $\R^2$: the Weyl group $W$ which is the group generated by $s_1,s_2$ the reflections with respect to the hyperplanes orthogonal to the simple roots. Put differently $W$ is the six-element group generated by the
	\[
	s_i:\alpha\mapsto \alpha-\ps{\alpha,e_i}e_i, \quad\text{for }i=1,2.
	\]
	The action of $W$ divides $\R^2$ into six components called \emph{Weyl chambers} on which it acts freely and transitively. Among these Weyl chambers we denote by 
	\begin{equation}
		\mathcal{C}\coloneqq \left\{x\in\V, \ps{x,e_i}>0\quad\text{for all }1\leq i\leq r\right\}
	\end{equation}
	the \emph{fundamental Weyl chamber} and set $\mathcal{C}_-\coloneqq -\mathcal{C}$ (which is also a Weyl chamber). The boundary of the Weyl chamber $\mathcal{C}_-$ is made of two components $\partial\mathcal{C}_i\coloneqq\partial\mathcal{C}_-\cap\{\ps{x,e_i}=0\}$. 
	
	We will also need to consider the Weyl group which is not centered at the origin of $\R^2$. Namely for $s\in W$ we set
	\[
	\hat s(\alpha)\coloneqq Q+s(\alpha-Q)
	\]
	where $Q$ is some vector in $\R^2$ (which will be fixed later on, see Equation~\eqref{eq:definition_Q} below). Finally by analogy with a group of permutations (note that $W\simeq \mathfrak{S}_3$) we denote $\epsilon(s)\coloneqq \det(s)$. 
	
	\subsubsection{Gaussian Free Fields and Gaussian Multiplicative Chaos}\label{subsec:GFF_GMC}
	The Toda field formally defined via the path integral~\eqref{eq:path_integral} can actually be interpreted using Gaussian Free Fields (GFFs hereafter). In the present context, these are therefore random distributions defined over $\C$ and with values in $\R^2$. This Gaussian random distribution is characterized by its covariance kernel: 
	\begin{equation}
		\E[\ps{u,\X(x)}\ps{v,\X(y)}]=\ps{u,v}  G(x,y),\quad x\neq y\in\C
	\end{equation}
	for any vectors $u,v\in\R^2$, with $G$ as defined in Equation~\eqref{Green_kernel} above. We defer to \emph{e.g.} in~\cite{dubedat, She07} for additional details on GFFs. This random generalized function does not make sense pointwise but rather belongs to the distributional space $\mathrm{H}^{-1}(\C\to\R^2,g)$.
	
	In order to provide a rigorous meaning to the path integral defining Toda CFTs~\eqref{eq:path_integral}, we need to make sense of the exponential of the GFF $\X$. However the latter being only defined as a random distribution we need to rely on a regularization procedure to define it properly. 
	For this purpose we first regularize the GFF at mesh $\eps$ by setting
	\begin{equation}\label{regularization}
		\X_\eps\coloneqq \X*\eta_\eps=\int_\C X(\cdot-z)\eta_\eps(z)d^2z
	\end{equation}
	with $\eta_\eps\coloneqq \frac1{\eps^2}\eta(\frac{\cdot}{\eps})$ a smooth, compactly supported mollifier. Then the field $\X_\eps$ defined in such a way is a smooth function. As a consequence its exponential is well-defined so that we can define a random measure over $\C$ by setting
	\begin{equation}
		M_\eps^{\gamma e_i}(d^2x)\coloneqq \:  e^{  \langle\gamma e_i,  \X_\eps(x) \rangle-\frac12\expect{\langle\gamma e_i, \X_\eps(x) \rangle^2}} \norm{x}^{-4}_+\norm{d^2x}.
	\end{equation}
	The limit as $\eps\rightarrow0$ of this random measure is well-defined and is called a \emph{Gaussian Multiplicative Chaos measure} (GMC hereafter). More precisely under the assumption that $\gamma<\sqrt 2$ and for $1\leq i\leq r$, the limit 
	\begin{equation}
		M^{\gamma e_i}(d^2x)\coloneqq \underset{\eps \to 0}{\lim} M_\eps^{\gamma e_i}(d^2x)
	\end{equation} 
	holds in probability within the space of Radon measures equipped with the weak topology and defines a non-trivial random measure~\cite{Ber,Sha}.
	
	\subsubsection{Correlation functions and the path integral}
	Toda CFTs depend on a coupling constant $\gamma\in(0,\sqrt 2)$\footnote{The range of values for the coupling constant is to ensure that the probabilistic construction does make sense. Note that it differs from the usual range of values due to the fact that the longest roots have length $\sqrt 2$.} as well as cosmological constants $\mu_1,\mu_2>0$. The background charge is the vector in $\R^2$ defined by
	\begin{equation}\label{eq:definition_Q}
		Q\coloneqq  q \rho,\quad\text{with } q= \gamma+\frac{2}{\gamma}\cdot
	\end{equation}
	In order to give a probabilistic meaning to the path integral~\eqref{eq:path_integral}, we define the Toda field as the random distribution
	\begin{equation}
		\varphi\coloneqq \X+\frac Q2\ln g+\bm c
	\end{equation}
	where $\bm{c}$ is uniformly distributed according to the Lebesgue measure on $\R^2$ and $\X$ is the GFF from Subsection~\ref{subsec:GFF_GMC}. This allows to derive the probabilistic interpretation of the Toda path integral~\eqref{eq:path_integral} by setting for $F$ positive and continuous on $\mathrm{H}^{-1}(\C\to\R^2,g)$ (see~\cite[Equation (2.32)]{Toda_construction}):
	\begin{equation}\label{eq:proba_TPI}
		\ps{F}\coloneqq \int_{\R^2}e^{-2\ps{Q,\bm c}}\E\left[F\left( \X+\frac Q2\ln g+\bm c \right)e^{- \sum_{i=1}^{2} \mu_i e^{\gamma  \langle e_i, \bm c\rangle }M^{ \gamma e_i}(\C) }\right]d\bm c
	\end{equation}
	where the terms $M^{ \gamma e_i}(\C)$ that appear in the exponential correspond to the total mass of the GMC measures introduced in Subsection~\ref{subsec:GFF_GMC}:
	\[
	M^{ \gamma e_i}(\C)=\int_{\C}M^{\gamma e_i}(d^2x).
	\]
	The correlation functions of \emph{Vertex Operators} formally correspond to taking $F=e^{\ps{\alpha,\X(z)}}$ in Equation~\eqref{eq:proba_TPI}, where $z$ is an insertion point in $\C$ and $\alpha$ is a weight in $\R^2$. They can be rigorously defined using a regularization procedure by setting 
	\[
	V_{\alpha,\eps}(z)\coloneqq \norm{z}_+^{-4\Delta_\alpha}e^{\ps{\alpha,\bm c}}e^{\ps{\alpha,\X_\eps(z)}-\frac{\expect{\ps{\alpha,\X_\eps(z)}^2}}2}
	\]
	with $\Delta_\alpha=\ps{\frac\alpha2,Q-\frac\alpha2}$. The exponential term can be interpreted thanks to the Girsanov (or Cameron-Martin) theorem:
	\begin{thmcite}\label{thm:girsanov}
		Let $D$ be a subdomain of $\C$ and \[
		(\bm Y(x))_{x\in D}\coloneqq (Y_1(x),\cdots,Y_{n}(x))_{x\in D}\]
		be a family of smooth centered Gaussian field. Also consider $Z$ any Gaussian variable belonging to the $L^2$ closure of the subspace spanned by $(\bm Y(x))_{x\in D}$. Then, for any bounded functional $F$ over the space of continuous functions, one has that
		\[
		\expect{e^{Z-\frac{\expect{Z^2}}{2}}F(\bm Y(x))_{x\in D}}=\expect{F\left(\bm Y(x)+\expect{Z\bm Y(x)}\right)_{x\in D}}.
		\]
	\end{thmcite}
	Based on Theorem~\ref{thm:girsanov} and Equation~\eqref{eq:proba_TPI} above, it was shown in~\cite{Toda_construction} that the correlation functions then take the form, for $z_1,\cdots,z_N$ in $\C$ all distinct,
	\begin{equation}\label{eq:correl_1}
		\Big\langle V_{\alpha_1}(z_1) \cdots V_{\alpha_N}(z_N) \Big\rangle= \prod_{ j < k} \norm{z_j-z_k}^{-\langle\alpha_j ,\alpha_k \rangle }  \int_{\R^2}e^{\ps{\bm s,\bm c}}\expect{e^{-\sum_{i=1}^2\mu_ie^{\ps{\gamma e_i,\bm c}}I^i(\C)}}d\bm c
	\end{equation}
	where $\bm s\coloneqq \sum_{k=1}^N\alpha_k-2Q$ and $I^i(\C)$ is a random variable defined using the GMC:
	\begin{equation}\label{eq:def_I}
		I^i(d^2x_i)\coloneqq \prod_{k=1}^N\left(\frac{\norm{x_i}_+}{\norm{z_k-x_i}}\right)^{\gamma\ps{\alpha_k,e_i}}M^{\gamma e_i}(d^2x_i),\quad I^i(\C)=\int_\C I^i(d^2x_i).
	\end{equation}
	The whole integral over the variable $\bm c\in\R^2$ was shown in~\cite{Toda_construction} to be meaningful as soon as the weights $\alpha$ satisfy certain conditions, referred to as \emph{Seiberg bounds}~\cite{Sei90}:
	\begin{equation}\label{eq:seiberg}
		\text{For }i=1,2\quad \ps{\alpha-Q,e_i}<0\text{ and }\ps{\bm s,\omega_i}>0.
	\end{equation}
	The first of these conditions corresponds to ensuring that the GMC integrals $I^i(\C)$ are well-defined random variables, which is the case only under the assumption that $\ps{\alpha-Q,e_i}<0$, while the other condition entering the definition of the Seiberg bounds merely amounts to saying that the integrand in Equation~\eqref{eq:correl_1} is integrable when $\bm c$ diverges in $\R^2$. In the case where $\ps{\bm s,\omega_i}<0$ for $i=1,2$, we can still define them by setting
	\begin{equation}\label{eq:correl_1bis}
		\Big\langle V_{\alpha_1}(z_1) \cdots V_{\alpha_N}(z_N) \Big\rangle= \prod_{ j < k} \norm{z_j-z_k}^{-\langle\alpha_j ,\alpha_k \rangle }  \prod_{i=1}^2\left(\frac{\Gamma\left(\frac{\ps{\bm s,\omega_i}}\gamma\right)\mu_i^{-\frac{\ps{\bm s,\omega_i}}\gamma}}{\gamma}\right)\expect{\prod_{i=1}^2 I^i(\C)^{-\frac{\ps{\bm s,\omega_i}}\gamma}}
	\end{equation}
	and in that case the assumption under which the correlation functions are well-defined is that $\ps{\bm s,\omega_i}>-\frac{2}{\gamma}\vee\max\limits_{1\leq k\leq N}\ps{\alpha_k-Q,e_i}$, which corresponds to demanding that the random variable $I^i$ has a finite moment of order $-\frac1\gamma\ps{\bm s,\omega_i}$ (see~\cite{Toda_construction} for more details).
	
	\subsection{Toda correlation functions beyond the Seiberg bounds}
	However in order to make sense of certain correlation functions we will need to go beyond these bounds and extend the definition of the correlation functions for $\bm s$ no longer satisfying the above assumptions. To do so we describe a probabilistic representation of Toda correlation functions in the case where the above assumptions on $\bm s$ are no longer satisfied, and show that by doing so the map thus defined is indeed analytic in the weights $\alpha_1,\cdots,\alpha_N$.
	
	\subsubsection{Asymptotic expansion and reflection coefficients}
	As we will see, a first step toward this extension is the description of the asymptotic expansion of the integrand in Equation~\eqref{eq:correl_1} when $\bm c$ diverges in $\R^2$. It was shown in~\cite{Toda_correl1} that Toda reflection coefficients naturally show up in this expansion. These reflection coefficients are real numbers defined by setting (in the $\mathfrak{sl}_3$ case considered in the present article) for $\alpha\in\R^2$ and $s$ an element of the Weyl group $W$
	\begin{equation}\label{eq:refl}
		\begin{split}
			R_s(\alpha)&\coloneqq\epsilon(s)\frac{A\left(s(\alpha-Q)\right)}{A(\alpha-Q)},\quad\text{where}\\
			A(\alpha)&=\prod_{i=1}^2\left(\mu_i\pi l\left(\frac{\gamma^2\ps{e_i,e_i}}{4}\right)\right)^{\frac{\ps{\alpha,\omega_i}}\gamma}\prod_{e\in\Phi^+}\Gamma\left(1-\frac{\gamma}2\ps{\alpha,e}\right)\Gamma\left(1-\frac{1}\gamma\ps{\alpha,e}\right).
		\end{split}
	\end{equation}
	Here the product runs over the set of positive roots $\Phi^+=\{e_1,e_2,\rho\}$. These reflection coefficients satisfy the remarkable property that $R_{s\tau}(\alpha)=R_s(\tau\alpha)R_\tau(\alpha)$ for $s,\tau\in W$, property crucial in the derivation of their probabilistic representation as described in~\cite{Toda_correl1}. 
	
	Now given any map $w:\{1,\cdots,N\}\to W$ and $\bm\alpha\in(\C^2)^N$let us introduce:
	\[
	\bm s(w)\coloneqq \sum_{k=1}^N \hat w(k)\alpha_k-2Q
	\]
	where recall that $\hat w\alpha=Q+w(\alpha-Q)$.
	We will prove in Subsection~\ref{subsec:asym} below that the following asymptotic holds for the correlation functions: 
	\begin{proposition}\label{prop:asymptot_many}
		For $N$ a positive integer, denote by $\mathcal{A}_N$ the subset of $(Q+\mathcal{C}_-)^N$ defined by 
		\begin{equation}
			\begin{split}
				\mathcal{A}_N\coloneqq \Big\{(\alpha_1,\cdots,\alpha_N)&\in\left(Q+\mathcal{C}_-\right)^N\text{ s.t. for }i=1,2,\quad \ps{\sum_{k=1}^N\alpha_k-2Q,\omega_i}>-\gamma\\
				&\text{and for any } 1\leq k\leq N,\quad \min\limits_{i=1,2}\ps{\alpha_k-Q,e_i}<-\gamma\Big\}.
			\end{split}
		\end{equation}
		Then for any $z_1,\cdots,z_N\in\C$ distinct there exists a positive $\xi$ such that the map
		\begin{equation}
			e^{\xi\norm{\bm c}}e^{-2\ps{Q,\bm c}}\expect{\prod_{k=1}^NV_{\alpha_k}(z_k)e^{-\sum_{i=1}^2\mu_ie^{\ps{\gamma e_i,\bm c}}M^{\gamma e_i}(\C)}-\mathcal{R}_{\bm\alpha}(\bm c)}
		\end{equation}
		is uniformly bounded in $\bm c$, where the remainder term is defined by setting:
		\begin{equation}
			\begin{split}
				\mathcal{R}_{\bm\alpha}(\bm c)\coloneqq  \mathcal{R}^1_{\bm\alpha}(\bm c)+\mathcal{R}^2_{\bm\alpha}(\bm c)&+\mathcal{R}^{1,2}_{\bm\alpha}(\bm c),\quad\text{with}\\
				\mathcal{R}^1_{\bm\alpha}(\bm c)\coloneqq \sum_{w:\{1,\cdots,N\}\to\{Id,s_2\}}&\mathds{1}_{\ps{\bm s(w),\omega_1}<0}\prod_{k=1}^NR_{w(k)}(\alpha_k)V_{\hat w(k)\alpha_k}(z_k)e^{-\mu_1e^{\gamma\ps{\bm c,e_1}}M^{\gamma e_1}(\C)}\\
				-\sum_{w:\{1,\cdots,N\}\to\{Id,s_1,s_2,s_2s_1\}}&\mathds{1}_{\ps{\bm s(w),\omega_1}<0}\prod_{k=1}^NR_{w(k)}(\alpha_k)V_{\hat w(k)\alpha_k}(z_k),\\
				\mathcal{R}^{1,2}_{\bm\alpha}(\bm c)\coloneqq  \sum_{w:\{1,\cdots,N\}\to W}&\mathds{1}_{\max_{j=1,2}\ps{\bm s(w),\omega_j}<0}\prod_{k=1}^NR_{w(k)}(\alpha_k)V_{\hat w(k)\alpha_k}(z_k).
			\end{split}
		\end{equation}
	\end{proposition}
	
	\subsubsection{Analytic continuation of the correlation functions}
	Based on Proposition~\ref{prop:asymptot_many} we are able to extend the definition of the correlation functions by relaxing the assumption that $\ps{\bm s,\omega_i}>\max\limits_{1\leq k\leq N}\ps{\alpha_k-Q,e_i}$. Our main statement in this perspective is the following:
	\begin{theorem}\label{thm:analycity}
		For $\bm\alpha=(\alpha_1,\cdots,\alpha_N)\in\mathcal{A}_N$, and $z_1,\cdots,z_N$ in $\C$ all distinct, set 
		\begin{equation}\label{eq:correl_2}
			\begin{split}
				\Big\langle V_{\alpha_1}(z_1) \cdots V_{\alpha_N}(z_N) \Big\rangle&= \int_{\R^2}e^{-2\ps{Q,\bm c}}\expect{\prod_{k=1}^NV_{\alpha_k}(z_k)e^{-\sum_{i=1}^2\mu_ie^{\gamma \ps{\bm c,e_i}}M^{\gamma e_i}(\C)}-\mathcal{R}_{\bm\alpha}(\bm c)}d\bm c.
			\end{split}
		\end{equation}
		Then the map 
		\[
		\bm\alpha\mapsto \Big\langle V_{\alpha_1}(z_1) \cdots V_{\alpha_N}(z_N) \Big\rangle
		\]
		is analytic in a complex neighborhood of $\mathcal{A}_N$ seen as a subset of $(\C^2)^N$\footnote{That is to say, for any $u\in\R^2$ and $1\leq k\leq N$, the map $\ps{u,\alpha_k}\mapsto \Big\langle V_{\alpha_1}(z_1) \cdots V_{\alpha_N}(z_N) \Big\rangle$ is analytic.}. 
		
		Moreover its poles over this domain are given by the set 
		\begin{equation}
			\begin{split}
				\mathcal{P}_{N}\coloneqq& \{\bm\alpha\in\mathcal{A}_N,\quad\text{$\ps{\alpha_k-Q,e}\in-\gamma\N^*\cup-\frac2\gamma\N^*$ for some $1\leq k\leq N$ and $e\in\Phi^+$}\}\\
				\bigcup&\{\bm\alpha\in\mathcal{A}_N,\quad\text{$\ps{\bm s(w),\omega_i}=0$ for some $w:\{1,\cdots,N\}\to W$ and $i=1$, or $2$}\}.
			\end{split}
		\end{equation}
	\end{theorem}
	This statement allows to provide an analytic continuation of the correlation functions beyond the Seiberg bounds and will be of crucial importance in the derivation of Theorem~\ref{thm:main_result}.
	
	Following the same reasoning as the one developed in the proof of~\cite[Theorem 3.1]{Toda_construction}, the defining expression for the correlation functions (Equation~\eqref{eq:correl_2}) admits an alternative representation based on the property that
	\begin{equation}\label{eq:correl_3}
		\begin{split}
			&e^{-2\ps{Q,\bm c}}\expect{\prod_{k=1}^NV_{\alpha_k}(z_k)\prod_{i=1}^2e^{-\mu_ie^{\gamma \ps{\bm c,e_i}}M^{\gamma e_i}(\C)}}=\\
			&\prod_{j<k}\norm{z_j-z_k}^{-\ps{\alpha_j,\alpha_k}}e^{\ps{\bm s,\bm c}}\expect{\prod_{i=1}^2e^{-\mu_ie^{\gamma \ps{\bm c,e_i}}I^i(\C)}}.
		\end{split}
	\end{equation}

	In the sequel we will consider three and four-point correlation functions with an insertion point at $z=\infty$. They are defined as the limit
	\[
	\Big\langle V_{\alpha_1}(z_1) \cdots V_{\alpha_N}(z_N)V_{\alpha_\infty}(\infty) \Big\rangle\coloneqq \lim\limits_{z'\to\infty}\norm{z'}^{4\Delta_{\infty}}\Big\langle V_{\alpha_1}(z_1) \cdots V_{\alpha_N}(z_N)V_{\alpha_\infty}(z') \Big\rangle,
	\]
	which translates as 
	\begin{equation}\label{eq:def_correl_infty}
		\begin{split}
			&\expect{\prod_{k=1}^NV_{\alpha_k}(z_k)V_{\alpha_\infty}(\infty)\prod_{i=1}^2e^{-\mu_ie^{\gamma \ps{\bm c,e_i}}M^{\gamma e_i}(\C)}}\coloneqq \\
			&\prod_{1\leq j<k\leq N}\norm{z_j-z_k}^{-\ps{\alpha_j,\alpha_k}}e^{\ps{\sum_{k=1}^N\alpha_k+\alpha_\infty,\bm c}}\expect{\prod_{i=1}^2e^{-\mu_ie^{\gamma \ps{\bm c,e_i}}\tilde I_i(\C)}}\quad\text{with}\\
			& \tilde I_i(d^2x_i)\coloneqq \prod_{k=1}^N\left(\frac{\norm{x_i}_+}{\norm{z_k-x_i}}\right)^{\gamma\ps{\alpha_k,e_i}}\norm{x_i}_+^{\gamma\ps{\alpha_\infty, e_i}}M^{\gamma e_i}(d^2x_i).
		\end{split}
	\end{equation}
	The same applies for terms containing expressions of the form $R_s(\alpha_\infty)V_{\hat s\alpha_\infty}(\infty)$ by using the fact that the conformal weight $\Delta_{\alpha}$ is invariant under the action of the Weyl group:
	\[
	\Delta_{\hat s\alpha}=\Delta_{\alpha}\quad\text{for all }s\in W.
	\]
	The Toda three-point correlation functions that enter Theorem~\ref{thm:main_result} are defined by setting for $(\alpha_0,\alpha_1,\alpha_\infty)\in\mathcal{A}_3$ \begin{equation}
		C_\gamma(\alpha_0,\alpha_1,\alpha_\infty)\coloneqq \Big\langle V_{\alpha_0}(0)V_{\alpha_1}(1)V_{\alpha_\infty}(\infty) \Big\rangle
	\end{equation}
	with the right-hand side as above. 
	
	
	\subsection{Additional properties of Gaussian Free Fields and Gaussian Multiplicative Chaos}
	Before actually moving on to the proof of Theorem~\ref{thm:analycity}, we provide additional properties of the GFFs and GMC measures considered, and that will be crucial in the rest of the document.
	
	\subsubsection{Radial-angular decomposition}
	The GFFs introduced in Subsection~\ref{subsec:GFF_GMC} have strong interplays with Brownian motions, since the latter naturally arise when writing GFFs in polar coordinates. Namely, let us introduce the process $B_t\coloneqq \X_t(z)$, where for any $z\in\C$ and $t\in\R$
	\begin{equation}
		\X_t(z)\coloneqq \frac{1}{2\pi}\int_0^{2\pi}\X(z+e^{-t+i\theta})d\theta.
	\end{equation}
	Then we can write the GFF $\X$ under the form
	\begin{equation}\label{eq:rad_ang_dec}
		\X(z+e^{-t+i\theta})= B_t+Y(t,\theta)
	\end{equation}
	where $Y(t,\theta)\coloneqq \X(z+e^{-t+i\theta})- B_t$.
	
	One easily checks using the covariance kernel of $\X$ that the process $(B_t)_{t\geq0}$ is a planar Brownian motion started from the origin. Besides $Y$ is a Gaussian field with covariance kernel given by 
	\begin{equation}\label{eq:cov_Y}
		\E[\ps{u,Y(t,\theta)}\ps{v,Y(t',\theta')}]=\ps{u,v}  \ln\frac{e^{-t}\vee e^{-t'}}{\norm{e^{-t+i\theta}-e^{-t'+i\theta'}}}
	\end{equation}
	for any $u,v\in\V$, with $B$ and $Y$ independent.
	
	Thanks to Equation~\eqref{eq:rad_ang_dec} above the GMC defined from the GFF $\X$ admits the alternative representation
	\begin{equation}\label{equ:radial_angular}
		M^{\gamma e_i}(d^2x)=e^{\gamma\ps{B_t-Qt,e_i}} M_Y^{\gamma e_i}(dt,d\theta),
	\end{equation}
	where $M_Y$ is the GMC measure defined from $Y$. Put differently, for any bounded map $F$ and $\mathcal{O}$ a measurable subset of $\C$ 
	\[
	\expect{F\left(\int_{\mathcal{O}}M^{\gamma e_i}(d^2x)\right)}=\expect{F\left(\int_{\R}e^{\gamma\ps{B_t-Qt,e_i}}\int_0^{2\pi}\mathds1_{e^{-t+i\theta}\in\mathcal{O}}M^{\gamma e_i}_Y(dt,d\theta)\right)}.
	\]
	
	\subsubsection{Path decompositions for Brownian motions}\label{subsec:brown_cond_neg}
	When considering the asymptotic of the correlation functions, we will need to understand the asymptotic properties of (correlated) GMC measures. In this regime we will need to study the behavior of the process $B_t$ introduced above, and more precisely we will need to know the law of this process under a certain conditioning. For this purpose, we introduce for negative $\nu$ the diffusion process $\mathcal{B}^\nu$, defined by joining
	\begin{itemize}
		\item a one-dimensional Brownian motion with positive drift $-\nu$, started from $x<0$ and run upon hitting $0$;
		\item an independent one-dimensional Brownian motion with negative drift $\nu$ and conditioned to stay negative, that is to say the diffusion process with generator
		\[
		\frac12\frac{d^2}{dx^2}+\nu\text{coth}(\nu x)\frac{d}{dx}\cdot
		\]
	\end{itemize}
	Then it follows from a celebrated result of Williams~\cite{williams} that the law of a one-dimensional Brownian motion $B^\nu$ with drift $\nu$ and started from the origin can be realized by sampling:
	\begin{itemize}
		\item an exponential random variable $\M$ with parameter $-2\nu$ and that represents $\sup_{t\geq 0} B^\nu_t$;
		\item the process $\M+\mathcal{B}^\nu$ described above, where $\mathcal{B}^\nu$ is started from $-\M$.
	\end{itemize}
	In particular this allows to give a meaning to the process whose law is that of $B^\nu$ conditionally on the value of its maximum.
	
	In the planar case, a similar picture also holds as proved in the first document of this two-part series~\cite{Toda_correl1}. Namely, we can define a process $\B^\nu$ with $\nu\in\mathcal{C}_-$, started from $x\in\mathcal{C}_-$, by joining:
	\begin{itemize}
		\item a diffusion process $\X^1$ with infinitesimal generator 
		\[
		\frac12\Delta+\bigtriangledown \log \partial_{12}h\cdot\bigtriangledown,
		\] run until hitting $\partial\mathcal{C}_-$, say at $z_1\in \partial\mathcal{C}_1$. 
		\item an independent diffusion process $\X^2$ started from $z_1$ and with infinitesimal generator 
		\[
		\frac12\Delta+\bigtriangledown \log \partial_{2}h\cdot\bigtriangledown,
		\] upon hitting $\partial\mathcal{C}_2$ at $z_2$.
		\item an independent process $\X^3$ started from $z_2$, whose law is that of $B^\nu$ conditioned on staying inside $\mathcal{C}_-$, that is to say the diffusion with generator
		\[
		\frac12\Delta+\bigtriangledown \log h\cdot\bigtriangledown.
		\]
	\end{itemize}
	Here we have denoted
	\begin{equation}\label{equ:max_drift}
		h(x)\coloneqq\sum_{s\in W}\epsilon(s)e^{\ps{s\nu-\nu,x}}\quad\text{and}\quad\partial_i\text{ is a shorthand for }\partial_{\ps{x,e_i}}.
	\end{equation}
	With this process at hand, a planar Brownian motion $B^\nu$ with drift $\nu\in\mathcal{C}_-$ and started from the origin can be realized by sampling:
	\begin{itemize}
		\item a random variable $\M$ defined by 
		\[
		\P\left(\ps{\M,e_i}\leq \bm m_i\quad\forall1\leq i\leq r\right)=h(-\bm m)\mathds1_{\bm m\in\mathcal{C}},
		\]
		and which is such that $\ps{\M,e_i}=\sup_{t\geq 0} \ps{B^\nu_t,e_i}$ for $i=1,2$;
		\item the process $\M+\B^\nu$ described above, where $\B^\nu$ is started from $-\M$.
	\end{itemize}
	In particular this decomposition allows to write that for any bounded map $F:\R^2\to\R$ and $B_r(z)=B(z,e^{-r})$ an Euclidean ball centered at some point $z\in\C$ with radius $e^{-r}$
	\begin{equation}\label{eq:williams}
		\begin{split}
			&\expect{F\left(\int_{B_r(z)}\norm{x-z}^{-\gamma\ps{\alpha,e_i}}M^{\gamma e_i}(d^2x)\right)_{i=1,2}}=\\
			&\int_{\mathcal{C}} \partial_{12}h(\M)\E_{-\M}\left[F\left(e^{\gamma (\X_r(z)+\nu r+\M)}\int_{r}^{+\infty}e^{\gamma\ps{\B^\nu_t,e_i}}\int_0^{2\pi}\mathds1_{e^{-t+i\theta}\in B_r(z)}M^{\gamma e_i}_Y(dt,d\theta)\right)\right]d\M
		\end{split}	
	\end{equation}
	where under $\E_{-\M}$ the process $\B^\nu$ is started from $-\M$, and with $\nu=\alpha-Q$. For future reference let us stress that 
	\[
	\partial_{12}h(\M)=\sum_{s\in W_{1,2}}\lambda_se^{\ps{s\nu-\nu,\M}}\quad\text{where}\quad \lambda_s=\ps{s\nu-\nu,\omega_1}\ps{s\nu-\nu,\omega_2}
	\]
	and $W_{1,2}=\{s_1s_2,s_2s_1,s_1s_2s_1\}$.
	
	\subsubsection{Moments of Gaussian Multiplicative Chaos measures}
	Recall the notation introduced in Equation~\eqref{eq:def_I}:
	\begin{equation*}
		I^i(d^2x_i)= \prod_{k=1}^N\left(\frac{\norm{x_i}_+}{\norm{z_k-x_i}}\right)^{\gamma\ps{\alpha_k,e_i}}M^{\gamma e_i}(d^2x_i).
	\end{equation*}
	The statement of~\cite[Lemma 4.1]{Toda_construction} shows that the random variable $I^i(\C)$ is well defined as soon as $\ps{\alpha_k-Q,e_i}<0$ for all $1\leq k\leq N$ and $i=1,2$, and satisfies the property that
	\begin{equation}
		\expect{I^i(\C)^p}<\infty\quad\text{as soon as $p<\frac{2}{\gamma^2}\vee \max_{1\leq k\leq N}\frac{\ps{Q-\alpha_k,e_i}}\gamma\cdot$}
	\end{equation} More precisely and using the same notations as above 
	\begin{align}\label{eq:moments_I}
		&\expect{\prod_{i=1}^2\left(\int_{B_r(z)}\norm{x-z}^{-\gamma\ps{\alpha,e_i}}M^{\gamma e_i}(d^2x)\right)^{p_i}}<\infty\quad\text{for $p_i<\frac{2}{\gamma^2}\vee \frac{\ps{Q-\alpha,e_i}}\gamma$, $i=1,2$.}
	\end{align}
	In particular this defines $L^1$ random variables as soon as $\ps{Q-\alpha,e_i}>\gamma$.
	
	In order to go beyond these bounds we will use the radial-angular decomposition from Equation~\eqref{eq:rad_ang_dec}, which is such that
	\begin{align}\label{eq:moments_J}
		&\expect{\prod_{i=1}^2\left(\int_{r}^{+\infty}e^{\gamma\ps{\B^\nu_t,e_i}}\int_0^{2\pi}M^{\gamma e_i}_Y(dt,d\theta)\right)^{p_i}}<\infty\quad\text{for $p_i<\frac{2}{\gamma^2},\quad i=1,2$.}
	\end{align}
	
	\subsubsection{Fusion asymptotics}\label{subsec:fusion}
	Another key question is to investigate what happens when two insertion points collide, that is to say for instance when $z_2\to z_1$ in a correlation function of the form
	\[
	\Big\langle V_{\alpha_1}(z_1)V_{\alpha_2}(z_2) \prod_{k=3}^N V_{\alpha_k}(z_k) \Big\rangle\cdot
	\]
	We may then distinguish between several cases:
	\begin{itemize}
		\item First of all, if $\ps{\alpha_1+\alpha_2-Q,e_i}<0$ then the limiting GMC measure still makes sense and is given by
		\begin{equation*}
			I^i(d^2x_i)= \left(\frac{\norm{x_i}_+}{\norm{z_1-x_i}}\right)^{\gamma\ps{\alpha_1+\alpha_2,e_i}}\prod_{k=3}^N\left(\frac{\norm{x_i}_+}{\norm{z_k-x_i}}\right)^{\gamma\ps{\alpha_k,e_i}}M^{\gamma e_i}(d^2x_i).
		\end{equation*}
		\item Conversely if $\ps{\alpha_1+\alpha_2-Q,e_1}\geq0$ then the limiting GMC measure is not well-defined. However we can estimate the rate of divergence of the random variable by means of the so-called \emph{fusion asymptotics} (see~\cite[Lemma 3.2]{Toda_OPEWV} or the freezing estimate~\cite[Lemma 6.5]{KRV_loc}\footnote{The factor $\frac14$ which appears in the exponent $\frac{\ps{\alpha_1+\alpha_2-Q,e_1}^2}{4}$ stems from the fact that $\ps{e_1,e_1}=2$.}), which imply that as $z_2\to z_1$:
		\begin{equation}\label{eq:fusion}
			\begin{split}
				&\expect{\exp\left(-\int_{B_r(z_1)}\norm{x-z_1}^{-\gamma\ps{\alpha_1,e_1}}\norm{x-z_2}^{-\gamma\ps{\alpha_2,e_1}}M^{\gamma e_1}(d^2x)\right)}\\
				&=o\left(\norm{z_1-z_2}^{\frac{\ps{\alpha_1+\alpha_2-Q,e_1}^2}{4}-\eps}\right)\quad\text{for any positive }\eps\text{ and fixed }r.
			\end{split}
		\end{equation}
	\end{itemize}
	
	
	\subsection{Analytic continuation of Toda correlation functions: proof of Theorem~\ref{thm:analycity}}
	This subsection is dedicated to proving Theorem~\ref{thm:analycity}. As we will see, the method developed in what follows can readily be implemented within the framework of Liouville CFT, so that the proof of Theorem~\ref{thm:analycity} can be adapted in a straightforward way to prove Corollary~\ref{cor:DOZZ}. 
	
	\subsubsection{Reducing the proof} 
	The key idea behind the continuation of the Toda correlation functions is that the quantity entering that definition of the correlation functions:
	\[
	e^{-\ps{2Q,\bm c}}\expect{\prod_{k=1}^NV_{\alpha_k}(z_k)e^{-\sum_{i=1}^r\mu_ie^{\gamma\ps{\bm c,e_i}}M^{\gamma e_i}(\C)}}
	\] 
	is integrable in $\bm c$ under the assumption that $\ps{\bm s,\omega_1}>0$, but in the case where $\ps{\bm s,\omega_1}<0$ it behaves like $e^{\ps{\bm s,\omega_1}c_1}$ as $c_1\coloneqq\ps{\bm c,e_1}$ goes to $-\infty$ and as such is no longer integrable. To remedy this issue we need to add a correction term $\mathcal R_{\bm\alpha}(\bm c)$ that will be such that the quantity
	\begin{equation}\label{eq:encoreune}
		e^{-\ps{2Q,\bm c}}\expect{\prod_{k=1}^NV_{\alpha_k}(z_k)e^{-\sum_{i=1}^r\mu_ie^{\gamma\ps{\bm c,e_i}}M^{\gamma e_i}(\C)}-\mathcal{R}_{\bm\alpha}(\bm c)}
	\end{equation}
	is no longer divergent as $c_1\to-\infty$. As discussed above it is natural to expect that reflection coefficients will enter the definition of this extra term since they encode the asymptotic behavior of such GMC measures in this asymptotic.
	
	To settle the ideas, let us first consider the subset $\mathcal{A}_N^{1;0}$ of $\mathcal{A}_N$ defined by assuming that for $2\leq k\leq N$ and $i=1,2$, $\ps{\alpha_k-Q,e_i}<-\gamma$, and that the weight $\alpha_1$ is such that $\ps{\alpha_1-Q,e_1}>-\gamma$ while $\ps{\alpha_1-Q,e_2}<-\gamma$. In that case the remainder term $\mathcal R_{\bm\alpha}(\bm c)$ will feature only the reflection coefficient $R_{s_1}(\alpha_1)$, and is given by
	\begin{align*}
		\mathcal{R}_{\bm\alpha}(\bm c)&=0\quad\text{if }\ps{\bm s,\omega_1}>0,\\
		\mathcal{R}_{\bm\alpha}(\bm c)&=\prod_{k=1}^NV_{\alpha_k}(z_k)e^{-\mu_2e^{\gamma\ps{\bm c,e_2}}M^{\gamma e_2}(\C)}\quad\text{if }\ps{\alpha_1-Q,e_1}<\ps{\bm s,\omega_1}<0\quad\text{and}\\
		\mathcal{R}_{\bm\alpha}(\bm c)&=\left(\prod_{k=1}^NV_{\alpha_k}(z_k)+R_{s_1}(\alpha_1)V_{\hat s_1\alpha_1}(z_1)\prod_{k=2}^NV_{\alpha_k}(z_k)\right)e^{-\mu_2e^{\gamma\ps{\bm c,e_2}}M^{\gamma e_2}(\C)}\quad
	\end{align*}
	if $-\gamma<\ps{\bm s,\omega_1}<\ps{\alpha-Q,e_1}$. As we now explain, these terms will compensate the possible divergence at $\bm c$ diverges in such a way that Equation~\eqref{eq:encoreune} will become indeed integrable. But even more, we will see that by doing so we provide an analytic continuation of the correlation functions beyond the Seiberg bounds.
	To see why this is the case let us introduce 
	\[
	\bm{\mathrm{R}_{\alpha}}(\bm c)\coloneqq \mathds1_{\ps{\bm c,e_1}<0}\prod_{k=2}^NV_{\alpha_k}(z_k)\left(V_{\alpha_1}(z_1)+R_{s_1}(\alpha_1)V_{\hat s_1\alpha_1}(z_1)\right)e^{-\mu_2e^{\gamma\ps{\bm c,e_2}}M^{\gamma e_2}(\C)}.
	\] 
	We aim to prove that the map $F:(\C^2)^N\to\R$ defined by setting
	\begin{equation}\label{eq:an_toprove} 
		\begin{split}
			&F(\bm\alpha)\coloneqq \int_{\R^2}e^{-2\ps{Q,\bm c}}\expect{\prod_{k=1}^NV_{\alpha_k}(z_k)e^{-\sum_{i=1}^2\mu_ie^{\gamma \ps{\bm c,e_i}}M^{\gamma e_i}(\C)}-\bm{\mathrm R}_{\bm\alpha}(\bm c)}d\bm c\\
			&+\frac{1}{\ps{\bm s,\omega_1}}\int_{\R}\prod_{j<k}\norm{z_j-z_k}^{-\ps{\alpha_j,\alpha_k}}e^{\ps{\bm s,\omega_2}c_2}\expect{e^{-\mu_2e^{\gamma c_2}I^2(\C)}}d c_2\\
			&+\frac{R_{s_1}(\alpha_1)}{\ps{\bm s+\hat s_1\alpha_1-\alpha_1,\omega_1}}\int_{\R}\prod_{j<k}\norm{z_j-z_k}^{-\ps{\hat\alpha_j,\hat\alpha_k}}e^{\ps{\bm{\hat s},\omega_2}c_2}\expect{e^{-\mu_2e^{\gamma c_1}\hat I^2(\C)}}dc_2
		\end{split}
	\end{equation}
	is analytic in a complex neighborhood of $\mathcal{A}_N^{1;0}$, where the quantities denoted with a $\lq\lq$ ${}^\wedge$ " sign are defined by replacing $\alpha_1$ with $\hat s_1\alpha_1$. 
	
	To start with, in the case where $\ps{\bm s,\omega_i}>0$ for $i=1,2$ we have that 
	\begin{align*}
		&\int_{\R^2}e^{-2\ps{Q,\bm c}}\expect{\mathds1_{\ps{\bm c,e_1}<0}\prod_{k=1}^NV_{\alpha_k}(z_k)e^{-\mu_2e^{\gamma\ps{\bm c,e_2}}M^{\gamma e_2}(\C)}}d\bm c\\
		&=\int^0_{-\infty}e^{\ps{\bm s,\omega_1}c_1}dc_1\int_{\R}\prod_{j<k}\norm{z_j-z_k}^{-\ps{\alpha_j,\alpha_k}}e^{\ps{\bm s,\omega_2}c_2}\expect{\prod_{k=1}^NV_{\alpha_k}(z_k)e^{-\mu_2e^{\gamma\ps{\bm c,e_2}}I^2(\C)}}dc_2\\
		&=\frac{1}{\ps{\bm s,\omega_1}}\int_{\R}\prod_{j<k}\norm{z_j-z_k}^{-\ps{\alpha_j,\alpha_k}}e^{\ps{\bm s,\omega_2}c_2}\expect{e^{-\mu_2e^{\gamma c_2}I^2(\C)}}d c_2\quad\text{while}
	\end{align*}
	\begin{align*}
		&\int_{\R^2}e^{-2\ps{Q,\bm c}}\expect{\mathds1_{\ps{\bm c,e_1}<0}R_{s_1}(\alpha_1)V_{\hat s_1\alpha_1}(z_1)\prod_{k=2}^NV_{\alpha_k}(z_k)e^{-\mu_2e^{\gamma\ps{\bm c,e_2}}M^{\gamma e_2}(\C)}}d\bm c\\
		&=R_{s_1}(\alpha_1)\int^0_{-\infty}e^{\ps{\bm s+\hat s_1\alpha_1-\alpha_1,\omega_1}c_1}dc_1\int_{\R}\prod_{j<k}\norm{z_j-z_k}^{-\ps{\hat\alpha_j,\hat\alpha_k}}e^{\ps{\bm{\hat s},\omega_2}c_2}\expect{e^{-\mu_2e^{\gamma\ps{\bm c,e_2}}\hat I^2(\C)}}dc_2\\
		&=\frac{R_{s_1}(\alpha_1)}{\ps{\bm s+\hat s_1\alpha_1-\alpha_1,\omega_1}}\int_{\R}\prod_{j<k}\norm{z_j-z_k}^{-\ps{\hat\alpha_j,\hat\alpha_k}}e^{\ps{\bm{\hat s},\omega_2}c_2}\expect{e^{-\mu_2e^{\gamma c_1}\hat I^2(\C)}}dc_2.
	\end{align*}
	This shows that when $\ps{\bm s,\omega_i}>0$ for $i=1,2$ the map $F$ is given by 
	\[
	\int_{\R^2}e^{-2\ps{Q,\bm c}}\expect{\prod_{k=1}^NV_{\alpha_k}(z_k)e^{-\sum_{i=1}^2\mu_ie^{\gamma \ps{\bm c,e_i}}M^{\gamma e_i}(\C)}}d\bm c,
	\]
	which is nothing but the defining expression for the correlation functions under the assumptions made.
	The same reasoning shows that when $\ps{\bm s,\omega_2}>0$ with $0>\ps{\bm s,\omega_1}>\ps{\alpha_1-Q,e_1}$:
	\begin{align*}
		&\frac{1}{\ps{\bm s,\omega_1}}\int_{\R}\prod_{j<k}\norm{z_j-z_k}^{-\ps{\alpha_j,\alpha_k}}e^{\ps{\bm s,\omega_2}c_2}\expect{e^{-\mu_2e^{\gamma c_2}I^2(\C)}}d c_2\\
		&=\int_{\R^2}e^{-2\ps{Q,\bm c}}\expect{\mathds1_{\ps{\bm c,e_1}>0}\prod_{k=1}^NV_{\alpha_k}(z_k)e^{-\mu_2e^{\gamma\ps{\bm c,e_2}}M^{\gamma e_2}(\C)}}d\bm c\quad\text{and }
	\end{align*}
	\begin{align*}
		&\frac{R_{s_1}(\alpha_1)}{\ps{\bm s+\hat s_1\alpha_1-\alpha_1,\omega_1}}\int_{\R}\prod_{j<k}\norm{z_j-z_k}^{-\ps{\hat\alpha_j,\hat\alpha_k}}e^{\ps{\bm{\hat s},\omega_2}c_2}\expect{e^{-\mu_2e^{\gamma c_1}\hat I^2(\C)}}dc_2\\
		&=\int_{\R^2}e^{-2\ps{Q,\bm c}}\expect{\mathds1_{\ps{\bm c,e_1}<0}R_{s_1}(\alpha_1)V_{\hat s_1\alpha_1}(z_1)\prod_{k=2}^NV_{\alpha_k}(z_k)e^{-\mu_2e^{\gamma\ps{\bm c,e_2}}M^{\gamma e_2}(\C)}}d\bm c
	\end{align*}
	so that if $\ps{\bm s,\omega_2}>0$ and $0>\ps{\bm s,\omega_1}>\ps{\alpha_1-Q,e_1}$ $F$ coincides with 
	\[
	\int_{\R^2}e^{-2\ps{Q,\bm c}}\expect{\prod_{k=1}^NV_{\alpha_k}(z_k)e^{-\sum_{i=1}^2\mu_ie^{\gamma \ps{\bm c,e_i}}M^{\gamma e_i}(\C)}-\mathcal{R}_{\bm \alpha}(\bm c)}d\bm c
	\]
	as well. The same applies if we now assume that $\ps{\bm s,\omega_2}>0$ with $\ps{\alpha_1-Q,e_1}>\ps{\bm s,\omega_1}>-\gamma$.
	We infer that the map defined by Equation~\eqref{eq:correl_2} is indeed analytic over an open complex neighborhood of $\mathcal{A}_N^{1;0}$. Furthermore as soon as every integral term is holomorphic over $\mathcal{A}_N^{1;0}$, we see that the poles of $F$ are given by $\ps{\bm s,\omega_1}=0$ and $\ps{\bm s+\hat s_1\alpha_1-\alpha_1,\omega_1}=0$. Therefore this claim allows to conclude for the proof of Theorem~\ref{thm:analycity} in this very case. 
	
	In general, the same reasoning still works but we now need to take into account the possible reflection coefficients stemming for all other insertions. For instance let us explain this in the case where we assume that $\ps{\alpha_1-Q,e_1}>-\gamma$ and $\ps{\alpha_2-Q,e_2}>-\gamma$ but $\ps{\alpha-Q,e}<-\gamma$ otherwise, so that we need to incorporate the two reflection coefficients $R_{s_1}(\alpha_1)$ and $R_{s_2}(\alpha_2)$. Let us detail how to do so on this specific example and for this study the asymptotics of the term $e^{-2\ps{Q,\bm c}}\expect{\prod_{k=1}^NV_{\alpha_k}(z_k)e^{-\sum_{i=1}^2\mu_ie^{\gamma \ps{\bm c,e_i}}M^{\gamma e_i}(\C)}}$.
	To start with as $c_1=\ps{\bm c,e_1}\to-\infty$ we can write an expansion of the form (up to lower order terms)
	\begin{align*}
		&\prod_{j<k}	\norm{z_j-z_k}^{-\ps{\alpha_j,\alpha_k}}e^{\ps{\bm s,\bm c}}\expect{e^{-\mu_2e^{\gamma c_2}I^2(\C)}}+R_{s_1}(\alpha_1)\prod_{j<k}\norm{z_j-z_k}^{-\ps{\hat \alpha_j,\hat \alpha_k}}e^{\ps{\bm{\hat s},\bm c}}\expect{e^{-\mu_2e^{\gamma c_2}\hat I^2(\C)}}\\
		&=e^{-2\ps{Q,\bm c}}\expect{\left(V_{\alpha_1}(z_1)+R_{s_1}(\alpha_1)V_{\hat s_1\alpha_1}(z_1)\right)\prod_{k=2}^NV_{\alpha_k}(z_k)e^{-\mu_2e^{\gamma c_2}M^{\gamma e_2}(\C)}}+l.o.t.
	\end{align*}
	But this term may diverge as $c_2\to-\infty$; to remedy this issue we need to add a corrective term, leading us to considering
	\begin{align*}
		&\bm{\mathrm R}^2(\bm c)\coloneqq\mathds 1_{c_1<0}e^{-2\ps{Q,\bm c}}\times\\
		&\expect{\left(e^{-\mu_2e^{\gamma c_2}M^{\gamma e_2}(\C)}-\mathds 1_{c_2<0}\left(V_{\alpha_2}(z_2)+R_{s_2}(\alpha_2)V_{\hat s_2\alpha_2}(z_2)\right)\right)\left(V_{\alpha_1}(z_1)+R_{s_1}(\alpha_1)V_{\hat s_1\alpha_1}(z_1)\right)\prod_{k=3}^NV_{\alpha_k}(z_k)}.
	\end{align*}
	Now if we proceed in the same way by first looking at what happens as $c_2\to-\infty$ we see that we need to incorporate an extra term $\bm{\mathrm R}^1(\bm c)$ defined analogously to $\bm{\mathrm R}^2(\bm c)$ by interchanging the roles of $\alpha_1$ and $\alpha_2$. We thus get the tentative expression $\bm{\mathrm R}^1(\bm c)+\bm{\mathrm R}^2(\bm c)$ for the remainder term. However by doing so there is a redundancy in the cross terms that appear when both $c_1$ and $c_2$ diverge to $-\infty$. To take care of this issue we improve the previous heuristic by defining a remainder term
	\begin{align*}
		&\bm{\mathrm R}(\bm c)=\bm{\mathrm R}^1(\bm c)+\bm{\mathrm R}^2(\bm c)+\bm{\mathrm R}^{12}(\bm c),\quad\text{with}\\
		&\bm{\mathrm R}^{12}(\bm c)\coloneqq \mathds 1_{c_1,c_2<0}e^{-2\ps{Q,\bm c}}\expect{\left(V_{\alpha_2}(z_2)+R_{s_2}(\alpha_2)V_{\hat s_2\alpha_2}(z_2)\right)\left(V_{\alpha_1}(z_1)+R_{s_1}(\alpha_1)V_{\hat s_1\alpha_1}(z_1)\right)\prod_{k=3}^NV_{\alpha_k}(z_k)}
	\end{align*}
	where $\bm{\mathrm R}^{12}(\bm c)$ allows to avoid such redundancies.
	Now we see that under the assumption that $\ps{\bm s,\omega_1}>0$, $\int_{\R^2}\bm{\mathrm R}^1(\bm c)d\bm c$ is equal to
	\begin{align*}
		&\frac1{\ps{\bm s,\omega_1}}\int_\R \scalebox{0.95}{$e^{\ps{\bm s,\omega_2}c_2}\expect{\left(e^{-\mu_2e^{\gamma c_2}M^{\gamma e_2}(\C)}-\mathds 1_{c_2<0}\left(V_{\alpha_2}(z_2)+R_{s_2}(\alpha_2)V_{\hat s_2\alpha_2}(z_2)\right)\right)\prod_{k=1}^NV_{\alpha_k}(z_k)}$}dc_2+\\
		&\scalebox{0.95}{$\frac{R_{s_1}(\alpha_1)}{\ps{\bm s+\hat s_1\alpha_1-\alpha_1,\omega_1}}\int_{\R}e^{\ps{\bm {\hat s},\omega_2}c_2}\expect{\left(e^{-\mu_2e^{\gamma c_2}M^{\gamma e_2}(\C)}-\mathds 1_{c_2<0}\left(V_{\alpha_2}(z_2)+R_{s_2}(\alpha_2)V_{\hat s_2\alpha_2}(z_2)\right)\right)V_{\hat s_1\alpha_1}(z_1)\prod_{k=2}^NV_{\alpha_k}(z_k)}dc_2$}
	\end{align*}
	and under the additional assumption that $\ps{\bm s,\omega_2}>0$:
	\begin{align*}
		&\int_{\R^2}\bm{\mathrm R}^{12}(\bm c)d\bm c=\frac{\expect{\prod_{k=1}^NV_{\alpha_k}(z_k)}}{\ps{\bm s,\omega_1}\ps{\bm s,\omega_2}}+\frac{\expect{R_{s_1}(\alpha_1)V_{\hat s_1\alpha_1}(z_1)\prod_{k=2}^NV_{\alpha_k}(z_k)}}{\ps{\bm s+\hat s_1\alpha_1-\alpha_1,\omega_1}\ps{\bm s,\omega_2}}\\
		&+\frac{\expect{R_{s_2}(\alpha_2)V_{\hat s_2\alpha_2}(z_2)\prod_{k\neq 2}V_{\alpha_k}(z_k)}}{\ps{\bm s,\omega_1}\ps{\bm s+\hat s_2\alpha_2-\alpha_2,\omega_2}}+\frac{\expect{R_{s_1}(\alpha_1)V_{\hat s_1\alpha_1}(z_1)R_{s_2}(\alpha_2)V_{\hat s_2\alpha_2}(z_2)\prod_{k=3}^NV_{\alpha_k}(z_k)}}{\ps{\bm s+\hat s_1\alpha_1-\alpha_1,\omega_1}\ps{\bm s+\hat s_2\alpha_2-\alpha_2,\omega_2}}\cdot
	\end{align*}
	As a consequence we can conclude that if $\ps{\bm s,\omega_i}>0$ for $i=1,2$ then $F$ is equal to 
	\begin{align*}
		\int_{\R^2}e^{-2\ps{Q,\bm c}}\expect{\prod_{k=1}^NV_{\alpha_k}(z_k)e^{-\sum_{i=1}^2\mu_ie^{\gamma \ps{\bm c,e_i}}M^{\gamma e_i}(\C)}}d\bm c
	\end{align*}
	while it is readily seen that $\mathcal{R}_{\bm \alpha}(\bm c)=0$.
	In the same fashion as before we see that in all the possible cases for $\ps{\bm s,\omega}$, the map $F(\bm\alpha)$ coincides with 
	$$\int_{\R^2}e^{-2\ps{Q,\bm c}}\expect{\prod_{k=1}^NV_{\alpha_k}(z_k)e^{-\sum_{i=1}^2\mu_ie^{\gamma \ps{\bm c,e_i}}M^{\gamma e_i}(\C)}-\mathcal R_{\bm\alpha}(\bm c)}d\bm c$$ and is analytic in some complex neighborhood provided that Proposition~\ref{prop:asymptot_many} holds true.

	In the general setting of Theorem~\ref{thm:analycity} we proceed in the same fashion but this time we need to take into account all the possible way to make appear reflection coefficients. For this purpose we introduce for $V$ a subset of the Weyl group $W$ the notation
	\[
	\mathcal{R}_V\prod_{k=1}^NV_{\alpha_k(z_k)}\coloneqq \prod_{k=1}^N\left(\sum_{s\in V}R_s(\alpha_k)V_{\hat s\alpha_k}(z_k)\right)
	\]
	and also set $\mathcal{R}_{V_1}\mathcal{R}_{V_2}\coloneqq \mathcal{R}_{V_1V_2}$ where $V_1V_2=\left\{v_1v_2,\text{ }v_i\in V_i\text{ for }i=1,2\right\}$. We further define for $\bm\alpha\in(Q+\mathcal{C}_-)^N$ the remainder terms that will allow to compensate the divergence in the constant mode $\bm c$
	\begin{equation}\label{eq:defRrm}
		\begin{split}
			\bm{\mathrm{R}}_{\bm\alpha}(\bm c)&\coloneqq\bm{\mathrm{R}}_{\bm\alpha}^1(\bm c)+\bm{\mathrm{R}}_{\bm\alpha}^2(\bm c)+\bm{\mathrm{R}}_{\bm\alpha}^{1,2}(\bm c),\quad\text{with}\\
			\bm{\mathrm{R}}_{\bm\alpha}^1(\bm c)&\coloneqq \mathds 1_{\ps{\bm c,e_2}<0}\left(e^{-\mu_{1}e^{\gamma\ps{\bm c,e_{1}}}M^{\gamma e_{1}}(\C)}-\mathds 1_{\ps{\bm c,e_1}<0}\mathcal{R}_{Id,s_1}\right)\mathcal{R}_{Id,s_2}\prod_{k=1}^NV_{\alpha_k}(z_k),\\
			\bm{\mathrm{R}}_{\bm\alpha}^{1,2}(\bm c)&\coloneqq \mathds 1_{\max\limits_{i=1,2}\ps{\bm c,e_i}<0}\mathcal R_{W}\prod_{k=1}^NV_{\alpha_k}(z_k).
		\end{split}
	\end{equation} Then we may consider the map $F$ defined by
	\begin{equation}\label{eq:ana_F}
		F(\bm\alpha)\coloneqq  \int_{\R^2}E_{\bm\alpha}(\bm c)d\bm c+\sum_{i=1}^2\int_{\R}E^i_{\bm\alpha}(c_i)dc_i+E^{1,2}_{\bm\alpha},\quad\text{where we have set}
	\end{equation}
	\begin{equation*}
		E_{\bm\alpha}(\bm c)=e^{-2\ps{Q,\bm c}}\E\Big[\prod_{k=1}^NV_{\alpha_k}(z_k)e^{-\sum_{i=1}^2\mu_ie^{\gamma \ps{\bm c,e_i}}M^{\gamma e_i}(\C)}-\bm{\mathrm{R}}_{\bm\alpha}(\bm c)\Big], 
	\end{equation*}
	\begin{equation*}    
		\begin{split}
			E^1_{\bm\alpha}(\ps{\bm c,e_1})=&\sum_{w:\{1,\cdots,N\}\to\{Id,s_2\}}\frac{e^{-\ps{\bm s(w),\omega_2}\ps{\bm c,e_2}}}{\ps{\bm s(w),\omega_2}}\\
			&e^{-2\ps{Q,\bm c}}\expect{\left(e^{-\mu_1e^{\gamma\ps{\bm c,e_1}}M^{\gamma e_1}(\C)}-\mathds 1_{\ps{\bm c,e_1}<0}\left(1+\mathcal R_{s_1}\right)\right)\prod_{k=1}^NR_{w(k)}(\alpha_k)V_{\hat w(k)\alpha_k}(z_k)}
		\end{split}
	\end{equation*}
	which indeed depends only on $\bm c$ through $\ps{\bm c,e_1}$, and
	\begin{equation*}
		\begin{split}
			E^{1,2}_{\bm\alpha}= &\sum_{w:\{1,\cdots,N\}\to W}\frac{e^{-\ps{\bm s(w),\bm c}}}{\ps{\bm s(w),\omega_1}\ps{\bm s(w),\omega_2}}e^{-2\ps{Q,\bm c}}\expect{\prod_{k=1}^NR_{w(k)}(\alpha_k)V_{\hat w(k)\alpha_k}(z_k)}\\
			-&\sum_{i=1}^2\sum_{w:\{1,\cdots,N\}\to \{Id,s_i\}}\frac{e^{-\ps{\bm s(w),\bm c}}}{\ps{\bm s(w),\omega_i}}e^{-2\ps{Q,\bm c}}\expect{\prod_{k=1}^NR_{w(k)}(\alpha_k)V_{\hat w(k)\alpha_k}(z_k)}
		\end{split}
	\end{equation*}
	which is independent of $\bm c$.
	The reason for introducing such a map $F$ follows from the observation that $F$ is seen to coincide with 
	\[
	\int_{\R^2}e^{-2\ps{Q,\bm c}}\expect{\prod_{k=1}^NV_{\alpha_k}(z_k)e^{-\sum_{i=1}^2\mu_ie^{\gamma \ps{\bm c,e_i}}M^{\gamma e_i}(\C)}-\mathcal{R}_{\bm \alpha}(\bm c)}d\bm c
	\]
	as soon as every term makes sense.
	Therefore proving Theorem~\ref{thm:analycity} under its most general assumptions boils down to showing that the following holds true:
	\begin{lemma}\label{lemma:analycity}
		Let us denote by $\mathcal{A}_N^{0}$ the subset  of $\mathcal{A}_N$ defined by the condition that $\ps{\alpha_k-Q,e}\not\in-\gamma\N^*\cup-\frac2\gamma\N^*$ for all $1\leq k\leq N$ and $e\in\Phi^+$. 
		Then the following map is holomorphic in a complex neighborhood of $\mathcal{A}_N^{0}$: 
		\begin{equation}
			\begin{split}
				G(\bm\alpha)\coloneqq \int_{\R^2}e^{-2\ps{Q,\bm c}}\E\Big[\prod_{k=1}^NV_{\alpha_k}(z_k)&e^{-\sum_{i=1}^2\mu_ie^{\gamma \ps{\bm c,e_i}}M^{\gamma e_i}(\C)}-\bm{\mathrm{R}}_{\bm\alpha}(\bm c)\Big]d\bm c.
			\end{split}
		\end{equation} 
	\end{lemma}
	The rest of this Subsection is dedicated to proving Lemma~\ref{lemma:analycity}. Proving that terms of the form $\int_R E^i_{\bm \alpha}(c_i)dc_i$ are also holomorphic follows from the very same arguments. This shows that as soon as $\ps{\alpha_k-Q,e}\not\in-\gamma\N^*\cup-\frac2\gamma\N^*$ for all $1\leq k\leq N$ and $e\in\Phi^+$, the poles of the map $F(\bm\alpha)$ defined by Equation~\eqref{eq:ana_F} are given by the $\ps{\bm s(w),\omega_i}=0$ for some $s\in W$ and $i\in\{1,2\}$. When $\ps{\alpha_k-Q,e}\in-\gamma\N^*\cup-\frac2\gamma\N^*$ for some $1\leq k\leq N$ and $e\in\Phi^+$ then the reflection coefficients have a pole (which may still be removable in some cases) so that $F$ does too. This shows that the statement of Theorem~\ref{thm:analycity} holds true as soon as Lemma~\ref{lemma:analycity} does.

	\subsubsection{Analycity of the expectation term}
	To prove Lemma~\ref{lemma:analycity} we start by showing that for fixed $\bm c$, the expectation term that enters the definition of the correlation functions~\eqref{eq:correl_2} is holomorphic. The proof of this statement is similar to that of~\cite[Theorem 6.1]{KRV_DOZZ} to which we refer for additional details. The basic idea is to rely on the fact that, thanks to Girsanov theorem~\ref{thm:girsanov}, this expectation term can be defined by regularizing the GFF and then taking a limit of this regularization. Namely let us consider the circle average regularization considered before and defined by setting for positive $r$
	\begin{equation}\label{regularization_bis}
		\X_r(z)\coloneqq \frac{1}{2\pi}\oint_{\partial B_r(z)}\X(w)\frac{dw}{w}
	\end{equation}
	where $B_r(z)=B(z,e^{-r})$ is the Euclidean ball centered at $z$ with radius $e^{-r}$. 
	Then it follows from the proof of ~\cite[Theorem 3.1]{Toda_construction} that the expectation term that enters the definition of $G$ can be defined using the property that
	\begin{equation}
		\begin{split}
			&\E\Big[\prod_{k=1}^NV_{\alpha_k}(z_k)e^{-\sum_{i=1}^2\mu_ie^{\gamma \ps{\bm c,e_i}}M^{\gamma e_i}(\C)}\Big]=\lim\limits_{r\to+\infty}\prod_{k=1}^N\norm{z_k}_+^{-4\Delta_{\alpha_k}}F_r(\bm\alpha;\bm c),\\
			&\text{with} \quad F_r(\bm\alpha;\bm c)\coloneqq\E\Big[\prod_{k=1}^Ne^{\ps{\alpha_k,\X_r(z_k)+\bm c}-\frac{\expect{\ps{\alpha_k,\X_r(z_k)}^2}}{2}}e^{-\sum_{i=1}^2\mu_ie^{\gamma \ps{\bm c,e_i}}M^{\gamma e_i}(\C_r)}\Big],
		\end{split}
	\end{equation}
	and where $\C_r\coloneqq\C\setminus\left(\cup_{k=1}^NB_r(z_k)\right)$. 
	
	From its expression it is readily seen that $\bm\alpha\to F_r(\bm\alpha;\bm c)$ is holomorphic over $(\C^2)^N$. Moreover, we will show that as $r\to\infty$ for any compact set $K$ of $\left(Q+\mathcal{C}_-\right)^N$ and $\bm\alpha\in K$, 
	\begin{equation}\label{eq:ana_incr}
		\norm{F_{r+1}(\bm\alpha+i\bm \beta;\bm c)-F_r(\bm\alpha+i\bm \beta;\bm c)}\leq Ce^{-\eta r}
	\end{equation}
	for some positive $C$ and $\eta$ uniform over $K$, as soon as $\bm\beta\in(\R^2)^N$ is taken sufficiently close to $0$. Therefore assuming that Equation~\eqref{eq:ana_incr} holds true, we can conclude that for any such $K$ one can find an open complex neighborhood of $K$ inside $(\C^2)^N$ over which $\left(F_r(\bm\alpha;\bm c)\right)_r$ converges uniformly (in $\bm\alpha$) as $r\to\infty$. Since the $F_r(\cdot;\bm c)$ are holomorphic for any $r$ this shows that the limit is also holomorphic. The explicit expression of the reflection coefficients $R_s(\alpha)$ given in Equation~\eqref{eq:refl} shows that such coefficients and therefore the quantities $\bm{\mathrm R}^i_{\bm\alpha}$ are holomorphic in $\bm\alpha$ as soon as $\ps{\alpha_k-Q,e}\not\in -\frac2\gamma\N^*\cup-\gamma\N^*$ for all $e\in\Phi^+$, which we assumed to hold. This shows that for any fixed $\bm c$, 
	\[
	\bm\alpha\mapsto\expect{\prod_{k=1}^NV_{\alpha_k}(z_k)\prod_{i=1}^2\left(e^{-\mu_ie^{\gamma \ps{\bm c,e_i}}M^{\gamma e_i}(\C)}-\mathds 1_{\ps{\bm c,e_i}<0}\bm{\mathrm{R}}_{\bm\alpha}^i\right)}
	\]
	is holomorphic in a complex neighborhood of $\mathcal{A}_N^0$. Therefore proving analycity of the expectation term boils down to proving that Equation~\eqref{eq:ana_incr} does indeed hold.
	
	To show that this is the case we rely on the fact that the increments $\X_{r+1}-\X_r$ are independent of the sigma-algebra generated by the $(\X(z))_{z\in\C_r}$, so that
	\begin{align*}
		&F_r(\bm\alpha+i\bm\beta;\bm c)=\expect{\prod_{k=1}^N e^{\ps{\alpha_k+i\beta_k,\X_{r+1}(z_k)}-\frac{\expect{\ps{\alpha_k+i\beta_k,\X_{r+1}(z_k)}^2}}{2}}e^{-\sum_{i=1}^2\mu_ie^{\gamma \ps{\bm c,e_i}}M^{\gamma e_i}(\C_r)}}.
	\end{align*}
	Interpreting the first term as a Girsanov transform we see that the latter is given by
	\begin{align*}
		&\prod_{l<k}e^{\expect{\ps{\X_r(z_k),\alpha_k}\ps{\X_r(z_l),\alpha_l}}}\times\\
		&\expect{\prod_{k=1}^N e^{\ps{i\beta_k,\X_{r+1}(z_k)}+\frac{\expect{\ps{\beta_k,\X_{r+1}(z_k)}^2}}{2}-i\expect{\ps{\alpha_k,\X_{r+1}(z_k)}\ps{\beta_k,\X_{r+1}(z_k)}}}e^{-\sum_{i=1}^2\mu_ie^{\gamma \ps{\bm c,e_i}}I^i_{r+1}(\C_r)}}
	\end{align*}
	with $I_r^i(d^2x_i)=\prod_{k=1}^Ne^{\gamma\ps{\alpha_k,e_i}\expect{\X(x_i),\X_{r+1}(z_k))}}M^{\gamma e_i}(d^2x_i)$. Therefore
	\begin{align*}
		&\norm{F_{r+1}(\bm\alpha+i\bm\beta;\bm c)-F_r(\bm\alpha+i\bm\beta;\bm c)}\\
		&\leq Ce^{\sum_{k=1}^N\frac{\norm{\beta_k}^2}{2}r}\expect{\norm{e^{-\sum_{i=1}^2\mu_ie^{\gamma \ps{\bm c,e_i}}I_{r+1}^i(\C_r)}-e^{-\sum_{i=1}^2\mu_ie^{\gamma \ps{\bm c,e_i}}I_{r+1}^i(\C_{r+1})}}}.
	\end{align*}
	The expectation term can then be rewritten as
	\begin{align*}
		&\expect{e^{-\sum_{i=1}^2\mu_ie^{\gamma \ps{\bm c,e_i}}I^i_{r+1}(\C_{r})}\left(1-e^{-\sum_{i=1}^2\mu_ie^{\gamma \ps{\bm c,e_i}}I^i_{r+1}(\C_{r+1}\setminus\C_r)}\right)}\\
		&=\expect{e^{-\sum_{i=1}^2\mu_ie^{\gamma \ps{\bm c,e_i}}I^i_{r+1}(\C_{r})}\left(1-e^{-\mu_1e^{\gamma \ps{\bm c,e_1}}I^1_{r+1}(\C_{r+1}\setminus\C_r)}+1-e^{-\mu_2e^{\gamma \ps{\bm c,e_2}}I^2_{r+1}(\C_{r+1}\setminus\C_r)}\right)}\\
		&-\expect{e^{-\sum_{i=1}^2\mu_ie^{\gamma \ps{\bm c,e_i}}I^i_{r+1}(\C_{r})}\left(1-e^{-\mu_1e^{\gamma \ps{\bm c,e_1}}I^1_{r+1}(\C_{r+1}\setminus\C_r)}\right)\left(1-e^{-\mu_2e^{\gamma \ps{\bm c,e_2}}I^2_{r+1}(\C_{r+1}\setminus\C_r)}\right)}.
	\end{align*}
	The set $\C_{r+1}\setminus\C_r$ is the disjoint union of annuli $A_k(r)$ centered at $z_k$ and with radii $(e^{-(r+1)}, e^{-r})$, so that
	\[
	I_{r+1}^i(\C_{r+1}\setminus\C_r)=\sum_{k=1}^N I_{r+1}^i(A_k(r)).
	\]
	This means that we can further decompose the above expectation to reduce the problem to that of showing that 
	\[
	\expect{e^{-\sum_{i=1}^2\mu_ie^{\gamma \ps{\bm c,e_i}}I^i_{r+1}(\C_{r})}\left(1-e^{-\mu_1e^{\gamma \ps{\bm c,e_1}}I^i_{r+1}(A_k(r))}\right)}\leq Ce^{-\eta' r}
	\]
	for positive $C$ and $\eta'$, uniformly on $1\leq k\leq N$ and $i=1,2$. 
	Put differently using H\"older inequality as well as the inequality $1-e^{-x}<x$ for positive $x$ we only need to bound $I_{r+1}^i(A_k(r))$, which has been done along the proof of~\cite[Theorem 6.1]{KRV_DOZZ}:
	\[
	\P(I_{r+1}^i(A_k(r))\geq\eps)\leq \eps^{-m}\expect{I_{r+1}^i(A_k(r))^m}\leq C \eps^{-m}e^{-r\theta_m}
	\]
	for any $m>0$ and with $\theta^k_m=m\gamma\ps{Q-\alpha_k,e_1}-\gamma^2m^2$. Therefore choosing $\eps=e^{-\frac{\theta_m}{p_1+m}r}$ with $\theta_m=\sup\limits_{1\leq k\leq N}\theta_m^k$ we have that $\P(I_{r+1}^i(A_k(r))\geq\eps)^{\frac1{p_1}}\leq e^{-\frac{\theta_m}{p_1+m}r}$, so collecting terms we end up with the bound
	\[
	\norm{F_{r+1}(\bm{\alpha+i\beta})-F_{r}(\bm{\alpha+i\beta})}\leq Ce^{-\eta r}
	\]
	where $\eta= \frac{\theta_m}{p_1+m}-\sum_{k=1}^N\frac{\norm{\beta_k}^2}{2}$. We can conclude provided that $\eta>0$, that is as soon as $\bm\beta$ and $m$ are chosen small enough. 
	
	\subsubsection{Asymptotics of the expectation term: proof of Proposition~\ref{prop:asymptot_many}}\label{subsec:asym}
	The second step in proving Lemma~\ref{lemma:analycity} is the study of the asymptotics of the expectation term that arises in the definition of $G$. Namely  we prove here the following:
	\begin{lemma}\label{lemma:anabis}
		For $\bm\alpha\in\mathcal{A}_N^0$,  there exists a positive constant $C$ and $\xi>0$ such that, uniformly on $\bm c$,
		\[
		\norm{E_{\bm\alpha}(\bm c)}\leq Ce^{-\xi\norm{\bm c}}.
		\]
	\end{lemma}
	The statement of Proposition~\ref{prop:asymptot_many} is then easily recovered from this statement via the assumptions on $\bm s$ made there, which ensure that the terms entering the definition of $\mathcal{R}_{\bm\alpha}(\bm c)$ are integrable in the region where $\ps{\bm c,e_1}$ or $\ps{\bm c,e_2}$ is positive.
	\begin{proof}
		To start with note that since the random variables $I^i(\C)$ have negative moments of any order, we know that for any positive $R>0$
		\[
		e^{R\ps{\bm c,e_1}}\expect{\prod_{k=1}^Ne^{-\sum_{i=1}^2\mu_ie^{\gamma \ps{\bm c,e_i}}I^i(\C)}}\to 0\quad\text{as }\ps{\bm c,e_1}\to+\infty.
		\]  
		Of course the same applies in the region where $\ps{\bm c,e_2}\to +\infty$ and for the three other terms. Therefore we focus on what happens as $\ps{\bm c,e_1}\to-\infty$ with $\ps{\bm c,e_2}$ bounded below, and when both $\ps{\bm c,e_1}$ and $\ps{\bm c,e_2}$ diverge to $-\infty$. To do so our strategy is to show that in these asymptotics the integrals that appear in the expectation term concentrate around the singular points $x=z_k$. Put differently we will see that the behavior of the expectation term is governed by the asymptotics of 
		\[
		\expect{\prod_{k=1}^Ne^{-\sum_{i=1}^2\mu_ie^{\gamma \ps{\bm c,e_i}}I^i(B_{r_k}(z_k))}}
		\]
		for some well-chosen $r_k$.
		
		\subsubsection{The case where $\ps{\bm c,e_2}$ is bounded below}
		We first consider the regime where $\ps{\bm c,e_2}$ is bounded below. Our goal is to prove that there exists a positive $\xi>0$ such that, as $\ps{\bm c,e_1}\to -\infty$ with $\ps{\bm c,e_2}$ bounded below:
		\begin{equation}\label{eq:asymptot_1term}
			e^{-2\ps{Q,c}}\expect{e^{-\mu_2e^{\gamma \ps{\bm c,e_2}}M^{\gamma e_2}(\C)}\left(e^{-\mu_1e^{\gamma \ps{\bm c,e_1}}M^{\gamma e_1}(\C)}-\mathcal{R}_{Id,s_1}\right)\prod_{k=1}^NV_{\alpha_k}(z_k)}=\mathcal{O}\left(e^{\xi\ps{\bm c,e_1}}\right).
		\end{equation}
		
		To start with we choose $\eps>0$ small enough and consider for $\bm c\in\R^2$ the radii $r_{k}=r_{k}(\bm c)$ such that \[
		(1+\eps)\left(\frac{\ps{\bm c,e_1}_-}{2\ps{\alpha_k-Q,\rho}}\vee\frac{\ps{\bm c,e_1}_-}{\ps{Q-\alpha_k,e_1}-2\gamma}\mathds 1_{\ps{\alpha_k-Q,e_1}>-\gamma}\right)\leq r_{k}\leq (1-\eps)\frac{\ps{\bm c,e_1}_-}{\ps{\alpha_k-Q,e_1}}
		\] and assume that $\ps{\bm c,e_1}$ is negative enough so that the balls $B_{r_k}(z_k)=B(z_k,e^{-r_k})$ remain disjoint. Note that the above bounds can be satisfied as soon as $\ps{\alpha_k-Q,e_1}\neq -\gamma$ (which we assumed to hold) and for $\eps$ small enough. Then the integral $M^{\gamma e_1}(\C)$ can be decomposed as $M^{\gamma e_1}(\C)=M^{\gamma e_1}(\C_{r})+\sum_{k=1}^NM^{\gamma e_1}(B_{r_{k}}(z_k))$ with $\C_{r}\coloneqq\C\setminus\left(\cup_{k=1}^NB_{r_k}(z_k)\right)$. This allows to write that for any bounded continuous map $F$ over $H^{-1}(\C\to\R^2,g)$
		\begin{align*}
			&e^{-2\ps{Q,\bm c}}\expect{\prod_{k=1}^NV_{\alpha_k}(z_k)e^{-\mu_1e^{\gamma \ps{\bm c,e_1}}M^{\gamma e_1}(\C)}F(\X)}=\prod_{j<k}\norm{z_j-z_k}^{-\ps{\alpha_k,\alpha_j}}e^{\ps{\bm s,\bm c}}\expect{e^{-\mu_1e^{\gamma \ps{\bm c,e_1}}I^1(\C)}G(\X)},
		\end{align*}
		where $G(\X)\coloneqq F\left(\X+\sum_{k=1}^N\alpha_k G(z_k,\cdot)\right)$ and with \begin{align*}
			\expect{e^{-\mu_1e^{\gamma \ps{\bm c,e_1}}I^1(\C)}G(\X)}&=\expect{\prod_{k=1}^Ne^{-\mu_1e^{\gamma \ps{\bm c,e_1}}I^1(B_{r_{k}}(z_k))}G(\X)}\\
			&+\expect{\left(e^{-\mu_1e^{\gamma \ps{\bm c,e_1}}I^1(\C_r)}-1\right)\prod_{k=1}^Ne^{-\mu_1e^{\gamma \ps{\bm c,e_1}}I^1(B_{r_{k}}(z_k))}G(\X)}.
		\end{align*}
		
		Let us start by considering the first expectation term, which we rewrite under the form
		\begin{align*}
			\sum_{\mathcal{U}\subset\{1,\cdots,N\}}\expect{\prod_{k\in\mathcal{U}}\left(e^{-\mu_1e^{\gamma \ps{\bm c,e_1}}I^1(B_{r_{k}}(z_k))}-1\right)G(\X)}
		\end{align*}
		where the sum ranges over subsets $\mathcal{U}$ of $\{1,\cdots,N\}$. 
		In the case where $\ps{\alpha_1,e_1}<\frac2\gamma$, we have $1<\frac{\ps{Q-\alpha_1,e_1}}\gamma$ and therefore the term
		\begin{align*}
			\expect{I^1(B_{r_{1}}(z_1))\prod_{k\neq 1\in\mathcal{U}}\left(e^{-\mu_1e^{\gamma \ps{\bm c,e_1}}I^1(B_{r_{k}}(z_k))}-1\right)G(\X)}
		\end{align*} 
		is well defined thanks to Equation~\eqref{eq:moments_I}, which allows to write that
		\begin{align*}
			\sum_{\mathcal{U}\subset\{1,\cdots,N\}}\expect{\left(e^{-\mu_1e^{\gamma \ps{\bm c,e_1}}I^1(B_{r_{1}}(z_1))}-1\right)\prod_{k\neq1\in\mathcal{U}}\left(e^{-\mu_1e^{\gamma \ps{\bm c,e_1}}I^1(B_{r_{k}}(z_k))}-1\right)G(\X)}=\mathcal{O}\left(e^{\gamma\ps{\bm c,e_1}}\right).
		\end{align*}
		
		When $\ps{\alpha_k-Q,e_1}>-\gamma$ for some $k\in\mathcal{U}$, the analysis is more subtle. Namely one needs to use the radial-angular decomposition~\eqref{eq:rad_ang_dec} around each insertion $z_k$  to put the integrals involved  under the form
		\begin{align*}
			I^1(B_{r_k}(z_k))&=\int_{r_k}^{+\infty}e^{\gamma \ps{ B^{k}_t,e_1}}\int_0^{2\pi}F^k_1(t,\theta)M^{\gamma e_1}_{\Y^k}(dt,d\theta)\\
			&=e^{\gamma \ps{B^{k}_{r_k},e_1}}\int_0^{+\infty}e^{\gamma \tilde B^{k}_t}\int_0^{2\pi}F^k_1(t+r_k,\theta)M^{\gamma e_1}_{\Y^k}(dt+r_k,d\theta)
		\end{align*}
		where $B^{k}_t\coloneqq \X_t(z_k)+(\alpha_k-Q)t$ and $\Y^{k}(t,\theta)\coloneqq \X(z_k+e^{-t+i\theta})-\X_t(z_k)$ have the law of the pair described in Equation~\eqref{eq:rad_ang_dec}, while $F^k_1(t,\theta)\coloneqq \prod_{l\neq k}\left(\frac{\norm{z_k+e^{-t+i\theta}}_+}{\norm{z_k-z_l+e^{-t+i\theta}}}\right)^{\gamma\ps{\alpha_k,e_1}}$. In the last equation $\tilde B^k$ is a one-dimensional Brownian motion with drift $\nu_k\coloneqq\ps{\alpha_k-Q,e_1}$ and variance $2$, started from the origin and independent of the sigma algebra generated by the $(\X(z))_{z\not\in B_{r_k}(z_k)}$. We can apply Williams path decomposition~\cite{williams} to this Brownian motion, which allows to write that 
		\begin{align*}
			I^1(B_{r_k}(z_k))&=e^{\gamma \ps{B^{k}_{r_k}+\M_k,e_1}}\int_0^{+\infty}e^{\gamma \mathcal B^{k}_t}\int_0^{2\pi}F^k_1(t+{r_k},\theta)M^{\gamma e_1}_{\Y^k}(dt+{r_k},d\theta)
		\end{align*}
		where $\mathcal{B}^k$ is the one-dimensional process described in Subsection~\ref{subsec:brown_cond_neg}, started from $-\ps{\M_k,e_1}$ which is is an exponential variable of parameter $-\nu_k$ independent of everything. Therefore, with $q=\norm{\mathcal{U}}$ and for any $G_{r}$ that only depends on $(\X_z)_{z\in \C_{r}}$,  
		\begin{align*}
			&\expect{\prod_{k\in\mathcal{U}}\left(e^{-\mu_1e^{\gamma \ps{\bm c,e_1}}I^1(B_{r_k}(z_k))}-1\right)G_{r}(\X)}\\
			&=\int_{(0,+\infty)^{q}}\prod_{k\in\mathcal{U}} (-\nu_k)e^{\nu_k\ps{\M_k,e_1}}\mathrm d\ps{\M_k,e_1}\expect{\prod_{k=1}^p\left(e^{-\mu_1e^{\gamma \ps{\bm c+B^k_{r_k}+\M_k,e_1}}J_k^r(-\ps{\M_k,e_1})}-1\right)G_{r}(\X)}\\
			&\text{with}\quad J_k^r(-\ps{\M_k,e_1})\coloneqq \int_0^{+\infty}e^{\gamma \mathcal B^{k}_t}\int_0^{2\pi}F^k_1(t+r_k,\theta)M^{\gamma e_1}_{\Y^k}(dt+r_k,d\theta)
		\end{align*}
		and where we use the notation $J_k^r(-\ps{\M_k,e_1})$ to stress that the process $\mathcal{B}^k$ is started from $-\ps{\M,e_1}$.
		We can now make a change of variable $\ps{\M_k,e_1}\leftrightarrow\ps{\M_k+\bm c+B^k_{r_k},e_1}$ to end up with
		\begin{align*}
			\expect{\int_{\R^q}\prod_{k\in\mathcal{U}} (-\nu_k)e^{\nu_k\ps{\M_k-\bm c-B^k_{r_k},e_1}}\mathrm d\ps{\M_k,e_1}\left(e^{-\mu_1e^{\gamma \ps{\M_k,e_1}}J_k^r(\ps{\bm c+B^k_{r_k}-\M_k,e_1})}-1\right)\mathds 1_{\ps{\M_k-\bm c-B^k_{r_k},e_1}>0}G_r(\X)}.
		\end{align*}
		We can interpret the exponential term $\prod_{k\in\mathcal{U}} e^{-\nu_k\ps{B^k_{r_k},e_1}}$ as a Girsanov transform. Namely recalling that $B^k_{r_k}=\X_{r_k}(z_k)$, we see that
		\begin{align*}
			&-\sum_{k\in\mathcal{U}}\nu_k\ps{B^k_{r_k},e_1}=\sum_{k\in\mathcal{U}}\ps{Q-\alpha_k,e_1}\ps{\X_{r_k}(z_k),e_1}-\ps{Q-\alpha_k,e_1}^2r_k\\
			&=\sum_{k\in\mathcal{U}}\ps{Q-\alpha_k,e_1}\ps{\X_{r_k}(z_k),e_1}-\frac{\expect{\left(\sum_{k\in\mathcal{U}}\ps{Q-\alpha_k,e_1}\ps{\X_{r_k}(z_k),e_1}\right)^2}}{2}\\
			&+\sum_{k< l\in\mathcal{U}}\ps{Q-\alpha_k,e_1}\ps{Q-\alpha_l,e_1}G_r'(z_k,z_l),\quad\text{where}\\
			&G_r'(z_k,z_l)\coloneqq \frac1{2\pi}\oint_{\partial B_{r_l}(z_l)}G_{r_k}(z_k,w)\frac{dw}{w}\quad\text{and}\quad G_{r_k}(z_k,\cdot)\coloneqq \frac1{2\pi}\oint_{\partial B_{r_k}(z_k)}G(w,\cdot)\frac{dw}{w}\cdot
		\end{align*}
		Note that if we take $\ps{\bm c,e_1}_-$ large enough, then $G'_{r_k} (z_k,z_l)=G(z_k,z_l)$ for $k\neq l$. Likewise explicit computations show that $G'_r(z_k,z_k)=r+2\ln\norm{z_k}_+$ for $r$ large enough. Therefore without loss of generality we can assume that both assumptions hold in the sequel. 
		
		Now in virtue of Theorem~\ref{thm:girsanov} this exponential term has the effect of shifting the law of $\X$ by $\sum_{k}\ps{Q-\alpha_k,e_1}G_{r_k}(z_k,\cdot)e_1$, and in particular shifts the law of $B^k_r$ by 
		\[
		\lambda_k(r_k)\coloneqq\ps{Q-\alpha_k,e_1}(r_k+2\ln\norm{z_k}_+)e_1+\sum_{l\neq k}\ps{Q-\alpha_k,e_1}G(z_k,z_l)e_1.
		\]
		This shows that the above expectation term is equal to
		\begin{align*}
			&=\prod_{k\neq l\in\mathcal{U}}e^{\ps{Q-\alpha_k,e_1}\ps{Q-\alpha_l,e_1}G(z_k,z_l)}\prod_{k\in\mathcal{U}} (-\nu_k)e^{-\nu_k\ps{\bm c,e_1}}\int_{\R^q}\prod_{k\in\mathcal{U}} (-\nu_k)e^{\nu_k\ps{\M_k,e_1}}\mathrm d\ps{\M_k,e_1}\\\
			&\expect{\left(e^{-\mu_1e^{\gamma \ps{\M_k,e_1}}J_k^r(\ps{\lambda_k(r_k)-\M_k,e_1})}-1\right)\mathds 1_{\ps{\M_k-\lambda_k(r_k),e_1}>0}G_r\left(\X+\sum_{k\in\mathcal{U}}\ps{Q-\alpha_k,e_1}G(z_k,\cdot)e_1\right)}
		\end{align*}
		where on the last line we have used that for $z\in\C_r$, $G_{r_k}(z_k,z)=G(z_k,z)$ for all $1\leq k\leq N$, and with
		\[
		\lambda_k(r_k)\coloneqq \bm c+B^k_{r_k}+\ps{Q-\alpha_k,e_1}(r_k+2\ln\norm{z_k}_+)e_1+\sum_{l\neq k}\ps{Q-\alpha_k,e_1}G(z_k,z_l)e_1.
		\]
		Now we have assumed that for some positive $\eps$ we have $r_{k,1}<(1-\eps)\frac{\ps{\bm c,e_1}}{\ps{\alpha_k-Q,e_1}}$ for all $1\leq k\leq N$. As a consequence we see that $\ps{\lambda_k(r),e_1}\to-\infty$ almost surely, so that along the same lines as in the proof of~\cite[Proposition 4.10]{Toda_correl1} the latter will be asymptotically equivalent to 
		\begin{align*}
			\prod_{k\neq l\in\mathcal{U}}e^{\ps{Q-\alpha_k,e_1}\ps{Q-\alpha_l,e_1}G(z_k,z_l)}\prod_{k\in\mathcal{U}} R_{s_1}(\alpha_k)e^{\ps{Q-\alpha_k,e_1}\ps{\bm c,e_1}}\expect{G_r\left(\X+\sum_{k\in\mathcal{U}}\ps{Q-\alpha_k,e_1}G(z_k,\cdot)e_1\right)}.
		\end{align*} 
		More precisely the reasoning developed in the proof of~\cite[Proposition 4.10]{Toda_correl1}, based on the Markov property for the process entering the definition of $J_k^r$ (see also the proof of Equation~\eqref{eq:asymptot_2term} below), allows to write that 
		\begin{align*}
			&\expect{\prod_{k\in\mathcal{U}}\left(e^{-\mu_1e^{\gamma \ps{\bm c,e_1}}I^1(B_r(z_k))}-1\right)G_r(\X)}=\\
			&\prod_{k\neq l\in\mathcal{U}}e^{\ps{Q-\alpha_k,e_1}\ps{Q-\alpha_l,e_1}G(z_k,z_l)}\prod_{k\in\mathcal{U}} R_{s_1}(\alpha_k)e^{\ps{Q-\alpha_k,e_1}\ps{\bm c,e_1}}\expect{G_r\left(\X+\sum_{k\in\mathcal{U}}\ps{Q-\alpha_k,e_1}G(z_k,\cdot)e_1\right)}\\
			&+\mathcal{O}\left(e^{(1-\eta)\ps{\bm c,\gamma e_1}}\right)
		\end{align*}
		as soon as $\norm{\mathcal{U}}\geq1$.
		
		As a consequence recollecting terms yields
		\begin{align*}
			&e^{-2\ps{Q,\bm c}}\expect{\prod_{k=1}^NV_{\alpha_k}(z_k)e^{-\mu_1e^{\gamma \ps{\bm c,e_1}}M^{\gamma e_1}(B_{r_k}(z_k))}F_r(\X)}\\
			&=\sum_{\mathcal{U}\subset \{1,\cdots,N\}}e^{-2\ps{Q,\bm c}}\expect{\prod_{k\not\in\mathcal{U}}V_{\alpha_k}(z_k)\prod_{k\in\mathcal{U}}R_{s_1}(\alpha_k)V_{\hat s_1\alpha_k}(z_k) F_r(\X)}+\mathcal{O}\left(e^{\ps{\bm s,\bm c}+(1-\eta)\ps{\bm c,\gamma e_1}}\right)\\
			&=e^{-2\ps{Q,\bm c}}\expect{\prod_{k=1}^N\left(V_{\alpha_k}(z_k)+R_{s_1}(\alpha_k)V_{\hat s_1\alpha_k}(z_k)\right) F_r(\X)}+\mathcal{O}\left(e^{\ps{\bm s,\bm c}+(1-\eta)\ps{\bm c,\gamma e_1}}\right).
		\end{align*}
		Under the assumption that $\ps{\bm s,\omega_1}>-\gamma$ (which we assumed to hold since $\bm\alpha\in\mathcal{A}_N$) the remainder term is as desired, a term with asymptotic bounded by a 
		$\mathcal{O}\left(e^{\xi\ps{\bm c, e_1}}\right)$ for some positive $\xi$ as $\ps{\bm c,e_1}\to-\infty$ with $\ps{\bm c,e_2}$ bounded below. By choosing $F_r(\X)=M^{\gamma e_2}(\C_r)$ we have therefore proved that
		\begin{align*}
			e^{-2\ps{Q,c}}\expect{\left(e^{-\mu_1e^{\gamma \ps{\bm c,e_1}}M^{\gamma e_1}(\C\setminus\C_r)}-\mathcal{R}_{Id,s_1}\right)\prod_{k=1}^NV_{\alpha_k}(z_k)e^{-\mu_2e^{\gamma \ps{\bm c,e_2}}M^{\gamma e_2}(\C_{r})}}
		\end{align*}
		is a $\mathcal{O}\left(e^{\xi\ps{\bm c, e_1}}\right)$
		in this asymptotic. As a consequence it remains to control the terms
		\begin{align*}
			&e^{-2\ps{Q,\bm c}}\expect{\prod_{k=1}^NV_{\alpha_k}(z_k)\left(e^{\mu_1e^{\gamma \ps{\bm c,e_1}}M^{\gamma e_1}(\C_r)}-1\right)e^{-\sum_{i=1}^2\mu_ie^{\gamma \ps{\bm c,e_i}}M^{\gamma e_i}(\C)}}\quad\text{and}\\
			&e^{-2\ps{Q,\bm c}}\expect{\left(e^{-\mu_1e^{\gamma \ps{\bm c,e_1}}M^{\gamma e_1}(\C\setminus\C_r)}-\mathcal{R}_{Id,s_1}\right)\left(e^{\mu_2e^{\gamma \ps{\bm c,e_2}}M^{\gamma e_2}(\C\setminus\C_{r})}-1\right)\prod_{k=1}^NV_{\alpha_k}(z_k)}.
		\end{align*}
		
		We start with the second term. Along the same lines as above it suffices to show that
		\begin{align*}
			\expect{\left(e^{-\mu_1e^{\gamma\ps{\bm c,e_1}}I^1(B_{r_{1}}(z_1))}-1-R_{s_1}(\alpha_1)V_{\hat s_1\alpha_1-\alpha_1}(z_1)\right)\left(e^{\mu_2e^{\gamma \ps{\bm c,e_2}}I^2(B_{r_{1}}(z_1))}-1\right)}
		\end{align*}
		is a $\mathcal{O}\left(e^{(1-\eta)\gamma\ps{\bm c,e_1}}\right)$. 
		Now like before we see that the law of the planar, drifted Brownian motion $X_{t+r_1}(z_1)-\X_{r_1}(z_1)+(\alpha_1-Q)t$ can be realized by sampling $\ps{\M_1,e_1}$ according to its marginal law and then sampling the process $\bm{\mathcal{B}}^1$, whose $e_1$ component is the process $\mathcal{B}^1_t$ used above (and described in Subsection~\ref{subsec:brown_cond_neg}) while its $\omega_2$ component is an independent Brownian motion $\tilde B^1$ with drift $\ps{\alpha_1-Q,\omega_2}$ and variance $\norm{\omega_2}^2=\frac23$. This shows that its $e_2$ component is the independent sum of $-\frac12\mathcal{B}^1$  and $\frac32\tilde B^1$, so that 
		\[
		I^2(B_{r_{1}}(z_1))=e^{\gamma\left(\ps{\X_{r_1}(z_1)+(\alpha_0-Q)r_1,e_2}\right)}\int_0^{+\infty}e^{\gamma \left(-\frac12\mathcal{B}^1+\frac32\tilde B^1_t\right)}\int_0^{2\pi}F^1_2(t+r_1,\theta)M^{\gamma e_2}_{\Y^1}(dt+r_1,d\theta).
		\]
		The Girsanov term $e^{\ps{Q-\alpha_1,e_1}B^1_{r_1}}$ has the effect of changing $I^2$ into $\hat I^2=e^{\ps{\alpha_0-Q,e_1}r_1}I^2$. Thanks to our assumption that $r_{1}\geq (1+\eps)\frac{\ps{\bm c,e_1}}{2\ps{\alpha_0-Q,\rho}}$, which ensures that $\ps{\alpha_0-Q,\rho}r_1+\frac{\ps{\M_1,e_1}}{2}\to-\infty$ almost surely for $\ps{\M_1,e_1}=-\ps{\lambda_1(r_1),e_1}$, we see that for $\xi>0$ small enough
		\[
		e^{\xi r_1}\times e^{\gamma\left(\ps{\X_{_1}(z_1),e_2}+\ps{\alpha_0-Q,\rho}r_1\right)}\int_0^{+\infty}e^{\gamma \left(-\frac12\mathcal{B}^1+\frac32\tilde B^1_t\right)}\int_0^{2\pi}F^1_2(t+r_1,\theta)M^{\gamma e_2}_{\Y^1}(dt+r_1,d\theta)\to 0
		\]
		as $\ps{\bm c,e_1}\to-\infty$, almost surely. Therefore 
		\begin{align*}
			\expect{\left(e^{-\mu_1e^{\gamma\ps{\bm c,e_1}}I^1(B_{r_{1}}(z_1))}-1-R_{s_1}(\alpha_1)V_{\hat s_1\alpha_1-\alpha_1}(z_1)\right)\left(e^{\mu_2e^{\gamma \ps{\bm c,e_2}}I^2(B_{r_{1}}(z_1))}-1\right)}
		\end{align*}
		is a lower order term compared to 
		\begin{align*}
			\expect{\left(e^{-\mu_1e^{\gamma\ps{\bm c,e_1}}I^1(B_{r_{1}}(z_1))}-1-R_{s_1}(\alpha_1)V_{\hat s_1\alpha_1-\alpha_1}(z_1)\right)e^{\mu_2e^{\gamma \ps{\bm c,e_2}}I^2(\C\setminus B_{r_{1}}(z_1))}}
		\end{align*}
		which was already shown to be a $\mathcal{O}\left(e^{(1-\eta)\gamma\ps{\bm c,e_1}}\right)$. As a consequence to show that Equation~\eqref{eq:asymptot_1term} does indeed hold it only remains to prove that
		\begin{align*}
			&e^{-2\ps{Q,\bm c}}\expect{\prod_{k=1}^NV_{\alpha_k}(z_k)\left(e^{\mu_1e^{\gamma \ps{\bm c,e_1}}M^{\gamma e_1}(\C_r)}-1\right)e^{-\sum_{i=1}^2\mu_ie^{\gamma \ps{\bm c,e_i}}M^{\gamma e_i}(\C)}}=\mathcal{O}\left(e^{\gamma\ps{\bm c,e_1}}\right).
		\end{align*}
		
		For this we can write that
		\begin{align*}
			&\expect{I^1(\C_{r})e^{-\sum_{i=1}^2\mu_ie^{\gamma \ps{\bm c,e_i}}I^i(\C)}}=\int_{\C_{r}}\prod_{k=1}^N\left(\frac{\norm{x}_+}{\norm{z_k-x}}\right)^{\gamma\ps{\alpha_k,e_1}} \expect{V_{\gamma e_1}(x)e^{-\sum_{i=1}^2\mu_ie^{\gamma \ps{\bm c,e_i}}I^i(\C)}}d^2x\\
			&=\int_{\C_{r}}\prod_{k=1}^N\left(\frac{\norm{x}_+}{\norm{z_k-x}}\right)^{\gamma\ps{\alpha_k,e_1}}\expect{e^{-\sum_{i=1}^2\mu_ie^{\gamma \ps{\bm c,e_i}}\tilde I^i(x)}}d^2x,\\
			&\text{where}\quad\tilde I^i(x)=\int_{\C}\prod_{k=1}^N\left(\frac{\norm{y}_+}{\norm{z_k-y}}\right)^{\gamma\ps{\alpha_k,e_1}}\left(\frac{\norm{y}_+}{\norm{y-x}}\right)^{2\gamma^2}M^{\gamma e_1}(d^2y).
		\end{align*}
		When $r\to+\infty$, the singularities around $x=z_k$ are integrable if $\ps{\alpha_k,e_1}<\frac2\gamma$, so that in that case the remainder term is a $\mathcal{O}\left(e^{\gamma\ps{\bm c,e_1}}\right)$. On the contrary if $\ps{\alpha_k,e_1}>\frac2\gamma$ then around $x=z_k$ (say in an annulus $A_k(r)$ centered at $z_k$ and radii $e^{-r_{k}}$ and $2e^{-r_{k}}$) we can make the change of variable $x\leftrightarrow z_k+e^{-r_{k}}(x-z_k)$ to see that 
		\begin{align*}
			&\expect{I^1(A_k(r))e^{-\sum_{i=1}^2\mu_ie^{\gamma \ps{\bm c,e_i}}I^i(\C)}}\\
			&\sim e^{-r_{k}\left(2-\gamma\ps{\alpha_k,e_1}\right)}\int_{A_k(1)}\prod_{k=1}^N\left(\frac{\norm{z_k}_+}{\norm{z_k-x}}\right)^{\gamma\ps{\alpha_k,e_1}}\expect{e^{-\sum_{i=1}^2\mu_ie^{\gamma \ps{\bm c,e_i}}\tilde I^i\left(z_k+e^{-r_k}(x-z_k)\right)}}d^2x.
		\end{align*}
		Now when $r_k\to+\infty$, we can use the fusion asymptotics~\eqref{eq:fusion} from Subsection~\ref{subsec:fusion} to see that for any positive $\eps$: 
		\begin{align*}
			&\expect{e^{-\sum_{i=1}^2\mu_ie^{\gamma \ps{\bm c,e_i}}\tilde I^i\left(z_k+e^{-r_k}(x-z_k)\right)}}=o\left(e^{-r_k(\frac{\ps{\alpha_k+\gamma e_1-Q,e_1}^2}4-\eps)}\right).
		\end{align*}
		Therefore we end up with the bound
		\[
		\expect{\left(e^{\mu_1e^{\gamma \ps{\bm c,e_1}}I^1(\C_r)}-1\right)e^{-\sum_{i=1}^2\mu_ie^{\gamma \ps{\bm c,e_i}}M^{\gamma e_i}(\C)}}\leq Ce^{\gamma\ps{\bm c,e_1}}e^{-r_k\left(2-\gamma\ps{\alpha_k,e_1}+\frac{\ps{\alpha_k+\gamma e_1-Q,e_1}^2}4-\eps\right)}
		\]
		which shows that this term is a lower order term too since the exponent $2-\gamma\ps{\alpha_k,e_1}+\frac{\ps{\alpha_k+\gamma e_1-Q,e_1}^2}4=\frac{\ps{\alpha_k-Q,e_1}^2}4$ is positive. 
		
		This concludes for the proof of Equation~\eqref{eq:asymptot_1term}.
		
		\subsubsection{The case where $\bm c\to\infty$ inside $\mathcal{C}_-$}
		We can proceed in a similar fashion for the asymptotic where both $\ps{\bm c,e_1}$ and $\ps{\bm c,e_2}$ diverge to $-\infty$. In that case we aim to prove that
		\begin{equation}\label{eq:asymptot_2term}
			\begin{split}
				&\expect{\prod_{k=1}^NV_{\alpha_k}(z_k)e^{-\sum_{i=1}^2\mu_ie^{\gamma \ps{\bm c,e_i}}M^{\gamma e_i}(\C)}}\\
				=&\expect{\prod_{k=1}^N\left(\sum_{s\in W}R_s(\alpha_{k})V_{\hat s\alpha_{k}}(z_k)\right)}+\mathcal{O}\left(\sum_{i=1}^2e^{(1-\eta)\gamma\ps{\bm c,e_i}}\right).
			\end{split}
		\end{equation}
		
		In order to prove Equation~\eqref{eq:asymptot_2term} we first pick $r_k$ like before and write
		\begin{align*}
			&\expect{\prod_{k=1}^Ne^{-\sum_{i=1}^2\mu_ie^{\gamma \ps{\bm c,e_i}}I^i(\C)}}=\sum_{\mathcal{U}_i\subset \{1,\cdots,N\}}\expect{\prod_{i=1}^2e^{-\mu_ie^{\gamma \ps{\bm c,e_i}}I^i(\C_r)}\prod_{k\in\mathcal{U}_i}\left(e^{-\mu_ie^{\gamma \ps{\bm c,e_i}}I^i(B_{r_k}(z_k))}-1\right)},\end{align*}
		which we can further put under the form 
		\begin{align*}
			\sum_{\stackrel{\mathcal{U}_1,\mathcal{U}_2,\mathcal{U}_3\subset \{1,\cdots,N\}}{disjoint}}\E\Big[\prod_{i=1}^2&e^{-\mu_ie^{\gamma \ps{\bm c,e_i}}I^i(\C_r)}\\
			&\prod_{k_i\in\mathcal{U}_i}\left(e^{-\mu_ie^{\gamma \ps{\bm c,e_i}}I^i(B_{r_k}(z_{k_i}))}-1\right)\prod_{k_3\in\mathcal{U}_3}\left(e^{-\mu_ie^{\gamma \ps{\bm c,e_i}}I^i(B_{r_k}(z_{k_3}))}-1\right)\Big].
		\end{align*}
		We now distinguish on whether $k\in\mathcal{U}_1$, $\mathcal{U}_2$ or $\mathcal{U}_3$. If $k$ belongs to $\mathcal{U}_i$ for $i=1$ or $i=2$ then the only integral surrounding the singular point $z=k$ is given by $I^i(B_{r_k}(z_k))$. This term can be processed along the same lines as above by applying Williams path decomposition to the one-dimensional Brownian motion $\ps{\X_{t+r}(z_k)-\X_r(z_k)+\ps{\alpha_k-Q}t,e_i}$. Hence 
		\begin{align*}
			&\expect{\prod_{k=1}^Ne^{-\sum_{i=1}^2\mu_ie^{\gamma \ps{\bm c,e_i}}I^i(\C)}}=\\
			&\sum_{\stackrel{\mathcal{U}_1,\mathcal{U}_2,\mathcal{U}_3\subset \{1,\cdots,N\}}{disjoint}}\expect{\prod_{i=1}^2e^{-\mu_ie^{\gamma \ps{\bm c,e_i}}I^i(\C_r)}\prod_{k_i\in\mathcal{U}_i}R_{s_i}(\hat s_i\alpha_{k_i})V_{\hat s_i\alpha_{k_i}-\alpha_{k_i}}(z_{k_i})\prod_{k_3\in\mathcal{U}_3}\left(e^{-\mu_ie^{\gamma \ps{\bm c,e_i}}I^i(B_{r_k}(z_{k_3}))}-1\right)}\\
			&+\mathcal{O}\left(\sum_{i=1}^2e^{(1-\eta)\gamma\ps{\bm c,e_i}}\right).
		\end{align*}
		
		For terms that correspond to $k_3\in\mathcal{U}_3$, because of our assumptions on the weights we know that $\ps{s\nu_k-\nu_k,\omega_i}<-\gamma$ for $i=1$ or $i=2$; without loss of generality we can assume that $\ps{s_1s_2\nu_k-\nu_k,\omega_1}<-\gamma$. The asymptotic of the corresponding term is then governed  by
		\begin{align*}
			-\mu_1e^{\gamma \ps{\bm c,e_1}}\expect{I^1(B_{r_k}(z_{k_3}))\prod_{i=1}^2e^{-\mu_ie^{\gamma \ps{\bm c,e_i}}I^i(\C_r)}\prod_{k_i\in\mathcal{U}_i}R_{s_i}(\hat s_i\alpha_{k_i})V_{\hat s_i\alpha_{k_i}-\alpha_{k_i}}(z_{k_i})}
		\end{align*}
		which is a $\mathcal{O}\left(e^{\gamma\ps{\bm c,e_1}}\right)$ via the same analysis as before. Thanks to the fact that for $s\not\in\{Id,s_1,s_2\}$, $V_{\hat s\alpha_k-\alpha_k}$ is a $\mathcal{O}\left(\sum_{i=1}^2e^{\gamma\ps{\bm c,e_i}}\right)$, this shows that Equation~\eqref{eq:asymptot_2term} does indeed hold.
		
		\subsubsection{Recollecting terms}
		Recall that we have defined $E_{\bm\alpha}(\bm c)$ via the expression
		\[
		E_{\bm\alpha}(\bm c)=e^{-2\ps{Q,\bm c}}\E\Big[\prod_{k=1}^NV_{\alpha_k}(z_k)e^{-\sum_{i=1}^2\mu_ie^{\gamma \ps{\bm c,e_i}}M^{\gamma e_i}(\C)}-\bm{\mathrm{R}}^1_{\bm\alpha}(\bm c)-\bm{\mathrm{R}}^2_{\bm\alpha}(\bm c)-\bm{\mathrm{R}}^{1,2}_{\bm\alpha}(\bm c)\Big],
		\]
		where the terms that appear are defined in Equation~\eqref{eq:defRrm}.
		
		With the notations introduced there, another way of formulating Equation~\eqref{eq:asymptot_1term} is to write that as $\ps{\bm c,e_1}\to-\infty$ with $\ps{\bm c,_2}$ bounded below
		\begin{align*}
			e^{-2\ps{Q,c}}\expect{\left(e^{-\mu_1e^{\gamma \ps{\bm c,e_1}}M^{\gamma e_1}(\C)}-\mathcal{R}_{Id,s_1}\right)e^{-\mu_2e^{\gamma \ps{\bm c,e_2}}M^{\gamma e_2}(\C)}\prod_{k=1}^NV_{\alpha_k}(z_k)}=\mathcal{O}\left(e^{\xi\ps{\bm c, e_1}}\right),
		\end{align*}
		while the same reasoning shows that in this regime too $e^{-2\ps{Q,\bm c}}\expect{\bm{\mathrm R}_{\bm\alpha}^1(\bm c)}=\mathcal{O}\left(e^{\xi\ps{\bm c, e_1}}\right)$. 
		As a consequence in this asymptotic we can write that
		\begin{align*}
			E_{\bm\alpha}(\bm c)&=e^{-2\ps{Q,\bm c}}\mathds1_{\ps{\bm c,e_2}<0}\expect{\mathcal{R}_{Id,s_1,s_2,s_2s_1}\prod_{k=1}^NV_{\alpha_k}(z_k)-\bm{\mathrm R}_{\bm\alpha}^{1,2}(\bm c)}+\mathcal{O}\left(e^{\xi\ps{\bm c, e_1}}\right).
		\end{align*}
		Now from its explicit expression we readily see that the above expectation term is asymptotically equivalent to
		\[
		\max_{1\leq k \leq N}e^{-2\ps{Q,\bm c}}\expect{R_{s_1s_2}(\alpha_k)V_{\hat s_1\hat s_2\alpha_k}(z_k)\prod_{l\neq k}V_{\alpha_l}(z_l)}
		\]
		which scales like $\max_{1\leq k \leq N}e^{\ps{\bm s+ \hat s_1\hat s_2\alpha_k-\alpha_k,\bm c}}$. As a consequence this remainder term is integrable as soon as $\max_{1\leq k \leq N}\ps{\bm s+ \hat s_1\hat s_2\alpha_k-\alpha_k,\omega_1}>0$, which follows from the fact that $\ps{\bm s,\omega_1}>-\gamma$ while $\ps{\hat s_1\hat s_2\alpha_k-\alpha_k,\omega_1}>\gamma$ because of our assumptions on $\bm\alpha$ for it to belong to $\mathcal{A}_N$. This shows that $e^{\xi\norm{\bm c}}E_{\bm\alpha}(\bm c)$ remains bounded over the domain where $\ps{\bm c,e_2}$ is bounded from below. Of course the same applies for $\ps{\bm c,e_1}$ bounded below.
		
		Likewise if $\bm c\to\infty$ inside $\mathcal{C}_-$ the same reasoning shows that 
		\begin{align*}
			e^{-2\ps{Q,c}}\expect{\bm{\mathrm R}_{\bm\alpha}^1(\bm c)+\bm{\mathrm R}_{\bm\alpha}^2(\bm c)}=\mathcal{O}\left(e^{\xi\ps{\bm c, e_1}}\right)+\mathcal{O}\left(e^{\xi\ps{\bm c, e_2}}\right).
		\end{align*}
		Now we have also proved that as soon as $\ps{\bm s+\gamma e_i,\bm c}\leq -\xi\norm{\bm c}$,
		\[
		e^{-2\ps{Q,\bm c}}\expect{\prod_{k=1}^NV_{\alpha_k}(z_k)e^{-\sum_{i=1}^2\mu_ie^{\gamma\ps{\bm c,e_i}}M^{\gamma e_i}(\C)}-\bm{\mathrm{R}}_{\bm\alpha}^{1,2}(\bm c)}=\mathcal{O}\left(e^{-\xi\norm{\bm c}}\right).
		\]
		This hypothesis follows from the assumption that $\ps{\bm s,\omega_i}>-\gamma$ made in the definition of $\mathcal{A}_N$. Therefore we see that for $\xi>0$ small enough, $e^{\xi\norm{\bm c}}E_{\bm\alpha}(\bm c)$ remains bounded when $\bm c\to\infty$ inside $\mathcal{C}_-$.
		All in all, this shows that as desired, 
		\[
		\norm{E_{\bm\alpha}(\bm c)}\leq C^{-\xi\norm{\bm c}}.
		\] 
		This concludes for the proof of Lemma~\ref{lemma:anabis} and Proposition~\ref{prop:asymptot_many}.
	\end{proof}
	\begin{remark}
		We could have been more precise in the expansion of the expectation term in the regime when $\bm c$ diverges inside $\mathcal{C}_-$, based on the two-dimensional path decomposition for the Brownian motion $\X_{r+t}(z_k)-\X_r(z_k)$, and with a choice of $r_k$ such that
		\[
		(1+\eps)\frac{\norm{\ps{\nu,e_1}\ps{\bm c,e_2}-\ps{\nu,e_2}\ps{\bm c,e_1}}}{\ps{\nu_k,e_1}^2+\ps{\nu_k,e_1}\ps{\nu_k,e_2}+\ps{\nu_k,e_2}^2} \leq r_k\leq (1-\eps)\frac{\ps{\bm c,e_1}\vee\ps{\bm c,e_2}}{\ps{\nu_k,\rho}}\cdot
		\]
		This would allow us to extend the validity of Proposition~\ref{prop:asymptot_many} and Theorem~\ref{thm:analycity} to the set
		\begin{align*}
			\mathcal{A}'_N\coloneqq \Big\{(\alpha_1,\cdots,\alpha_N)&\in\left(Q+\mathcal{C}_-\right)^N,\\
			&\text{for }i=1,2,\quad \ps{\sum_{k=1}^N\alpha_k-2Q,\omega_i}>-\gamma\vee \max\limits_{1\leq k\leq N}\ps{\alpha_k-Q,\rho}\Big\}.
		\end{align*}
		However the analysis to be carried out is rather technical and would not have improved the range of values for $\gamma$ in the statement of Theorem~\ref{thm:main_result}, so we have omitted it to keep the document readable.
	\end{remark}

	\subsubsection{Concluding the proof of Theorem~\ref{thm:analycity}}
	Let us consider $\mathcal{O}$ an open subset of $\mathcal{A}_N$ and $K$ a compact subset of $\mathcal{A}_N$ that contains $\mathcal{O}$. We have seen above that the expression $E_{\bm\alpha}(\bm c)$ was holomorphic in a complex neighborhood $K'$ of $K$ for any fixed $\bm c$. Moreover thanks to Lemma~\ref{lemma:anabis} we know that the family of holomorphic functions $\left(E_{\bm\alpha}(\bm c)e^{\xi\norm{\bm c}}\right)_{\bm c}$ is uniformly bounded in $\bm c$ as soon as $\alpha\in \mathcal{O}$. It is also readily seen to depend continuously in $\bm c$. This implies that in a small complex neighborhood $\mathcal{O}'\subset K'$ of $\mathcal{O}$ the family of maps $\left(E_{\bm\alpha}(\bm c)e^{\xi\norm{\bm c}}\right)_{\bm c\in\R^2}$ remains bounded uniformly in $\bm c$. In particular we see that for $\bm\alpha\in\mathcal{O}'$ the map $\bm c\mapsto E_{\bm\alpha}(\bm c)$ is absolutely integrable over $\R^2$. 
	As a consequence if we take $\bm\alpha\in\mathcal{O}'$ and $\Gamma$ any closed and piecewise $C^1$ curve (say a triangle) surrounding $\ps{\alpha_1,e_1}$ on which $(\beta\omega_1+\ps{\alpha,e_2}\omega_2,\alpha_2,\cdots,\alpha_N)$ stays inside $\mathcal{O}'$ if $\beta\in\Gamma$ we see that 
	\begin{align*}
		&\int_{\Gamma}\left(\int_{\R^2}E_{(\beta\omega_1+\ps{\alpha,e_2}\omega_2,\alpha_2,\cdots,\alpha_N)}(\bm c)d\bm c\right)d\beta\\
		&=\int_{\R^2}\left(\int_{\Gamma}E_{(\beta\omega_1+\ps{\alpha,e_2}\omega_2,\alpha_2,\cdots,\alpha_N)}(\bm c)d\beta\right)d\bm c
	\end{align*}
	using Fubini-Tonelli theorem. 
	
	Now for fixed $\bm c$ the map $\beta\mapsto E_{(\beta\omega_1+\ps{\alpha,e_2}\omega_2,\alpha_2,\cdots,\alpha_N)}(\bm c)$ is holomorphic; therefore the integral over $\Gamma$ vanishes. This shows that for any such $\Gamma$
	\[
	\int_{\Gamma}\left(\int_{\R^2}E_{(\beta\omega_1+\ps{\alpha,e_2}\omega_2,\alpha_2,\cdots,\alpha_N)}(\bm c)d\bm c\right)d\beta=0.
	\]
	By Morera's theorem this implies that $\beta\mapsto\int_{\R^2}E_{(\beta\omega_1+\ps{\alpha,e_2}\omega_2,\alpha_2,\cdots,\alpha_N)}(\bm c)d\bm c$ is holomorphic in a neighborhood of $\ps{\alpha_1,e_1}$. The same reasoning being of course valid if we replace $\ps{\alpha_1,e_1}$ by any of the $\ps{\alpha_k,e_i}$ for $1\leq k\leq N$ and $i=1,2$ we see that the map considered in Lemma~\ref{lemma:analycity} is indeed holomorphic in a complex neighborhood of $\mathcal{A}_N$. This wraps up the proof of Theorem~\ref{thm:analycity}.

	
	
	
	
	
	\section{Four-point correlation functions and BPZ-type differential equations}\label{sec:BPZ}
	One of the achievements of~\cite{Toda_OPEWV} is to shed light on the presence of $W$-symmetry within the probabilistic framework for Toda CFTs proposed in~\cite{Toda_construction}. This symmetry manifests itself for instance via Ward identities, and one of its consequences is the derivation of a BPZ-type differential equation for some four-point correlation functions of the form, for $\bm\alpha$ satisfying the Seiberg bounds~\eqref{eq:seiberg},
	\[
	\ps{V_{\alpha}(z)V_{\alpha_0}(0)V_{\alpha_1}(1)V_{\alpha_\infty}(\infty)}.
	\]
	The assumptions made on these correlation functions are the presence of one fully degenerate field $V_\alpha$ ---that is a Vertex Operator $V_\alpha=V_{-\chi\omega_1}$ with $\chi\in\{\gamma,\frac2\gamma\}$--- as well as a semi-degenerate field, by which is meant a Vertex Operator $V_{\alpha_1}$ with weight $\alpha_1\coloneqq\kappa\omega_2$ where $\kappa$ is a real number.
	The latter assumption on the weight of the Vertex Operator $V_{\alpha_1}$ stems from the fact that, without this semi-degenerate field and unlike what happens in Liouville theory, in Toda CFTs $W$-symmetry does not sufficiently constrains the system to obtain a differential equation on four-point correlation functions with one fully degenerate field. This is the reason why, for the time being, we are only able to compute three-point correlation functions with a semi-degenerate field. 
	
	Based on the results of~\cite{Toda_OPEWV}, our goal in this section is to prove that such four-point correlation functions
	can actually be computed completely explicitly up to an explicit multiplicative constant, constant given by a three-point correlation function. 
	For this purpose, let us introduce for $\alpha=-\chi\omega_1$ and $\alpha_1=\kappa\omega_2$
	\begin{equation}
		\begin{split}
			A_i&\coloneqq\frac{\chi}2\ps{\alpha_0+\alpha_1+\alpha-Q,h_1}+\frac{\chi}2\ps{\alpha_\infty-Q,h_i}\\
			B_i&\coloneqq1+\frac{\chi}2\ps{\alpha_0-Q,h_1-h_{i+1}}.
		\end{split}
	\end{equation}
	To these quantities one can associate the ${}_3F_2$ hypergeometric function
	\begin{equation}
		\mathcal{H}_0(z)\coloneqq\hyper{A_1,A_2,A_3}{B_1,B_2}{z} = \sum_{n\in\mathbb{N}}\frac{(A_1)_n(A_2)_n(A_3)_n}{(B_1)_n(B_2)_n}\frac{z^n}{n!},
	\end{equation}
	where $(a)_n$ denotes the Pochhammer symbol $(a)_n=(a)(a+1)\cdots(a+n-1)$. Such a function is well-defined over the unit disc $\mathbb{D}$, and is a solution of an hypergeometric differential equation of order three:
	\begin{equation}\label{eq:hypergeometric}
		\begin{split}
			&\Big[z\left(A_1+z\partial_z\right)\left(A_2+z\partial_z\right)\left(A_3+z\partial_z\right)-\left(B_1-1+z\partial_z\right)\left(B_2-1+z\partial_z\right)z\partial_z\Big]\mathcal H=0.
		\end{split}
	\end{equation}
	Other complex-valued solutions of Equation~\eqref{eq:hypergeometric} are given by 
	\begin{align}
		\mathcal{H}_1(z)&\coloneqq z^{1-B_1}\hyper{1-B_1+A_1,1-B_1+A_2,1-B_1+A_3}{2-B_1,1-B_1+B_2}{z}\quad\text{and}\\
		\mathcal{H}_2(z)&\coloneqq z^{1-B_2}\hyper{1-B_2+A_1,1-B_2+A_2,1-B_2+A_3}{1-B_2+B_1,2-B_2}{z}
	\end{align}
	where we have set a branch cut for the logarithm to be the negative real axis $(-\infty,0]$. The hypergeometric functions admit an analytic continuation outside of the interval $(1,\infty)$ that we will work with in the sequel and denote in the same way. 
	
	With all these objects at hand we are in position to state the main result of the present section:
	\begin{theorem}\label{thm:BPZ}
		Assume that $\alpha=-\chi\omega_1$ with $\chi\in\{\gamma,\frac2\gamma\}$ and that $\alpha_1=\kappa\omega_2$ for $\kappa<q$.
		Further assume that $\alpha_0$, $\alpha_\infty$ are such that $(\alpha,\alpha_0,\alpha_1,\alpha_\infty)$ belongs to $\mathcal{A}_4$. Then 
		\begin{equation}\label{eq:BPZ}
			\begin{split}
				&\ps{V_{\alpha}(z)V_{\alpha_0}(0)V_{\alpha_1}(1)V_{\alpha_\infty}(\infty)}=\norm{z}^{\chi\ps{h_1,\alpha_0}}\norm{z-1}^{\frac{\chi\kappa}3}\mathcal H(z),\quad\text{where}\\
				&\mathcal H(z)=C_\gamma(\alpha_0+\alpha,\alpha_1,\alpha_\infty)\left(\norm{\mathcal H_0(z)}^2+\sum_{i=1}^2A_\gamma^{(i)}(\alpha,\alpha_0,\alpha_1,\alpha_\infty)\norm{\mathcal H_i(z)}^2\right).
			\end{split}
		\end{equation}
		The constants $A_\gamma^{(i)}(\alpha,\alpha_0,\alpha_1,\alpha_\infty)$, $i=1,2$, are given by
		\begin{align}
			A_\gamma^{(i)}(\alpha,\alpha_0,\alpha_1,\alpha_\infty)\coloneqq \frac{\prod_{j=1}^3l(A_j)l(B_i-A_j)}{l(B_1)l(B_2)}\frac{l(1+B_1+B_2-2B_i)}{l(B_i-1)}.
		\end{align}
	\end{theorem}
	The condition imposed on the form of the solution $\mathcal H$ is a consequence of the so-called \textit{crossing symmetry} assumption made in the physics literature and that states that both the \textit{$s$-channel} (which correspond to a basis of solutions around the point $z=0$) and \textit{$t$-channel decompositions} (a basis around $z=\infty$) of the solutions of Equation~\eqref{eq:hypergeometric} must be valid and consistent. For the time being the prefactor $C_\gamma(\alpha_0+\alpha,\alpha_1,\alpha_\infty)$ is given by a probabilistic expression, but will be evaluated later on thanks to Theorem~\ref{thm:main_result}. 
	
	
	\subsection{A BPZ equation for four-point correlation functions}
	Let us start by recalling from~\cite[Theorem 1.3]{Toda_OPEWV} that, under the assumptions that $\bm\alpha\coloneqq(\alpha,\alpha_0,\alpha_1,\alpha_\infty)$ satisfies the generalized Seiberg bounds~\eqref{eq:correl_1bis}, we have the equality
	\begin{equation}\label{eq:edp_to_extend}
		\mathcal D_0(z)\ps{V_{-\chi\omega_1}(z)V_{\alpha_1}(0)V_{\alpha_2}(1)V_{\alpha_3}(\infty)}=\frac{\bm{\mathcal W_{-1}^{(2)}}}{z-1}\ps{V_{-\chi\omega_1}(z)V_{\alpha_1}(0)V_{\alpha_2}(1)V_{\alpha_3}(\infty)}
	\end{equation}
	where on the left-hand side we have considered the differential operator 
	\begin{equation}
		\begin{split}
			&\mathcal D_0(z)\coloneqq \frac{8z^2(z-1)}{\chi^3}\partial_z^3+\frac{4}{\chi}z(3-5z)\partial^2_z\\
			&+\Big[\frac{2}{\chi}\left(\frac{4z^2-5z+2}{z-1}+2z(\Delta_3-\Delta)-2\Delta_1+2\Delta_2\frac z{z-1}\right)+\chi\frac{12z^2-15z+4}{z-1}\Big]\partial_z\\
			&+\frac{2}{\chi}\left((1-2z)\frac{\Delta_3-\Delta}{z-1}+(3z-2)\frac{\Delta_1}{z(z-1)}-(3z-1)\frac{\Delta_2}{(z-1)^2}\right)\\
			&+\frac{\chi}{3}\left((7-10z)\frac{\Delta_3-\Delta}{z-1}+(9z-6)\frac{\Delta_1}{z(z-1)}+(7-9z)\frac{\Delta_2}{(z-1)^2}\right)\\
			&+w+w_3+\frac{w_1}{z}-\frac{w_2}{(z-1)^2}\left(2z-1\right).
		\end{split}
	\end{equation}
	In order to derive a differential equation for four-point correlation functions we first need to extend the validity of the above equation to the whole range of values prescribed by Theorem~\ref{thm:analycity}, that is when $\bm\alpha\in\mathcal{A}_4$. 
	
	For this purpose we will consider a weak formulation of the above problem since correlation functions may not be differentiable if the weights no longer satisfy the generalized Seiberg bounds. However the last quantity involving a $W_{-1}$ descendant is not directly a differential operator applied to the correlation functions if the Vertex Operator $V_{\alpha_2}$ is not a semi-degenerate field. To remedy this issue we will be a bit astute and rely on the fact that we can actually express the $W_{-1}$ descendant as a differential operator but acting on a larger space. To be more specific we note that 
	\[
	\bm{\mathrm W}_{-1}(z,\alpha)=\left(q-2\ps{\alpha,\omega_2}\right)\ps{\alpha,e_1}\ps{\partial \Phi(z),\omega_1}-\left(q-2\ps{\alpha,\omega_1}\right)\ps{\alpha,e_2}\ps{\partial \Phi(z),\omega_2}
	\]
	so that we can formally write that
	\begin{align*}
		&\bm{\mathcal W}_{-1}\ps{V_{\alpha}(z)\V}=\Big[\left(q-2\ps{\alpha,\omega_2})\right)\partial_{x_1}-\left(q-2\ps{\alpha,\omega_1})\right)\partial_{x_2}\Big]\ps{V_{\ps{\alpha,e_1}\omega_1}(x_1)V_{\ps{\alpha,e_2}\omega_2}(x_2)\V}
	\end{align*}
	evaluated at $x_1=x_2=z$. However due to the fact that correlation functions are defined based on a regularization procedure this is no longer true and we need to be more precise. This is done by writing at the regularized level that
	\begin{align*}
		\bm{\mathcal W}_{-1}\ps{V_{\alpha,\eps}(z)\V_\eps}&=\hat g(z)^{\Delta_{\alpha}}\ps{\bm{\mathrm W}_{-1,\eps}(z,\alpha)e^{\ps{\alpha,\X^{\hat g}_\eps(z)+\bm c}-\frac12\expect{\ps{\alpha,\X^{\hat g}_\eps(z)}^2}}V_\eps}\\
		&=\hat g(z)^{\Delta_{\alpha}}\Big[\left(q-2\ps{\alpha,\omega_2})\right)\partial_{x_1}-\left(q-2\ps{\alpha,\omega_1})\right)\partial_{x_2}\Big]\vert_{x_1=x_2=z}\\
		&\ps{e^{\ps{\alpha,e_1}\ps{\X^{\hat g}_\eps(x_1)+\bm c,\omega_1}}e^{\ps{\alpha,e_2}\ps{\X^{\hat g}_\eps(x_2)+\bm c,\omega_2}}e^{-\frac12\expect{\ps{\alpha,\X^{\hat g}_\eps(z)}^2}}V_\eps}
	\end{align*}
	up to metric-dependent term that vanish thanks to the KPZ identity~\cite[Lemma 3.3]{Toda_OPEWV} in the same fashion as in the proof of ~\cite[Proposition 3.4]{Toda_OPEWV}.
	Now we note that $\expect{\ps{\X^{\hat g}_\eps(z),\omega_2}^2}$ is given by
	\[
	\expect{\left(\ps{\alpha,e_1}\ps{\X^{\hat g}_\eps(z),\omega_1}\right)^2}+\expect{\left(\ps{\alpha,e_2}\ps{\X^{\hat g}_\eps(z),\omega_2}\right)^2}+2\expect{\ps{\alpha,e_1}\ps{\alpha,e_2}\ps{\X^{\hat g}_\eps(z),\omega_1}\ps{\X^{\hat g}_\eps(z),\omega_2}}.
	\]
	Recollecting terms this shows that
	\begin{equation}\label{eq:W1_ext}
		\begin{split}
			\bm{\mathcal W}_{-1}\ps{V_{\alpha,\eps}(z)\V_\eps}&=\mathcal T\vert_{x_1=x_2=z}
			\left(\norm{x_1-x_2}^{\ps{\alpha,e_1}\ps{\alpha,e_2}\ps{\omega_1,\omega_2}}\ps{V_{\ps{\alpha,e_1}\omega_1,\eps}(x_1)V_{\ps{\alpha,e_2}\omega_2,\eps}(x_2)\V_\eps}\right)\\
			\text{with}\quad\mathcal T&\coloneqq \left(q-2\ps{\alpha,\omega_2})\right)\partial_{x_1}-\left(q-2\ps{\alpha,\omega_1})\right)\partial_{x_2}.
		\end{split}
	\end{equation}
	Note that straightforward computations show that the limit\\ $\lim\limits_{x_1=x_2=z}\norm{x_1-x_2}^{\ps{\alpha,e_1}\ps{\alpha,e_2}\ps{\omega_1,\omega_2}}\ps{V_{\ps{\alpha,e_1}\omega_1}(x_1)V_{\ps{\alpha,e_2}\omega_2}(x_2)\V}$ exists and is given by $\ps{V_{\alpha}(z)\V}$.
	
	We can now define the weak formulation of the above problem that will allow us to extend the range of validity of the above differential equation. Namely let us denote 
	\[
	F_{\bm\alpha}(z,x_1,x_2)\coloneqq \norm{x_1-x_2}^{\ps{\alpha,e_1}\ps{\alpha,e_2}\ps{\omega_1,\omega_2}}\ps{V_{\ps{\alpha,e_1}\omega_1}(x_1)V_{\ps{\alpha,e_2}\omega_2}(x_2)\V}
	\]
	and consider test functions $\phi:\C^3\to \R$ that are smooth, bounded and compactly supported, with $\phi(z,x_1,x_2)=0$ if $\norm{z-x_i}<\delta$ or $\norm{z}<\delta$ for some positive $\delta$.
	Then for any such function the quantity 
	\begin{align*}
		&\int_{\C^3}\mathcal D_0^*(z)\phi(z,x_1,x_2)F_{\bm\alpha}(z,x_1,x_2)d^2zd^2x_1d^2x_2
	\end{align*}
	with $\mathcal D_0^*(z)$ the adjoint operator of $\mathcal D_0(z)$ can analytically continued  over the whole $\mathcal A_4$. 
	Likewise we can extend analytically over $\mathcal A_4$ the quantity
	\begin{align*}
		&\int_{\C^3}\mathcal T\phi(z,x_1,x_2)F_{\bm\alpha}(z,x_1,x_2)d^2zd^2x_1d^2x_2.
	\end{align*}
	
	If we now consider $\phi=\phi_\eps$ to be a sequence of such test functions with $\phi_\eps(z,x_1,x_2)=0$ as soon as $\norm{x_1-1}>\eps$ or $\norm{x_2-1}>\eps$, we see that Equation~\eqref{eq:edp_to_extend} shows that as soon as $\bm\alpha$ satisfies the generalized Seiberg bounds, 
	\begin{equation}\label{eq:weak_edp}
		\lim\limits_{\eps\to0}\int_{\C^3}\left(\mathcal D_0^*(z)+\mathcal T\right)\phi_\eps(z,x_1,x_2)F_{\bm\alpha}(z,x_1,x_2)d^2zd^2x_1d^2x_2=0.
	\end{equation}
	Now the sequence of integrals is meromorphic in $\bm\alpha\in\mathcal A_4$ since the integral are absolutely convergent. Moreover the limit is also uniformly convergent in a complex neighborhood of $\mathcal A_4$ since we avoid the singularities when points may merge. This shows that for any $\bm\alpha\in\mathcal A_4$ Equation~\eqref{eq:weak_edp} remains valid too.
	
	Now it was shown in~\cite{Toda_OPEWV} that a consequence of this fact was that the four-point correlation being considered considered in Theorem~\ref{thm:BPZ} was a strong solution of the hypergeometric differential equation~\eqref{eq:hypergeometric}. This allows to extend the range of validity of this equation to the whole $\mathcal A_4$.
	
	
	\subsection{Proof of Theorem~\ref{thm:BPZ}}
	We have seen above that under the assumption that $\bm\alpha\in\mathcal A_4$, the four-point correlation functions considered in Theorem~\ref{thm:BPZ} are solutions of an hypergeometric differential equation of the third order. This completely determines them up to a global constant, thanks to the following:
	\begin{proposition}\label{prop:hyper_sol}
		Assume that for any two distinct elements $U, V$ in the set $\{0,B_1,B_2,A_1,A_2,A_3\}$, the quantity $U-V$ is a non-integer real number.
		Then real-valued solutions of the hypergeometric differential equation of the third order~\eqref{eq:hypergeometric} in $\C\setminus\lbrace0,1\}$ are of the form 
		\begin{equation}
			\mathcal H(z)=\mathcal H(0)\left(\norm{\mathcal{H}_0(z)}^2+\sum_{i=1}^2A_\gamma^{(i)}(\alpha,\alpha_0,\alpha_1,\alpha_\infty)\norm{\mathcal H_i(z)}^2\right).
		\end{equation}
	\end{proposition}
	As a consequence in order to prove that Theorem~\ref{thm:BPZ} does indeed hold we only need to check that $\mathcal{H}$ defined via $\ps{V_{-\chi\omega_1}(z)V_{\alpha_0}(0)V_{\kappa\omega_2}(1)V_{\alpha_\infty}(\infty)}=\norm{z}^{\chi\ps{h_1,\alpha_0}}\norm{z-1}^{\frac{\chi\kappa}3}\mathcal H(z)$ is real-valued, satisfies $\mathcal{H}(0)=C_\gamma(\alpha_0-\chi\omega_1,\kappa\omega_2,\alpha_\infty)$ and that the coefficients $A_i$, $1\leq i\leq3$ and $B_j$, $j=1,2$ satisfy the assumptions of Proposition~\ref{prop:hyper_sol}. 
	
	The first point is straightforward, while the second one follows from the probabilistic representation of the correlation functions, which allows to evaluate 
	\begin{equation*}
		\mathcal H(0)=C_\gamma(\alpha_0-\chi \omega_1,\alpha_1,\alpha_\infty).
	\end{equation*}
	Therefore the result holds true as soon as the coefficients $A$ and $B$ meet the requirements of Proposition~\ref{prop:hyper_sol}. 
	Because both $\mathcal{H}$ and the right-hand side in Equation~\eqref{eq:BPZ} depend analytically on the weights $\alpha$ via Theorem~\ref{thm:analycity}, the statement extends to the whole range of values of Theorem~\ref{thm:BPZ}.
	
	Before actually proving Proposition~\ref{prop:hyper_sol}, we start with the general description of the set of real-valued solutions of Equation~\eqref{eq:hypergeometric} on the subset $\C\setminus\{(-\infty,0]\cup[1,+\infty)\}$ of $\C\setminus\{0,1\}$.
	\begin{lemma}
		Real-valued solutions of Equation~\eqref{eq:hypergeometric} on $\C\setminus\{(-\infty,0]\cup[1,+\infty)\}$ are linear combinations of 
		\[
		\norm{F_i(z)}^2\text{ for }i=0,1,2;\quad\mathfrak{Re}\left(F_i(z)F_j(\bar z)\right)\text{ for } 0\leq i\neq j\leq 2;\quad\mathfrak{Im}\left(F_i(z)F_j(\bar z)\right)\text{ for } 0\leq i\neq j\leq 2.
		\] 
	\end{lemma} 
	\begin{proof}
		The functions proposed in the statement being real-valued solutions of Equation~\eqref{eq:hypergeometric}, it is enough to bound the dimension of the set of such functions by nine. Now the statement of~\cite[Theorem 1.3]{Toda_OPEWV} also implies that such a function $\mathcal{H}$ is actually real analytic; therefore it suffices to bound this dimension on some open subset of $\C\setminus\{0,1\}$.
		
		To do so we view $\mathcal{H}$ as a function of two real variables $x,y$ with $x$ (resp. $y$) being the real (resp. imaginary) part of $z$. Then, by taking the real and imaginary parts of Equation~\eqref{eq:hypergeometric}, $\mathcal{H}$ is a solution of
		\begin{align*}
			\partial_{xxx}\mathcal H-3\partial_{xyy}\mathcal H+\text{lower derivatives}=0\quad\quad\partial_{yyy}\mathcal H-3\partial_{xxy}\mathcal H+\text{lower derivatives}=0.
		\end{align*}
		The characteristic equation $\tau^3-3\tau\xi^2=0$ being of the hyperbolic type we see that $\mathcal H$, viewed as a function in the variables $x,y$, is a solution of a pair of hyperbolic partial differential equations of the third order. As a consequence it is completely determined in a complex neighborhood $\mathcal O$ of $-1$ by the data of $u_0(x)\coloneqq \mathcal{H}(x,0)$, $u_1(x)\coloneqq \partial_y\mathcal{H}(x,0)$ and $u_2(x)\coloneqq \partial_{yy}\mathcal{H}(x,0)$.
		Using the explicit expression of the differential operator~\eqref{eq:hypergeometric} these real valued functions are solutions of
		\begin{align*}
			\Big[x\left(A_1+x\partial\right)\left(A_2+x\partial\right)\left(A_3+x\partial\right)&-\left(B_1-1+x\partial\right)\left(B_2-1+x\partial\right)x\partial\Big]u_0\\
			&=3x^3(x-1)u_2'+x^2\big(x(A_1+A_2+A_3+3)-(B_1+B_2+1)\big)u_2,
		\end{align*}
		which corresponds to taking the real part of~\eqref{eq:hypergeometric}, and, by taking $\partial_y$ derivatives of the previous expression combined with the imaginary part of~\eqref{eq:hypergeometric},
		\begin{align*}
			\Big[x\left(A_1+x\partial\right)\left(A_2+x\partial\right)\left(A_3+x\partial\right)&-\left(B_1-1+x\partial\right)\left(B_2-1+x\partial\right)x\partial\Big]u_1\\
			&=3x^3(x-1)u_3'+x^2\left(x(A_1+A_2+A_3+3)-(B_1+B_2+1)\right)u_3,
		\end{align*}
		\begin{align*}
			\Big[x\left(A_1+x\partial\right)\left(A_2+x\partial\right)\left(A_3+x\partial\right)&-\left(B_1-1+x\partial\right)\left(B_2-1+x\partial\right)x\partial\Big]u_2\\
			&=3x^3(x-1)u_4'+x^2\left(x(A_1+A_2+A_3+3)-(B_1+B_2+1)\right)u_4,
		\end{align*}
		\begin{align*}
			\text{with}\quad u_3=3u_1^{(2)}&+2\frac{x(A_1+A_2+A_3+3)-(B_1+B_2+1)}{x(x-1)}u_1'\\
			&+\frac{x(1+A_1+A_2+A_3+A_1A_2+A_1A_3+A_2A_3)-B_1B_2}{x^2(x-1)}u_1, \\
			\text{and}\quad u_4=3u_2^{(2)}&+2\frac{x(A_1+A_2+A_3+3)-(B_1+B_2+1)}{x(x-1)}u_2'\\
			&+\frac{x(1+A_1+A_2+A_3+A_1A_2+A_1A_3+A_2A_3)-B_1B_2}{x^2(x-1)}u_2.
		\end{align*}  
		Therefore $u_1$ and $u_2$ both live in a three-dimensional space; similarly $u_0$ lives in a three-dimensional space determined by $u_2$. This implies that $\mathcal{H}$ lives in a set of dimension at most nine as expected.  
	\end{proof}
	
	We are now in position to address the proof of Proposition~\ref{prop:hyper_sol}.
	\begin{proof}[Proof of Proposition~\ref{prop:hyper_sol}]
		We rely on the fact that any solution $\mathcal H$, along with its derivatives, must be regular when crossing the real axis. This fact will imply that, among the solutions of Equation~\eqref{eq:hypergeometric} on $\C\setminus\{(-\infty,0]\cup[1,+\infty)\}$ the only ones that meet this requirement will be those of the form $\norm{F_i(z)}^2\text{ for }i=0,1,2$. In addition to ensure this condition one must further assume that $\mathcal H$ has the form prescribed by Proposition~\ref{prop:hyper_sol}.
		
		To see why, let us first consider what happens when a solution crosses the negative real axis $(-\infty,0)$. Because the hypergeometric functions ${}_3F_2$ are continuous over $\C\setminus(1,+\infty)$, the lack of continuity of the solutions across the negative real axis would come from the fractional power of $z$. Indeed, let us take $r$ to be some positive real number for which the hypergeometric functions evaluated at $z=-r$ are non-zero. Since we have defined the branch cut of the logarithm to be on $(-\infty,0)$, we can write that for $z_{\pm,\eps}\coloneqq re^{\pm i(\pi-\eps)}$, $F_0(z_{+,\eps})F_1(z_{+,\eps})-F_0(z_{-,\eps})F_1(z_{-,\eps})$ will converge as $\eps\rightarrow0$ to $\left(e^{i(1-B_1)\pi}-e^{-i(1-B_1)\pi}\right)$ up to a non-zero (real) multiplicative constant. Since we have assumed $B_1$ not to be an integer, this means that $\mathfrak{Im}\left(F_0(z)F_1(\bar z)\right)$ will not be continuous across $(-\infty,0)$. Similarly, when considering the partial derivatives of $(x,y)\mapsto F_0(x+iy)F_1(x-iy)$ we see that, expanding around $z=0$,
		\[
		\partial_y F_0(z)F_1(\bar z)=\bar z^{1-B_1}\partial_y\left(F_0(z)\bar z^{B_1-1} F_1 (\bar z)\right)+i(B_1-1)\frac1{\bar z}F_0(z) F_1 (\bar z)= i(B_1-1) \bar z^{-B_1}+o(z^{-B_1}).
		\]
		Because of the branch cut on the negative real axis we see again that for $\mathfrak{Re}\left(F_0(z)F_1(\bar z)\right)$ to be continuous we must assume $B_1$ to be an integer, which it is not. As a consequence for solutions of Equation~\eqref{eq:hypergeometric} to be continuous along with their derivatives we must rule out (under the assumptions that $B_1$, $B_2$, $B_1-B_2$ are not integers) those of the form $\mathfrak{Re}\left(F_i(z)F_j(\bar z)\right)\text{ for } 0\leq i\neq j\leq 2$ or $\mathfrak{Im}\left(F_i(z)F_j(\bar z)\right)\text{ for } 0\leq i\neq j\leq 2$. Since the three hypergeometric functions are linearly independent a linear combination of such solutions will not be continuous either. 
		
		The remaining solutions $\norm{F_i(z)}^2$ are continuous along with their derivatives across the negative real axis; therefore the next step is the investigation of continuity across $(1,\infty)$. To make this explicit let us consider an alternative basis of solutions defined around $z=\infty$. These are defined via the expressions for $i=1,2,3$ (and the convention that $A_{k}=A_{k\text{ mod }3}$):
		\begin{align}
			G_i(z)\coloneqq (-z)^{A_i}\hyper{A_i,1+A_i-B_1,1+A_i-B_2}{1+A_i-A_{i-1},1+A_i-A_{i+1}}{\frac1z}.
		\end{align}
		Again the hypergeometric functions can be continued to analytic functions over $\C\setminus[0,1]$.
		Both basis are related thanks to the remarkable equality (see \cite{Thomae} or \cite[Theorem 1]{Smith}) valid outside of the real axis:
		\begin{equation}
			\frac{\Gamma(A_1)\Gamma(A_2)\Gamma(A_3)}{\Gamma(B_1)\Gamma(B_2)}\hyper{A_1,A_2,A_3}{B_1,B_2}{z}=\sum_{i=1}^3\frac{\Gamma(A_i)\Gamma(A_{i+1}-A_i)\Gamma(A_{i-1}-A_i)}{\Gamma(B_1-A_i)\Gamma(B_2-A_i)} G_i(z).
		\end{equation}
		Whence the general solution $\mathcal H(z)=\lambda_0 \norm{F_0(z)}^2+\lambda_1\norm{F_1(z)}^2+\lambda_2\norm{F_2(z)}^2$ admits the expansion
		\[
		\mathcal H(z)=\sum_{i=0}^2 \alpha_i\norm{G_i(z)}^2+\beta_3\mathfrak{Re}\left(G_0(z)G_1(\bar z)\right)+\beta_0\mathfrak{Re}\left(G_1(z)G_2(\bar z)\right)+\beta_{1}\mathfrak{Re}\left(G_2(z)G_0(\bar z)\right)
		\]
		where the coefficients $\beta_{i}$ are given by $\frac{\Gamma(A_i-A_{i-1})\Gamma(A_i-A_{i+1})\Gamma(A_{i-1}-A_{i+1})\Gamma(A_{i+1}-A_{i-1})l(B_1)l(B_2)}{\Gamma(B_1-A_{i+1})\Gamma(B_2-A_{i+1})\Gamma(A_{i+1})\Gamma(B_1-A_{i-1})\Gamma(B_2-A_{i-1})\Gamma(A_{i-1})l(A_i)}\times$
		\begin{align*}
			&\lambda_0 \frac{\Gamma(B_1)\Gamma(1-B_1)\Gamma(B_2)\Gamma(1-B_2)}{\Gamma(A_i)\Gamma(1-A_i)}\\
			&+\lambda_1 \frac{\Gamma(B_1-1)\Gamma(2-B_1)\Gamma(B_2-B_1)\Gamma(1-B_1+B_2)} {\Gamma(B_1-A_i)\Gamma(1-B_1+A_i)}\frac{\prod_{j=1}^3l(A_j)l(B_1-A_j)}{l(B_1)l(B_2)}\frac{l(1-B_1+B_2)}{l(B_1-1)}\\
			&+\lambda_2 \frac{\Gamma(B_2-1)\Gamma(2-B_2)\Gamma(B_1-B_2)\Gamma(1-B_2+B_1)}{\Gamma(B_2-A_i)\Gamma(1-B_2+A_i)}\frac{\prod_{j=1}^3l(A_j)l(B_2-A_j)}{l(B_1)l(B_2)}\frac{l(1-B_2+B_1)}{l(B_2-1)}\cdot
		\end{align*}
		This can be further reduced (up to a global multiplicative factor) to the form 
		\begin{align*}
			&\lambda_0 \sin(\pi A_i)\sin\pi(B_1-B_2)-\tilde{\lambda_1} \sin\pi(B_1-A_i)\sin(\pi B_2)+\tilde{\lambda_2} \sin\pi(B_2-A_i)\sin(\pi B_1)
		\end{align*}
		with $\tilde{\lambda_i}\coloneqq \lambda_i A_\gamma^{(i)}(\alpha,\alpha_0,\alpha_1,\alpha_\infty)$, and where the identity $\Gamma(z)\Gamma(1-z)=\frac\pi{\sin(\pi z)}$, valid for $z\not\in\Z$, has been used.
		Since we have assumed the $A_j-A_l$ not to be integers for $j\neq l$ along the same lines as above continuity across the $(1,\infty)$ axis of (derivatives of) $\mathcal H$ implies that the $\beta_{i}$ are zero. After a little algebra this implies that either $\lambda_0=\tilde\lambda_1=\tilde{\lambda}_2$ or the factor in front of the $\lambda$ vanishes, which won't occur since it is assumed that none of the $A_i$ and $B_j-A_i$ are integers.  
		Whence the desired relations between the $\lambda$: this wraps up the proof of Proposition~\ref{prop:hyper_sol}.
	\end{proof}

	
	
	
	\section{Four-point correlation functions and Operator Product Expansions}\label{sec:OPE}
	We have proved in the previous section that certain four-point correlation functions in the $\mathfrak{sl}_3$ Toda theory can be expressed as a sum of hypergeometric functions, and to do so we relied on the fact that they are solutions of a BPZ-type differential equation. Our goal here is to provide an alternative way of deriving such an expansion, based on the probabilistic representation of these correlation functions. More precisely this expansion will be obtained using \textit{Operator Product Expansions} (OPEs hereafter), which are based on an asymptotic expansion of the correlation functions when two insertion points collide (here $z\to0$). This alternative expansion will allow to obtain shift equations for three-point correlation functions, which as we will see in Section~\ref{sec:toda_end}, fully characterize these correlation functions. 
	
	As a byproduct of these OPEs we provide a rigorous meaning to the reflection relation between Vertex Operators from Theorem~\ref{thm:refl} $V_\alpha=R_s(\alpha)V_{s\alpha}$ which we prove to hold for all elements $s$ of the Weyl group $W$ and within certain three-point correlation functions. The expression of the reflection coefficients is in agreement with predictions from the physics literature~\cite{reflection_simplylaced, refl_non_simply_laced, Fat_refl}. Namely, let us recall that Toda reflection coefficients are defined by setting for $\alpha\in\R^2$ and $s\in W$
	\begin{equation}
		\begin{split}
			R_s(\alpha)&=\epsilon(s)\frac{A\left(s(\alpha-Q)\right)}{A(\alpha-Q)},\quad\text{with}\\
			A(\alpha)&=\prod_{i=1}^r\left(\mu\pi l\left(\frac{\gamma^2}{2}\right)\right)^{\frac{\ps{\alpha,\omega_i}}\gamma}\prod_{e\in\Phi^+}\Gamma\left(1-\frac{\gamma}2\ps{\alpha,e}\right)\Gamma\left(1-\frac{1}\gamma\ps{\alpha,e}\right)\cdot
		\end{split}
	\end{equation}
	We introduce for fixed $\alpha_1,\alpha_\infty\in\R^2$ the set $\mathcal{U}(\alpha_1,\alpha_\infty)$ defined via
	\begin{equation}
		\mathcal{U}(\alpha_1,\alpha_\infty)\coloneqq\left\{\alpha_0\in\R^2, (\hat s\alpha_0,\alpha_1,\alpha_\infty)\in\mathcal{A}_3\text{ for some }s\in W\right\}.
	\end{equation}
	\begin{theorem}\label{thm:OPE}
		Given $\alpha_1,\alpha_\infty\in\R^2$, extend the function $C_\gamma(\alpha_0,\alpha_1,\alpha_\infty)$ over $\mathcal{U}(\alpha_1,\alpha_\infty)$ by setting
		\begin{equation}\label{eq:ext_3ptsbis}
			C_\gamma(\alpha_0,\alpha_1,\alpha_\infty)\coloneqq R_s(\alpha_0)C_\gamma(\hat s\alpha_0,\alpha_1,\alpha_\infty)\text{ where $s\in W$ is such that }s(\alpha_0-Q)\in\mathcal{C}_-.
		\end{equation}
		Then the map thus defined is analytic in a complex neighborhood of $\mathcal{U}(\alpha_1,\alpha_\infty)$. 
		
		Moreover under the same assumptions as in Theorem~\ref{thm:BPZ}, 
		\begin{equation}\label{eq:OPE}
			\begin{split}
				&\ps{V_{\alpha}(z)V_{\alpha_0}(0)V_{\alpha_1}(1)V_{\alpha_\infty}(\infty)}=\norm{z}^{\chi\ps{h_1,\alpha_0}}\norm{z-1}^{\frac{\chi\kappa}3}\mathcal H(z),\quad\text{where}\\
				&\mathcal H(z)=\sum_{i=0}^2B_\gamma^{(i)}(\alpha_0,\chi)C_\gamma(\alpha_0-\chi h_{i+1},\alpha_1,\alpha_\infty)\norm{\mathcal H_{i}(z)}^2
			\end{split}
		\end{equation}
		as soon as $\alpha_0-\chi h_i\in\mathcal{U}(\alpha_1,\alpha_\infty)$ for $1\leq i\leq 3$.
		The coefficients $B_\gamma^{(i)}(\alpha_0,\chi)$ admit the explicit expression
		\begin{equation}
			B_\gamma^{(i)}(\alpha_0,\chi)=\prod_{j=1}^{i}\left(\pi\mu l\left(\frac{\gamma^2}{2}\right)\right)^{\frac{\chi}{\gamma}}\left(\frac{\chi^2}{2}\right)^2\frac{ l(\frac\chi2\ps{\alpha_0-Q,h_{j}-h_{i+1}})}{l(1+\frac{\chi^2}2+\frac\chi2\ps{\alpha_0-Q,h_j-h_{i+1}})}\cdot
		\end{equation}
	\end{theorem}
	\subsection{Method of proof}
	Let us briefly recall the framework introduced in the proof of Theorem~\ref{thm:analycity}. We have seen there that the four-point correlation functions we consider could be defined using Equation~\eqref{eq:ana_F} via 
	\[
	\langle V_{\alpha_0}(0)V_{-\chi h_1}(z)V_{\alpha_1}(1)V_{\alpha_\infty}(\infty)\rangle=  \int_{\R^2}E_{\bm\alpha}(\bm c)d\bm c+\sum_{i=1}^2\int_{\R}E^i_{\bm\alpha}(c_i)dc_i+E^{1,2}_{\bm\alpha},
	\]
	where the quantities $E_{\bm\alpha}^{\cdot}(\bm c)$ are expectation terms. In order to highlight the dependency in $z$ of these quantities, let us introduce the notation
	\begin{equation}\label{eq:last_label?}
		\begin{split}
			\Phi_{\alpha_0,\alpha}(z;\bm c)&\coloneqq e^{\ps{\bm s,\bm c}}\expect{\exp\left(-\sum_{i=1}^2\mu_i e^{\ps{\gamma e_i,\bm c}}I_{\alpha_0,\alpha}^i(z)\right)},\quad\text{with}\\
			I_{\alpha_0,\alpha}^i(z)&\coloneqq \int_\C \frac{\norm{x-z}^{-\gamma\ps{\alpha,e_i}}}{\norm{x}^{\gamma\ps{\alpha_0,e_i}}}F_i(x)M^{\gamma e_i}(d^2x)\text{ and } F_i(x)\coloneqq\frac{\norm x_+^{\gamma\ps{\alpha+\alpha_0+\alpha_1+\alpha_\infty,e_i}}}{\norm{x-1}^{\gamma\ps{\alpha_1,e_i}}}
		\end{split}
	\end{equation}
	so that 
	\[
	E_{\bm\alpha}(\bm c)=\norm{z}^{\chi\ps{\alpha_0,h_1}}\norm{z-1}^{\chi\ps{h_1,\alpha_1}}\left(\Phi_{\alpha_0,\alpha}(z;\bm c)-\mathfrak R_{\alpha_0,\alpha}(z;\bm c)\right),
	\] where the remainder term is given using Equation~\eqref{eq:defRrm} by $$\mathfrak R_{\alpha_0,\alpha}(z;\bm c)\coloneqq\norm{z}^{-\chi\ps{\alpha_0,h_1}}\norm{z-1}^{-\chi\ps{h_1,\alpha_1}}e^{-2\ps{Q,\bm c}}\expect{\bm{\mathrm{R}}_{\bm\alpha}(\bm c)}.$$ 
	Likewise we put $E_{\bm\alpha}^i(c_1)$ under the form $$E_{\bm\alpha}^i(c_1)=  \norm{z}^{\chi\ps{\alpha_0,h_1}}\norm{z-1}^{\chi\ps{h_1,\alpha_1}}\left( \Phi^i_{\alpha_0,\alpha}(z;c_1)-\mathfrak R_{\alpha_0,\alpha}^i(z;c_1)\right)\quad\text{where}$$ 
	\begin{align*}
		&\norm{z}^{\chi\ps{\alpha_0,h_1}}\norm{z-1}^{\chi\ps{h_1,\alpha_1}}\Phi^1_{\alpha_0,\alpha}(z;\ps{\bm c,e_1})\coloneqq\\
		&\sum_{w:\{0,1,\infty\}\to\{Id,s_2\}}\frac{e^{-\ps{\bm s(w),\omega_2}\ps{\bm c,e_2}}}{\ps{\bm s(w),\omega_2}}e^{-2\ps{Q,\bm c}}\expect{V_{-\chi h_1}(z)\prod_{x\in\{0,1,\infty\}}R_{w(x)}(\alpha_x)V_{\hat w(x)\alpha_x}(x)e^{-\mu_1e^{\gamma\ps{\bm c,e_1}}M^{\gamma e_1}(\C)}}.
	\end{align*}
	Finally we set $\Phi^{1,2}_{\alpha_0,\alpha}(z)\coloneqq \norm{z}^{-\chi\ps{\alpha_0,h_1}}\norm{z-1}^{-\chi\ps{h_1,\alpha_1}}E_{\bm\alpha}^{1,2}$. Note that reflection terms corresponding to the fully degenerate field $V_{-\chi h_1}$ do not show up in the expression of $\bm{\mathrm{R}_\alpha}(\bm c)$ since $\ps{-\chi h_1-Q,e_i}<-\gamma$ for $i=1,2$. 
	With these notations at hand we can write down the four-point correlation functions considered in Theorems~\ref{thm:BPZ} and~\ref{thm:OPE} via
	\begin{align*}
		\mathcal{H}(z)=\int_{\R^2}\left(\Phi_{\alpha_0,\alpha}(z;\bm c)-\mathfrak R_{\alpha_0,\alpha}(z;\bm c)\right)d\bm c+\sum_{i=1}^2\int_\R \left(  \Phi^i_{\alpha_0,\alpha}(z;c_i)-\mathfrak R_{\alpha_0,\alpha}^i(z;c_i)\right)dc_i+\Phi^{1,2}_{\alpha_0,\alpha}(z).
	\end{align*}
	
	Now a consequence of Theorem~\ref{thm:BPZ} is that these four-point correlation functions have an explicit expansion as $z\to0$. For instance in the case where $2(1-B_2)<1$ we can write that $\mathcal{H}(z)$ is given around $z=0$ by
	\begin{align*}
		C_\gamma(\alpha_0-\chi h_1,\alpha_1,\alpha_\infty)\Big(1+A_\gamma^{(1)}(-\chi h_1,\alpha_0,\alpha_1,\alpha_\infty)&\norm{z}^{\chi\ps{Q-\alpha_0,e_1}}\\
		+&A_\gamma^{(2)}(-\chi h_1,\alpha_0,\alpha_1,\alpha_\infty)\norm{z}^{\chi\ps{Q-\alpha_0,\rho}}\Big)
	\end{align*}
	up to lower order terms. The three-point correlation functions that appear in this expansion are defined analogously to $\mathcal{H}(z)$. Namely we will write that 
	\begin{align*}
		C_\gamma(\alpha_0,\alpha_1,\alpha_\infty)=\int_{\R^2}\left(\Phi_{\alpha_0}(\bm c)-\mathfrak R_{\alpha_0}(\bm c)\right)d\bm c+\sum_{i=1}^2\int_\R \left(  \Phi^1_{\alpha_0}(c_1)-\mathfrak R_{\alpha_0}^i(c_1)\right)dc_i+\Phi^{1,2}_{\alpha_0}(\bm c)
	\end{align*}
	where the quantities that appear in this expression are defined like above.

	In order to prove Theorem~\ref{thm:OPE} we will study the asymptotic of terms of the form $\Phi^\cdot_{\alpha_0,\alpha}(z;\bm c)$ and $\mathfrak{R}^\cdot_{\alpha_0,\alpha}(z;\bm c)$ around $z=0$ and prove that we end up with an expansion of the correlation functions similar to that from Theorem~\ref{thm:BPZ} but with the coefficients given by the ones in Equation~\eqref{eq:OPE}. Theorem~\ref{thm:OPE} then follows by identifying these coefficients with the ones from Equation~\eqref{eq:BPZ}.
	
	To do so we first show that Equation~\eqref{eq:OPE} holds true under certain different set of assumptions on the weights $\bm\alpha$ depending on the values of $\chi$, $\ps{\alpha_0,e_1}$ and $\ps{\alpha_0,e_2}$, and then use the results of the previous Section~\ref{sec:BPZ} to infer that the extension defined by Equation~\eqref{eq:ext_3ptsbis} is analytic. We will then recover the whole range of values for $\bm\alpha$ prescribed by Theorem~\ref{thm:OPE} using analycity of the correlation functions (Theorem~\ref{thm:analycity}). Throughout the rest of this Section we assume (up to shifting $\bm c$ by $\sum_{i}\frac{\ln\mu_i}{\gamma}\omega_i$) that $\mu_i=1$ for $i=1,2$.

	\subsection{The case where $\chi=\gamma$, $\ps{\alpha_0,e_1}<\frac2\gamma$ and $\ps{\alpha_0,e_2}>\frac2\gamma$}\label{subsec:OPE_1}
	To start with we consider in this subsection the case where the weight $\alpha_0$ is such that where $\ps{\alpha_0,e_1}<\frac2\gamma$ and $\ps{\alpha_0,e_2}>\frac2\gamma$, with $\chi=\gamma$.  
	Our goal is then to prove that under these assumptions the following holds true:
	\begin{lemma}\label{lemma:OPE1}
		Assume that $\ps{\alpha_0,e_1}<\frac2\gamma$ and $\frac2\gamma<\ps{\alpha_0,e_2}<q$. Then  
		\begin{equation}\label{eq:OPE11}
			\begin{split}
				\mathcal{H}(z)=C_\gamma(&\alpha_0-\gamma h_1,\alpha_1,\alpha_\infty)\norm{\mathcal{H}_0(z)}^2+B^{(1)}_\gamma(\alpha_0,\gamma)C_\gamma(\alpha_0-\gamma h_2,\alpha_1,\alpha_\infty)\norm{\mathcal{H}_1(z)}^2\\
				&+B^{(2)}_\gamma(\alpha_0,\gamma)R_{s_2}(\alpha_0-\gamma h_3)C_\gamma(\hat s_2(\alpha_0-\gamma h_3),\alpha_1,\alpha_\infty)\norm{\mathcal{H}_2(z)}^2
			\end{split}
		\end{equation}
		as soon as $(\alpha_0,-\gamma h_1,\alpha_1,\alpha_\infty)\in\mathcal{A}_4$ with $\ps{\bm s,\omega_2}>0$.
	\end{lemma}
	\begin{proof}
		In order to prove this statement we will first study the asymptotics of the expression $\mathcal{H}(z)$ when $z\to0$ under some additional assumptions on the weight $\alpha_0$. Namely we wish to prove that there exists a positive $\eps$ such that if $\ps{\alpha_0,e_1}>\frac2\gamma-\eps$ and $\ps{\alpha_0,e_2}>q-\eps$ then as $z\to0$ 
		\begin{equation*}
			\begin{split}
				\mathcal{H}(z)&=C_\gamma(\alpha_0-\gamma h_1,\alpha_1,\alpha_\infty)+\norm{z}^{\gamma\ps{Q-\alpha_0,e_1}}B^{(1)}_\gamma(\alpha_0,\gamma)C_\gamma(\alpha_0-\gamma h_2,\alpha_1,\alpha_\infty)\\
				&+\norm{z}^{\gamma\ps{Q-\alpha_0,\rho}}B^{(2)}_\gamma(\alpha_0,\gamma)R_{s_2}(\alpha_0-\gamma h_3)C_\gamma(\hat s_2(\alpha_0-\gamma h_3),\alpha_1,\alpha_\infty)+l.o.t.
			\end{split}
		\end{equation*}
		if $\gamma<1$, while if  $1<\gamma<\sqrt 2$:
		\begin{equation*}
			\begin{split}
				\mathcal{H}(z)&=C_\gamma(\alpha_0-\gamma h_1,\alpha_1,\alpha_\infty)+zC+\bar z\bar C+\norm{z}^{\gamma\ps{Q-\alpha_0,e_1}}B^{(1)}_\gamma(\alpha_0,\gamma)C_\gamma(\alpha_0-\gamma h_2,\alpha_1,\alpha_\infty)\\
				&+\norm{z}^{\gamma\ps{Q-\alpha_0,\rho}}B^{(2)}_\gamma(\alpha_0,\gamma)R_{s_2}(\alpha_0-\gamma h_3)C_\gamma(\hat s_2(\alpha_0-\gamma h_3),\alpha_1,\alpha_\infty)+l.o.t.
			\end{split}
		\end{equation*}
		where $C$ is some complex constant. The linear term $Cz+\bar C\bar z$ comes from expanding $\mathcal{H}_0(z)$ around $z=0$: in the case where $\gamma<1$ this term is of lower order compared to $\norm{z}^{\gamma\ps{Q-\alpha_0,e_1}}$ and $\norm{z}^{\gamma\ps{Q-\alpha_0,\rho}}$. However this is no longer true when $1<\gamma<\sqrt 2$ since in that case we have $1<\gamma\ps{Q-\alpha_0,e_1}<\gamma\ps{Q-\alpha_0,\rho}<2$ and therefore this term must be taken into account in the expansion.
		
		\subsubsection{The first expectation term}
		So as to prove such an expansion for $\mathcal{H}$ we start by providing an expansion of the term $\Phi_{\alpha_0,-\gamma h_1}(z;\bm c)$ that enters the expression of $\mathcal{H}$ and defined via Equation~\eqref{eq:last_label?}. For this purpose around $z=0$ we can write $I_{\alpha,\alpha_0}^1(z)$ under the form
		\begin{equation}\label{eq:deltaI}
			I_{\alpha_0,-\gamma h_1}^1(z)= I_{\alpha_0-\gamma h_1}^1 +\delta I(z),\quad  \delta I(z)\coloneqq \int_\C \frac{\norm{x-z}^{\gamma^2}-\norm{x}^{\gamma^2}}{\norm{x}^{\gamma\ps{\alpha_0,e_1}}}F_1(x)M^{\gamma e_1}(d^2x).
		\end{equation}
		Note that the integrals $I_{\alpha_0-\gamma h_1}^i$ are well-defined since $\ps{\alpha_0-\gamma h_1-Q,e_i}\leq\ps{\alpha_0-Q,e_i}<0$; also note that since $\ps{h_1,e_2}=0$ we have $I_{\alpha_0,-\gamma h_1}^2(z)=I_{\alpha_0-\gamma h_1}^2$. This allows to write 
		\begin{align*}
			&\Phi_{\alpha_0,-\gamma h_1}(z;\bm c)= e^{\ps{\bm s,\bm c}}\E\left[e^{-e^{\gamma\ps{\bm c,e_1}}\delta I(z)}\prod_{i=1}^2e^{-e^{\gamma\ps{\bm c,e_i}}I_{\alpha_0-\gamma h_1}^i}\right] \\
			&=\Phi_{\alpha_0-\gamma h_1}(\bm c)+e^{\ps{\bm s,\bm c}}\E\left[\left(\exp\left(-e^{\gamma\ps{\bm c,e_{1}}}\delta I(z)\right)-1\right)\prod_{i=1}^2e^{-e^{\gamma\ps{\bm c,e_i}}I_{\alpha_0-\gamma h_1}^i}\right].
		\end{align*}
		We focus on the expectation term. To this end, we rewrite
		\begin{align*}
			&\exp\left(-e^{\gamma\ps{\bm c,e_{1}}}\delta I(z)\right)-1=-e^{\gamma\ps{\bm c,e_{1}}}\delta I(z)\int_0^1 \exp\left(-te^{\gamma\ps{\bm c,e_{1}}}\delta I(z)\right)dt\\
			&=-e^{\gamma\ps{\bm c,e_{1}}}\int_\C \frac{\norm{x-z}^{\gamma^2}-\norm{x}^{\gamma^2}}{\norm{x}^{\gamma\ps{\alpha_0,e_1}}}F_1(x)M^{\gamma e_1}(d^2x)\int_0^1 \exp\left(-te^{\gamma\ps{\bm c,e_{1}}}\delta I(z)\right)dt.
		\end{align*}
		Moreover by the Girsanov (Cameron-Martin) theorem~\ref{thm:girsanov} we can write that
		\begin{align*}
			&\E\left[M^{\gamma e_i}(d^2x)\exp\left(-te^{\gamma\ps{\bm c,e_{1}}}\delta I(z)\right)\prod_{i=1}^2e^{-e^{\gamma\ps{\bm c,e_i}}I_{\alpha_0-\gamma h_1}^i}\right]\\
			=&\E\left[\exp\left(-te^{\gamma\ps{\bm c,e_{1}}}\delta J(z,x)\right)\prod_{i=1}^2e^{-e^{\gamma\ps{\bm c,e_i}}J_{\alpha_0-\gamma h_1}^i(x)}\right]d^2x,\quad\text{with}\\
			\delta& J(z,x)=\int_\C \frac{\norm{y-z}^{\gamma^2}-\norm{y}^{\gamma^2}}{\norm{y}^{\gamma\ps{\alpha_0,e_1}}\norm{y-x}^{2\gamma^2}}\tilde F_1(y)M^{\gamma e_1}(d^2y),\quad \tilde F_i(y)\coloneqq F_i(y)\norm{y}_+^{\gamma^2\ps{e_1,e_i}},\quad\text{and}\\
			&J_{\alpha_0-\gamma h_1}^i(x)=\int_\C \frac{\norm{y-x}^{-\gamma^2\ps{e_1,e_i}}}{\norm{y}^{\gamma\ps{\alpha_0-\gamma h_1,e_i}}}\tilde F_i(y)M^{\gamma e_i}(d^2y).
		\end{align*}
		Therefore $\Phi_{\alpha_0,-\gamma h_1}(z;\bm c)-\Phi_{\alpha_0-\gamma h_1}(\bm c)$ is equal to
		\begin{align*}
			\delta\Phi(z)\coloneqq-e^{\ps{\bm s+\gamma e_1,\bm c}}\int_\C \frac{\norm{x-z}^{\gamma^2}-\norm{x}^{\gamma^2}}{\norm{x}^{\gamma\ps{\alpha_0,e_1}}}&F_1(x)\times\\
			&\E\left[\int_0^1e^{-te^{\gamma\ps{\bm c,e_{1}}}\delta J(z,x)}dt\prod_{i=1}^2e^{-e^{\gamma\ps{\bm c,e_i}}J_{\alpha_0-\gamma h_1}^i(x)}\right]d^2x.
		\end{align*}
		
		\textbf{The case $\gamma<1$.}
		If we assume that $\gamma<1$, then under the additional assumption that $\gamma+\frac1\gamma<\ps{\alpha_0,e_1}<\frac2\gamma$ we can make the change of variable $x\leftrightarrow zx$ in the integral to get
		\begin{align*}
			\delta\Phi(z)=-e^{\ps{\bm s+\gamma e_1,\bm c}}\norm{z}^{\gamma\ps{Q-\alpha_0,e_1}}&\int_\C \frac{\norm{x-1}^{\gamma^2}-\norm{x}^{\gamma^2}}{\norm{x}^{\gamma\ps{\alpha_0,e_1}}}\times\\
			&\E\left[F_1(zx)\int_0^1e^{-te^{\gamma\ps{\bm c,e_{1}}}\delta J(z,zx)}dt\prod_{i=1}^2e^{-e^{\gamma\ps{\bm c,e_i}}J_{\alpha_0-\gamma h_1}^i(zx)}\right]d^2x.
		\end{align*}
		As $z\to0$, $\delta J(z,zx)\to0$ while $J_{\alpha_0-\gamma h_1}^i(zx)\to I_{\alpha_0-\gamma h_2}^i$ since $\alpha_0-\gamma h_1+\gamma e_1=\alpha_0-\gamma h_2$. Therefore in the $z\to0$ limit
		\begin{align*}
			\delta\Phi(z)=\norm{z}^{\gamma\ps{Q-\alpha_0,e_1}}\int_\C \frac{\norm{x}^{\gamma^2}-\norm{x-1}^{\gamma^2}}{\norm{x}^{\gamma\ps{\alpha_0,e_1}}}d^2x \times e^{\ps{\bm s+\gamma e_1,\bm c}}\E\left[\prod_{i=1}^2e^{-e^{\gamma\ps{\bm c,e_i}}I_{\alpha_0-\gamma h_2}^i}\right]+l.o.t.
		\end{align*}
		The last term is nothing but $\Phi_{\alpha_0-\gamma h_2}$: we conclude that when $z\to 0$ $\delta\Phi(z)$ is asymptotically equivalent to 
		\begin{align}\label{eq:B1}
			B^{(1)}_\gamma(\alpha_0,\gamma)\norm{z}^{\gamma\ps{Q-\alpha_0,e_1}}\Phi_{\alpha_0-\gamma h_2}(\bm c),\quad B^{(1)}_\gamma(\alpha_0,\gamma)=\int_\C \frac{\norm{x}^{\gamma^2}-\norm{x-1}^{\gamma^2}}{\norm{x}^{\gamma\ps{\alpha_0,e_1}}}d^2x.
		\end{align}
		The factor $B^{(1)}_\gamma(\alpha_0,\gamma)$ is evaluated in Lemma~\ref{lemma:evaluate_B} below, whose assumptions are satisfied as soon as $\gamma+\frac1\gamma<\ps{\alpha_0,e_1}<q$. It is found there to be equal to the factor from Theorem~\ref{thm:OPE}, so that we have obtained the first two terms in the expansion. However in order to get the third term in the expansion we have to be more precise. 
		
		For this purpose, let us decompose $J^2_{\alpha_0-\gamma h_1}(zx)$ as
		\begin{align*}
			J_{\alpha_0-\gamma h_1,B_r}^2(zx)+J_{\alpha_0-\gamma h_1,\C_r}^2(zx),\quad\text{where}\quad J_{\alpha_0-\gamma h_1,B_r}^2(zx)\coloneqq\int_{B_r} \frac{\norm{y-zx}^{\gamma^2}}{\norm{y}^{\gamma\ps{\alpha_0,e_2}}}\tilde F_2(y)M^{\gamma e_2}(d^2y)
		\end{align*}
		and $B_r=B_{e^{-r}}(0)$ with $r=-(1+\eps)\ln\norm{zx}$ for some positive $\eps$ small enough. Then
		\begin{align*}
			\delta\Phi(z)=&B^{(1)}_\gamma(\alpha_0,\gamma)\norm{z}^{\gamma\ps{Q-\alpha_0,e_1}}\Phi_{\alpha_0-\gamma h_2}(\bm c)+\delta\Phi^1(z)+\mathfrak{R}(z),\quad\text{where}\\
			\delta\Phi^1(z)\coloneqq&-e^{\ps{\bm s+\gamma e_1,\bm c}}\norm{z}^{\gamma\ps{Q-\alpha_0,e_1}}\int_\C \frac{\norm{x-1}^{\gamma^2}-\norm{x}^{\gamma^2}}{\norm{x}^{\gamma\ps{\alpha_0,e_1}}}\times\\
			&\E\left[\left(e^{-e^{\gamma\ps{\bm c,e_2}}J_{\alpha_0-\gamma h_1,B_r}^2(zx)}-1\right)e^{-\sum_{i=1}^2e^{\gamma\ps{\bm c,e_i}}I^i_{\alpha_0-\gamma h_2}}\right]d^2x \quad\text{and}\\
			\mathfrak{R}(z)\coloneqq&-e^{\ps{\bm s+\gamma e_1,\bm c}}\int_\C \frac{\norm{x-z}^{\gamma^2}-\norm{x}^{\gamma^2}}{\norm{x}^{\gamma\ps{\alpha_0,e_1}}}\E\Big[e^{-e^{\gamma\ps{\bm c,e_2}}J_{\alpha_0-\gamma h_1,B_r}^2(x)}e^{-\sum_{i=1}^2e^{\gamma\ps{\bm c,e_i}}I^i_{\alpha_0-\gamma h_2}}\\
			&\left(F_1(x)\int_0^1e^{-e^{\gamma\ps{\bm c,e_{1}}}\left(t\delta J(z,x)+\delta J_{\alpha_0-\gamma h_1}^i(x)\right)}dte^{-e^{\gamma\ps{\bm c,e_2}} J_{\alpha_0-\gamma h_1,\C_r}^2(x)}-1\right)\Big]d^2x.
		\end{align*}
		where we have denoted 
		\begin{align*}
			\delta J_{\alpha_0-\gamma h_1}^i(zx)=\int_\C \frac{\norm{y-zx}^{-\gamma^2\ps{e_1,e_i}}-\norm{y}^{-\gamma^2\ps{e_1,e_i}}}{\norm{y}^{\gamma\ps{\alpha_0-\gamma h_1,e_i}}}\tilde F_i(y)M^{\gamma e_i}(d^2y),\quad \tilde F_i(y)=F_i(y)\norm{y}_+^{\gamma^2\ps{e_1,e_i}}.
		\end{align*}
		
		The second remainder term $\mathfrak{R}(z)$ is a lower order term in the asymptotic studied. To be more specific it is a $o\left(\norm{z}^{\gamma\ps{Q-\alpha_0,\rho}}\right)$ as soon as $\ps{\alpha_0,e_2}$ is close enough to $q$. To see why, we first show that terms of the form
		\begin{align*}
			&\int_\C \frac{\norm{x-z}^{\gamma^2}-\norm{x}^{\gamma^2}}{\norm{x}^{\gamma\ps{\alpha_0,e_1}}}\E\left[\delta J_{\alpha_0-\gamma h_1}^1(x)\prod_{i=1}^2e^{-e^{\gamma\ps{\bm c,e_i}}I_{\alpha_0-\gamma h_2}^i}\right]d^2x\quad\text{and}\\
			&\int_\C \frac{\norm{x-z}^{\gamma^2}-\norm{x}^{\gamma^2}}{\norm{x}^{\gamma\ps{\alpha_0,e_1}}}\E\left[\delta J(z,x)\prod_{i=1}^2e^{-e^{\gamma\ps{\bm c,e_i}}I_{\alpha_0-\gamma h_2}^i}\right]d^2x
		\end{align*}
		are $o\left(\norm{z}^{\gamma\ps{Q-\alpha_0,\rho}}\right)$. To this end let us write that
		\begin{align*}
			&\E\left[\delta J_{\alpha_0-\gamma h_1}^1(x)\prod_{i=1}^2e^{-e^{\gamma\ps{\bm c,e_i}}I_{\alpha_0-\gamma h_2}^i}\right]=\int_\C \frac{\norm{y-x}^{-2\gamma^2}-\norm{y}^{-2\gamma^2}}{\norm{y}^{\gamma\ps{\alpha_0-\gamma h_1,e_1}}}\E\left[\prod_{i=1}^2e^{-e^{\gamma\ps{\bm c,e_i}}\tilde I_{\alpha_0-\gamma h_2}^i(y)}\right]\tilde F_1(y)d^2y\\
			&\text{where }\tilde I_{\alpha_0-\gamma h_2}^i(y)=\int_\C \frac{\norm{w-y}^{-\gamma^2\ps{e_1,e_i}}}{\norm{w}^{\gamma\ps{\alpha_0-\gamma h_2,e_i}}}\tilde F_i(w)\norm{w}_+^{\gamma^2\ps{e_1,e_i}}M^{\gamma e_i}(d^2w).
		\end{align*}
		Here we need to distinguish between two cases according to the value of $\ps{\alpha_0,e_1}$. In the case where $\xi\coloneqq \ps{\alpha_0-\gamma h_2+\gamma e_1-Q,e_1}$ is negative, then the integrals $\tilde I_{\alpha_0-\gamma h_2}^i(0)$ do make sense. However if $\xi\geq 0$ then the integral $\tilde I_{\alpha_0-\gamma h_2}^1(0)$ is not well-defined. Still we can proceed along the same lines as in the proof of the fusion asymptotics in~\cite[Lemma 3.2]{Toda_OPEWV} to see that in that case the leading term in the $y\to0$ asymptotic will be given by
		\[
		\E\left[\prod_{i=1}^2e^{-e^{\gamma\ps{\bm c,e_i}} J_{\alpha_0-\gamma h_1}^i(y)}\right]=o\left(\norm{y}^{\frac{\xi^2} 4-\eta}\right)
		\]
		for any positive $\eta$. As a consequence we see that as $z\to0$, for some constant $C$:
		\begin{align*}
			&\E\left[\delta J_{\alpha_0-\gamma h_1}^1(zx)\prod_{i=1}^2e^{-e^{\gamma\ps{\bm c,e_i}}I_{\alpha_0-\gamma h_2}^i}\right]\leq C\int_\C \frac{\norm{y-zx}^{-2\gamma^2}-\norm{y}^{-2\gamma^2}}{\norm{y}^{\gamma\ps{\alpha_0-\gamma h_1,e_1}}}\norm{y}^{\frac{\xi^2}{4}-\eta}\tilde F_1(y)d^2y.
		\end{align*}
		Making the change of variable $y\leftrightarrow zx y$ in the same fashion as before shows that
		\begin{align*}
			\E\left[\delta J_{\alpha_0-\gamma h_1}^1(zx)\prod_{i=1}^2e^{-e^{\gamma\ps{\bm c,e_i}}I_{\alpha_0-\gamma h_2}^i}\right]&=o\left(\norm{zx}^{2-2\gamma^2-\gamma\ps{\alpha_0-\gamma h_1,e_1}+\frac{\xi^2}{4}\mathds 1_{\xi\geq0}-\eta}\right)\\
			&= o\left(\norm{zx}^{2-\gamma^2-\gamma\ps{\alpha_0,e_1}+\frac{\xi^2}{4}\mathds 1_{\xi\geq0}-\eta}\right).
		\end{align*}
		Therefore we end up with the bound
		\begin{align*}
			&\int_\C \frac{\norm{x-z}^{\gamma^2}-\norm{x}^{\gamma^2}}{\norm{x}^{\gamma\ps{\alpha_0,e_1}}}\E\left[\left(e^{-e^{\gamma\ps{\bm c,e_1}}\delta J_{\alpha_0-\gamma h_1}^1(x)}-1\right)\prod_{i=1}^2e^{-e^{\gamma\ps{\bm c,e_i}}I_{\alpha_0-\gamma h_2}^i}\right]d^2x\\
			&\leq C \norm{z}^{\gamma\ps{Q-\alpha_0,e_1}+\frac{\ps{\alpha_0-\gamma h_2-Q,e_1}^2}4}.
		\end{align*}
		This in particular implies that if we take $\ps{\alpha_0,e_2}$ arbitrarily close to $q$ in such a way that $\gamma\ps{Q-\alpha_0,e_1}+\frac{\ps{\alpha_0-\gamma h_2-Q,e_1}^2}4>\gamma\ps{Q-\alpha_0,\rho}$ (this is possible since $\ps{Q-\alpha_0,\rho}=\ps{Q-\alpha_0,e_1}+\ps{Q-\alpha_0,e_2}$) then this term becomes a $o\left(\norm{z}^{\gamma\ps{Q-\alpha_0,\rho}}\right)$, hence a lower order term in the asymptotic studied. The same applies for
		\[
		\int_\C \frac{\norm{x-z}^{\gamma^2}-\norm{x}^{\gamma^2}}{\norm{x}^{\gamma\ps{\alpha_0,e_1}}}\E\left[\left(\int_0^1e^{-te^{\gamma\ps{\bm c,e_{1}}}\delta J(z,x)}dt-1\right)\prod_{i=1}^2e^{-e^{\gamma\ps{\bm c,e_i}}I_{\alpha_0-\gamma h_2}^i}\right]d^2x,
		\]
		which is also seen to be a $o\left(\norm{z}^{\gamma\ps{Q-\alpha_0,\rho}}\right)$. 
		Likewise, the fusion asymptotics from Equation~\eqref{eq:fusion} imply that
		\begin{align*}
			&\E\Big[V_{\gamma e_2}(x)e^{-\sum_{i=1}^2e^{\gamma\ps{\bm c,e_i}}I^i_{\alpha_0-\gamma h_2}}\Big]=o\left(\norm{x}^{\frac{\left(\ps{\alpha_0,e_2}-\frac2\gamma\right)^2}4-\eta}\right)\quad\text{so that}\\
			&\E\Big[\delta J_{\alpha_0-\gamma h_1,\C_r}^2(zx)e^{-\sum_{i=1}^2e^{\gamma\ps{\bm c,e_i}}I^i_{\alpha_0-\gamma h_2}}\Big]=o\left(\norm{zx}^{\gamma ^2}e^{-r\left(2-\gamma\ps{\alpha_0,e_2}+\frac{\left(\ps{\alpha_0,e_2}-\frac2\gamma\right)^2}4-\eta\right)}\right).
		\end{align*}
		As a consequence the associated remainder term is at most a $$o\left(\norm{z}^{\gamma\ps{Q-\alpha_0,e_1}+\gamma^2+(1+\eps)\left(2-\gamma\ps{\alpha_0,e_2}+\frac{\left(\ps{\alpha_0,e_2}-\frac2\gamma\right)^2}4-\eta\right)}\right).$$ 
		This exponent can be rewritten more conveniently as $$\gamma\ps{Q-\alpha_0,\rho}+\eps\left(2-\gamma\ps{\alpha_0,e_2}\right)+(1+\eps)\left(\frac{\left(\ps{\alpha_0,e_2}-\frac2\gamma\right)^2}4-\eta\right)$$ which is strictly larger than $\gamma\ps{Q-\alpha_0,\rho}$ provided that $\eps$ is chosen small enough. Eventually collecting up terms we see that $\mathfrak{R}(z)=o\left(\norm{z}^{\gamma\ps{Q-\alpha_0,\rho}}\right)$
		as expected.
		
		Let us now turn to $\delta\Phi^1(z)$. We wish to prove that
		\begin{align}\label{eq:reste_B2}
			\delta\Phi^1(z)= \norm{z}^{\gamma\ps{Q-\alpha_0,\rho}}B^{(1)}(\hat s_2\alpha_0,\gamma)R_{s_2}(\alpha_0)\Phi_{\hat s_2\alpha_0-\gamma h_2}(\bm c)+l.o.t.
		\end{align}
		where recall that $\hat s_2\alpha_0=\alpha_0+\ps{Q-\alpha_0,e_2}e_2$.
		For this we use the radial-angular decomposition as stated in Equations~\eqref{eq:rad_ang_dec} and~\eqref{equ:radial_angular}: inside $B_r$ we have 
		\begin{align*}
			\norm{y}^{\ps{\alpha_0-\gamma h_1,e_2}}\norm{y-zx}^{\gamma^2}M^{\gamma e_2}(d^2y)=e^{\ps{\X_r(0)+(\alpha_0-Q)r,\gamma e_2}}\norm{e^{-t+i\theta}-zx}^{\gamma^2}e^{\gamma B^\nu_t}M_Y^{\gamma e_2}(dt+r,d\theta)
		\end{align*}
		where we have set $\nu=\ps{\alpha_0,e_2}$ and $B^\nu$ is a one-dimensional Brownian motion with drift $\nu$. Recalling that $r=-(1+\eps)\ln\norm{zx}$ and $\ps{h_2,e_2}=1$, we can thus rewrite $J^2_{\alpha_0-\gamma h_1,B_r}(zx)$ as
		\begin{align*}
			&e^{\gamma\ps{\X_r(0)+(\alpha_0-\frac{\gamma}{1+\eps} h_2-Q)r,e_2}}\int_{0}^{+\infty}e^{\gamma B^\nu_t}\int_0^{2\pi} \norm{1-\norm{zx}^{\eps}e^{-t+i\theta}}^{\gamma ^2}\tilde F_2(e^{-t-r+i\theta})M_Y^{\gamma e_2}(dt+r,d\theta)\\
			=&e^{\gamma\ps{\X_r(0)+(\alpha_0-\frac{\gamma}{1+\eps} h_2-Q)r+\M,e_2}}\int_{0}^{+\infty}e^{\gamma \mathcal B^\nu_t}\int_0^{2\pi} \norm{1-\norm{zx}^{\eps}e^{-t+i\theta}}^{\gamma ^2}\tilde F_2(e^{-t-r+i\theta})M_Y^{\gamma e_2}(dt+r,d\theta)
		\end{align*}
		where $B^\nu$ has been decomposed as $B^\nu=\mathcal B^\nu+\M$ using the (one-dimensional) Brownian path decomposition in the form of Equation~\eqref{eq:williams}. As a consequence we can follow the same lines as in the proof of Lemma~\ref{lemma:anabis} to see that
		\begin{align*}
			&\delta\Phi^1(z)=-e^{\ps{\bm s+\gamma e_1,\bm c}}\norm{z}^{\gamma\ps{Q-\alpha_0,e_1}}\int_\C \frac{\norm{x-1}^{\gamma^2}-\norm{x}^{\gamma^2}}{\norm{x}^{\gamma\ps{\alpha_0,e_1}}}d^2x   \times\\
			&\E\left[\int_{\ps{\bm c+\lambda_r-\frac{\gamma h_2}{1+\eps}r,e_2}}^{+\infty}(-\nu)e^{\nu\ps{\M-\lambda_r-\bm c -\frac{\gamma h_2}{1+\eps}r,e_2}}\mathrm{d}\ps{\M,e_2}\left(e^{-e^{\gamma\ps{\M,e_2}}J_r}-1\right)e^{-\sum_{i=1}^2e^{\gamma\ps{\bm c,e_i}}I^i_{\alpha_0-\gamma h_2}}\right]\\
			&=-e^{\ps{\bm s+\gamma e_1+\hat s_2\alpha_0-\alpha_0,\bm c}}\norm{z}^{\gamma\ps{Q-\alpha_0,\rho}}\int_\C \frac{\norm{x-1}^{\gamma^2}-\norm{x}^{\gamma^2}}{\norm{x}^{\gamma\ps{\alpha_0,e_1}+\gamma\ps{\alpha_0-Q,e_2}}}d^2x   \times\\
			&\E\left[e^{-\nu\ps{\lambda_r,e_2}}\int_{\ps{\bm c+\lambda_r-\frac{\gamma h_2}{1+\eps}r,e_2}}^{+\infty}(-\nu)e^{\nu\ps{\M,e_2}}\mathrm{d}\ps{\M,e_2}\left(e^{-e^{\gamma\ps{\M,e_2}}J_r}-1\right)e^{-\sum_{i=1}^2e^{\gamma\ps{\bm c,e_i}}I^i_{\alpha_0-\gamma h_2}}\right]
		\end{align*}
		where we have set $\lambda_r\coloneqq \X_r(0)+(\alpha_0-Q)r$ and \[
		J_r=\int_{0}^{+\infty}e^{\gamma \mathcal B^\nu_t}\int_0^{2\pi} \norm{1-\norm{zx}^{\eps}e^{-t+i\theta}}^{\gamma ^2}\tilde F_2(e^{-t-r+i\theta})M^{\gamma e_2}(dt+r,d\theta)
		\] with $\mathcal{B}$ started from $\ps{\bm c+\lambda_r-\frac{\gamma h_2}{1+\eps}r,e_2}-\ps{\M,e_2}$. The exponential term $e^{-\nu \ps{\lambda_r,e_2}}$ is a Girsanov transform whose effect is to shift the law of $\X$ by $G_r(\cdot,0)\ps{Q-\alpha_0,e_2}e_2$, which as $r\to 0$ converges towards $G_r(\cdot,0)\ps{Q-\alpha_0,e_2}e_2=G(\cdot,0)(\hat s_2\alpha_0-\alpha_0)$.  Therefore proceeding in the same way as in the proof of Lemma~\ref{lemma:anabis} we see that 
		\begin{align*}
			&\delta\Phi^1(z)=\norm{z}^{\gamma\ps{Q-\alpha_0,e_1}}\int_\C \frac{\norm{x}^{\gamma^2}-\norm{x-1}^{\gamma^2}}{\norm{x}^{\gamma\ps{\hat s_2\alpha_0,e_1}}}d^2x R_{s_2}(\alpha_0)e^{\ps{\bm s+\gamma e_1,\bm c}}\E\left[e^{-\sum_{i=1}^2e^{\gamma\ps{\bm c,e_i}}I^i_{\hat s_2\alpha_0-\gamma h_2}}\right]+l.o.t.
		\end{align*}
		which is nothing but Equation~\eqref{eq:reste_B2}, where the integral over $\C$ is well-defined at least for $\ps{\alpha_0-Q,e_2}$ close enough to $0$ and evaluated using Lemma~\ref{lemma:evaluate_B} below, thanks to which it is found to be equal to $B^{(1)}_\gamma(\hat s_2\alpha_0,\gamma)$.
		
		Now from the explicit expressions of the quantities $B^{(i)}$ and $R_{s_2}$ one can deduce that
		\[
		B^{(1)}_\gamma(\hat s_2\alpha_0,\gamma)R_{s_2}(\alpha_0)=B^{(2)}_\gamma(\alpha_0,\gamma)R_{s_2}(\alpha_0-\gamma h_3).
		\]
		As a consequence we can infer the following expansion for $\Phi_{\alpha_0,-\gamma h_1}(z;\bm c)$ as $z\to 0$:
		\begin{equation}\label{eq:exp_Phi}
			\begin{split}
				\Phi_{\alpha_0,-\gamma h_1}(z;\bm c)=&\Phi_{\alpha_0-\gamma h_1}(\bm c)+\norm{z}^{2(1-B_1)}B^{(1)}_\gamma(\alpha_0,\gamma)\Phi_{\alpha_0-\gamma h_2}(\bm c)\\
				&+\norm{z}^{2(1-B_2)}B^{(2)}_\gamma(\alpha_0,\gamma)R_{s_2}(\alpha_0-\gamma h_3)\Phi_{\hat s_2(\alpha_0-\gamma h_3)}(\bm c)+l.o.t.
			\end{split}
		\end{equation}
		

		\textbf{The case $\gamma>1$.}
		We proceed in the same way for $\gamma >1$. Our assumptions on the weight $\alpha_0$ are then that $\gamma<\ps{\alpha_0,e_1}<\frac2\gamma$ and $q-\eps<\ps{\alpha_0,e_2}<q$. The only difference with the above case is that we write $\delta\Phi(z)$ as
		\begin{align*}
			\frac{\gamma^2}{2}e^{\ps{\bm s+\gamma e_1,\bm c}}\int_\C \frac{\frac{z}{x}+\frac{\bar z}{\bar x}}{\norm{x}^{\gamma\ps{\alpha_0-\gamma h_1,e_1}}}&F_1(x)\E\left[\int_0^1e^{-te^{\gamma\ps{\bm c,e_{1}}}\delta J(z,x)}dt\prod_{i=1}^2e^{-e^{\gamma\ps{\bm c,e_i}}J_{\alpha_0-\gamma h_1}^i(x)}\right]d^2x\\
			-\norm{z}^{\gamma\ps{Q-\alpha_0,e_1}}e^{\ps{\bm s+\gamma e_1,\bm c}}\int_\C &\frac{\norm{1-\frac 1x}^{\gamma^2}-1-\frac{\gamma^2}{2}\left(\frac{1}{x}+\frac{1}{\bar x}\right)}{\norm{x}^{\gamma\ps{\alpha_0-\gamma h_1,e_1}}}F_1(zx)\times\\
			&\E\left[\int_0^1e^{-te^{\gamma\ps{\bm c,e_{1}}}\delta J(z,zx)}dt\prod_{i=1}^2e^{-e^{\gamma\ps{\bm c,e_i}}J_{\alpha_0-\gamma h_1}^i(zx)}\right]d^2x.
		\end{align*}
		Then we can proceed in the very same fashion as above to see that the term on the last two lines is equal to  
		\[
		\norm{z}^{2(1-B_1)}B^{(1)}_\gamma(\alpha_0,\gamma)\Phi_{\alpha_0-\gamma h_2}(\bm c)+\norm{z}^{2(1-B_2)}B^{(2)}_\gamma(\alpha_0,\gamma)R_{s_2}(\alpha_0-\gamma h_3)\Phi_{\hat s_2(\alpha_0-\gamma h_3)}(\bm c)+l.o.t.
		\]
		The only difference lies in the evaluation of the integrals that arise, but they are still found to be given by $B^{(1)}_\gamma(\alpha_0,\gamma)$ in agreement with Lemma~\ref{lemma:evaluate_B}.
		
		This means that all we have to prove is that the term on the first line is of the form $Cz+\bar C\bar z +o\left(\norm{z}^{\gamma\ps{Q-\alpha_0,\rho}}\right)$. To see why this is indeed the case note that
		\begin{align*}
			C\coloneqq e^{\ps{\bm s+\gamma e_1,\bm c}}\int_\C \frac{F_1(x)}{x\norm{x}^{\gamma\ps{\alpha_0-\gamma h_1,e_1}}}&\E\left[\prod_{i=1}^2e^{-e^{\gamma\ps{\bm c,e_i}}J_{\alpha_0-\gamma h_1}^i(x)}\right]d^2x
		\end{align*}
		is well-defined since the singularity around $x=0$ is integrable thanks to the fact that $\ps{\alpha_0,e_1}<\frac2\gamma<\gamma+\frac1\gamma$.
		As a consequence the corresponding term in the expansion of $\delta\Phi(z)$ is given by 
		\begin{align*}
			&\frac{\gamma^2}2\left(Cz+\bar C\bar z\right)\\
			&+e^{\ps{\bm s+\gamma e_1,\bm c}}\int_\C \frac{\frac{\gamma^2}{2}\left(\frac{z}{x}+\frac{\bar z}{\bar x}\right)}{\norm{x}^{\gamma\ps{\alpha_0-\gamma h_1,e_1}}}F_1(x)\E\left[\int_0^1\left(e^{-te^{\gamma\ps{\bm c,e_{1}}}\delta J(z,x)}-1\right)dt\prod_{i=1}^2e^{-e^{\gamma\ps{\bm c,e_i}}J_{\alpha_0-\gamma h_1}^i(x)}\right]d^2x.
		\end{align*}
		To conclude it remains to check that 
		\begin{align*}
			\int_\C \frac{\frac{\gamma^2}{2}\left(\frac{z}{x}+\frac{\bar z}{\bar x}\right)}{\norm{x}^{\gamma\ps{\alpha_0-\gamma h_1,e_1}}}F_1(x)\E\left[\delta J(z,x)\prod_{i=1}^2e^{-e^{\gamma\ps{\bm c,e_i}}J_{\alpha_0-\gamma h_1}^i(x)}\right]d^2x
		\end{align*}
		is a lower order term, which can be seen via the same reasoning:
		\begin{align*}
			\delta J(z,x)\sim \int_\C\frac{\frac{\gamma^2}{2}\left(\frac{z}{y}+\frac{\bar z}{\bar y}\right)}{\norm{y}^{\gamma\ps{\alpha_0-\gamma h_1,e_1}}\norm{x-y}^{2\gamma^2}}\tilde F_1(y)M^{\gamma e_1}(d^2y)
		\end{align*}
		with the integral
		\begin{align*}
			\int_{\D\times\D}\frac{1}{x\norm{x}^{\gamma\ps{\alpha_0-\gamma h_1,e_1}}y\norm{y}^{\gamma\ps{\alpha_0-\gamma h_1,e_1}}\norm{x-y}^{2\gamma^2}}d^2y
		\end{align*}
		absolutely integrable. This implies that this term is at most of order $z^2$, hence is a $o\left(\norm{z}^{\gamma\ps{Q-\alpha_0,\rho}}\right)$ as soon as $\ps{\alpha_0,e_2}$ is close enough to $q$ since $\gamma^2<2$.
		This shows that under the assumptions that $\gamma >1$ with $\gamma<\ps{\alpha_0,e_1}<\frac2\gamma$ and $q-\eps<\ps{\alpha_0,e_2}<q$ then
		\begin{equation}\label{eq:exp_Phibis}
			\begin{split}
				\Phi_{\alpha_0,-\gamma h_1}(z;\bm c)=&\Phi_{\alpha_0-\gamma h_1}(\bm c)+Cz+\bar C\bar z+\norm{z}^{2(1-B_1)}B^{(1)}_\gamma(\alpha_0,\gamma)\Phi_{\alpha_0-\gamma h_2}(\bm c)\\
				&+\norm{z}^{2(1-B_2)}B^{(2)}_\gamma(\alpha_0,\gamma)R_{s_2}(\alpha_0-\gamma h_3)\Phi_{\hat s_2(\alpha_0-\gamma h_3)}(\bm c)+l.o.t.
			\end{split}
		\end{equation}
		
		\subsubsection{The reflection terms}\label{sub:refl1}
		Having described the asymptotic expansion of $\Phi_{\alpha_0,-\gamma h_1}(z;\bm c)$, we now turn to the other terms in the expression of $\mathcal{H}$, where recall that
		\begin{align*}
			&\mathcal{H}(z)=\int_{\R^2}\left(\Phi_{\alpha_0,\alpha}(z;\bm c)-\mathfrak R_{\alpha_0,\alpha}(z;\bm c)\right)d\bm c+\sum_{i=1}^2\int_\R \left(  \Phi^i_{\alpha_0,\alpha}(z;c_i)-\mathfrak R_{\alpha_0,\alpha}^i(z;c_i)\right)dc_i+\Phi^{1,2}_{\alpha_0,\alpha}(z),\\
			&\mathfrak R_{\alpha_0,\alpha}(z;\bm c)=\norm{z}^{-\gamma\ps{\alpha_0,h_1}}\norm{z-1}^{-\gamma\ps{h_1,\alpha_1}}e^{-2\ps{Q,\bm c}}\expect{\bm{\mathrm{R}}_{\bm\alpha}(\bm c)}
		\end{align*} 
		with $\bm{\mathrm{R}}_{\bm\alpha}(\bm c)=\bm{\mathrm{R}}^1_{\bm\alpha}(\bm c)+\bm{\mathrm{R}}^2_{\bm\alpha}(\bm c)+\bm{\mathrm{R}}^{1,2}_{\bm\alpha}(\bm c)$ as defined in Equation~\eqref{eq:defRrm}. Our goal is to show that they satisfy the same type of expansion too: this would in turn allow to derive an asymptotic expansion of $\mathcal{H}(z)$ and thus provide the equality stated in Theorem~\ref{thm:OPE}.
		
		To start with, let us focus on the term $\expect{\bm{\mathrm{R}}_{\bm\alpha}^2(\bm c)}$, which is given by
		\begin{align*}
			&\mathds 1_{\ps{\bm c,e_1}<0}e^{-\mu_{2}e^{\gamma\ps{\bm c,e_2}}M^{\gamma e_{2}}(\C)}\left(V_{\alpha_0}(0)+R_{s_1}(\alpha_0)V_{\hat s_1\alpha_0}(0)\right)V_{-\gamma h_1}(z)\mathcal{R}_{Id,s_1}V_{\alpha_1}(1)V_{\alpha_\infty}(\infty)\\
			&-\mathds 1_{\max\ps{\bm c,e_i}<0}\left(V_{\alpha_0}(0)+R_{s_2}(\alpha_0)V_{\hat s_2\alpha_0}(0)\right)V_{-\gamma h_1}(z)\mathcal{R}_{Id,s_1,s_2,s_2s_1}V_{\alpha_1}(1)V_{\alpha_\infty}(\infty)\\
			&-\mathds 1_{\max\ps{\bm c,e_i}<0}\left(R_{s_1}(\alpha_0)V_{\alpha_0}(0)+R_{s_2s_1}(\alpha_0)V_{\hat s_2\hat s_1\alpha_0}(0)\right)V_{-\gamma h_1}(z)\mathcal{R}_{Id,s_1,s_2,s_2s_1}V_{\alpha_1}(1)V_{\alpha_\infty}(\infty),
		\end{align*}
		Since $\ps{-\gamma h_1,e_2}=0$ $V_{-\gamma h_1}(z)$ does not interact with the potential defined from $M^{\gamma e_2}$ and as a consequence the dependence in $z$ of the first term is actually completely explicit:
		\begin{align*}
			&\expect{e^{-\mu_{2}e^{\gamma\ps{\bm c,e_2}}M^{\gamma e_{2}}(\C)}V_{\alpha_0}(0)V_{-\gamma h_1}(z)V_{\alpha_1}(1)V_{\alpha_\infty}(\infty)}\\
			&=\norm{z}^{\gamma\ps{h_1,\alpha_0}}\norm{z-1}^{\gamma\ps{h_1,\alpha_1}}\expect{e^{-\mu_{2}e^{\gamma\ps{\bm c,e_2}}M^{\gamma e_{2}}(\C)}V_{\alpha_0-\gamma h_1}(0)V_{\alpha_1}(1)V_{\alpha_\infty}(\infty)}\quad\text{and}\\
			&\expect{e^{-\mu_{2}e^{\gamma\ps{\bm c,e_2}}M^{\gamma e_{2}}(\C)}V_{\alpha_0}(0)V_{-\gamma h_1}(z)V_{\alpha_1}(1)V_{\alpha_\infty}(\infty)}\\
			&=\norm{z}^{\gamma\ps{h_1,\hat s_1\alpha_0}}\norm{z-1}^{\gamma\ps{h_1,\alpha_1}}\expect{e^{-\mu_{2}e^{\gamma\ps{\bm c,e_2}}M^{\gamma e_{2}}(\C)}R_{s_1}(\alpha_0)V_{\hat s_1\alpha_0-\gamma h_1}(0)V_{\alpha_1}(1)V_{\alpha_\infty}(\infty)}
		\end{align*}
		with $\norm{z}^{\gamma\ps{h_1,\hat s_1\alpha_0}}=\norm{z}^{\gamma\ps{h_1,\alpha_0}}\norm{z}^{\gamma\ps{Q-\alpha_0,e_1}}$. We can proceed in the same fashion for the two other terms to get,
		for some constant $C$,
		\begin{align*}
			&\norm{z}^{-\gamma\ps{h_1,\alpha_0}}\norm{z-1}^{-\gamma\ps{h_1,\alpha_1}}\expect{\bm{\mathrm{R}}_{\bm\alpha}^{2}(\bm c)}=\\
			&\mathds 1_{\ps{\bm c,e_1}<0}\E\Big[e^{-\mu_{2}e^{\gamma\ps{\bm c,e_{2}}}M^{\gamma e_{2}}(\C)}\left(V_{\alpha_0-\gamma h_1}(0)+\norm{z}^{\gamma\ps{Q-\alpha_0,e_1}}R_{s_1}(\alpha_0)V_{\hat s_1\alpha_0-\gamma h_1}(0)\right)\mathcal{R}_{Id,s_1}V_{\alpha_1}(1)V_{\alpha_\infty}(\infty)\Big]\\
			&-\mathds 1_{\max_i\ps{\bm c,e_i}<0}\expect{\left(V_{\alpha_0-\gamma h_1}(0)+R_{s_2}(\alpha_0)V_{\hat s_2\alpha_0-\gamma h_1}(0)\right)\mathcal{R}_{Id,s_1,s_2,s_2s_1}V_{\alpha_1}(1)V_{\alpha_\infty}(\infty)}\\
			&-\mathds 1_{\max_i\ps{\bm c,e_i}<0}\norm{z}^{\gamma\ps{Q-\alpha_0,e_1}}\expect{\left(R_{s_1}(\alpha_0)V_{\hat s_1\alpha_0-\gamma h_1}(0)+R_{s_2s_1}(\alpha_0)V_{\hat s_2\hat s_1\alpha_0-\gamma h_1}(0)\right)\mathcal{R}_{Id,s_1,s_2,s_2s_1}V_{\alpha_1}(1)V_{\alpha_\infty}(\infty)}\\
			&+Cz+\bar C\bar z + l.o.t.
		\end{align*}
		Here the linear terms in $z$ comes from the fact that when we incorporate a reflection coefficient for $V_{\alpha_1}(1)$, we do not get a prefactor $\norm{z-1}^{\ps{\gamma h_1,\alpha_1}}$ but rather $\norm{z-1}^{\ps{\gamma h_1,\hat s_1\alpha_1}}$. Still the ratio of the two is smooth at $z=0$ so that the error can be incorporated in the $Cz,\bar C\bar z$ and $l.o.t.$ terms.
		
		Now since $\ps{h_1,e_2}=0$ the explicit expression~\eqref{eq:refl} of the reflection coefficients shows that $R_{s_2}(\alpha_0)=R_{s_2}(\alpha_0-\gamma h_1)$. Likewise we have the equality
		\[
		R_{s_1}(\alpha_0)=B^{(1)}_\gamma(\alpha_0,\gamma)R_{s_1}(\alpha_0-\gamma h_2).
		\]
		This allows to rewrite to rewrite $\norm{z}^{-\gamma\ps{h_1,\alpha_0}}\norm{z-1}^{-\gamma\ps{h_1,\alpha_1}}\expect{\bm{\mathrm{R}}_{\bm\alpha}^{2}(\bm c)}$ as
		\begin{align*}
			&\mathds 1_{\ps{\bm c,e_1}<0}\E\Big[\left(e^{-\mu_{2}e^{\gamma\ps{\bm c,e_{2}}}M^{\gamma e_{2}}(\C)}-\mathds 1_{\ps{\bm c,e_2}<0}\mathcal{R}_{Id,s_2}\right)V_{\alpha_0-\gamma h_1}(0)\mathcal{R}_{Id,s_1}V_{\alpha_1}(1)V_{\alpha_\infty}(\infty)\Big]\\
			&+\norm{z}^{\gamma\ps{Q-\alpha_0,e_1}}B^{(1)}_\gamma(\alpha_0,\gamma)\times\\
			&\mathds 1_{\ps{\bm c,e_1}<0}\E\Big[\left(e^{-\mu_{2}e^{\gamma\ps{\bm c,e_{2}}}M^{\gamma e_{2}}(\C)}-\mathds 1_{\ps{\bm c,e_2}<0}\mathcal{R}_{Id,s_2}\right)R_{s_1}(\alpha_0-\gamma h_2)V_{\hat s_1(\alpha_0-\gamma h_2)}(0)\mathcal{R}_{Id,s_1}V_{\alpha_1}(1)V_{\alpha_\infty}(\infty)\Big]\\
			&+Cz+\bar C\bar z + l.o.t.
		\end{align*}
		Recollecting terms we arrive to the following equality
		\begin{align*}
			\expect{\bm{\mathrm{R}}_{\bm\alpha}^{2}(\bm c)}=&\expect{\bm{\mathrm{R}}_{\alpha-\gamma h_1}^{2}(\bm c)}+\norm{z}^{\gamma\ps{Q-\alpha_0,e_1}}B^{(1)}_\gamma(\alpha_0,\gamma)\expect{\bm{\mathrm{R}}_{\alpha-\gamma h_2}^{2}(\bm c)}+Cz+\bar C\bar z + l.o.t.\\
			-\mathds 1_{\ps{\bm c,e_1}<0}\E\Big[&\left(e^{-\mu_{2}e^{\gamma\ps{\bm c,e_{2}}}M^{\gamma e_{2}}(\C)}-\mathds 1_{\ps{\bm c,e_2}<0}\mathcal{R}_{Id,s_2}\right)\mathcal{R}_{s_1}V_{\alpha_0-\gamma h_1}(0)\mathcal{R}_{Id,s_1}V_{\alpha_1}(1)V_{\alpha_\infty}(\infty)\Big]\\
			-\norm{z}^{\gamma\ps{Q-\alpha_0,e_1}}&B^{(1)}_\gamma(\alpha_0,\gamma)\mathds 1_{\ps{\bm c,e_1}<0}\times\\
			&\E\Big[\left(e^{-\mu_{2}e^{\gamma\ps{\bm c,e_{2}}}M^{\gamma e_{2}}(\C)}-\mathds 1_{\ps{\bm c,e_2}<0}\mathcal{R}_{Id,s_2}\right)V_{\alpha_0-\gamma h_2}(0)\mathcal{R}_{Id,s_1}V_{\alpha_1}(1)V_{\alpha_\infty}(\infty)\Big].
		\end{align*}
		
		The first line is what we expect to get to have well-defined three-point correlation functions $\ps{V_{\alpha_0-\gamma h_1}(0)V_{\alpha_1}(1)V_{\alpha_\infty}(\infty)}$ and $\ps{V_{\alpha_0-\gamma h_2}(0)V_{\alpha_1}(1)V_{\alpha_\infty}(\infty)}$.
		As a consequence we want to discard the two terms that appear between the second and fourth lines, and for this purpose we first show that they are integrable in the variable $\bm c$ (since if a remainder term is integrable then it should not appear in the expansion of $\bm{\mathrm{R}}_{\bm\alpha}^{2}(\bm c)$ according to the reasoning conducted in the proof of Theorem~\ref{thm:analycity}). To see why this is true we use that for any $w:\{0,1,\infty\}\to \{Id,s_1,s_2,s_2s_1\}$ such that $w(0)\in\{Id,s_2\}$:
		\[
		\ps{\hat w(0)\hat s_1(\alpha_0-\gamma h_1)+\hat w(1)\alpha_1+\hat w(\infty)\alpha_\infty-2Q,\omega_1}\geq \ps{\hat s_1(\alpha_0-\gamma h_1)+\alpha_1+\alpha_\infty-2Q,\omega_1},
		\]
		the latter being given by $\ps{\bm s,\omega_1}+\ps{Q-\alpha_0,e_1}$. Since $\ps{\alpha_0,e_1}<\frac2\gamma$ and $\ps{\bm s,\omega_1}>-\gamma$, $$\ps{\hat w(0)\hat s_1(\alpha_0-\gamma h_1)+\hat w(1)\alpha_1+\hat w(\infty)\alpha_\infty-2Q,\omega_1}>0.$$ 
		As a consequence the term in the second line is integrable over $\R$ with respect to $\ps{\bm c,e_1}$. The same reasoning shows that the term on the third and fourth lines is integrable too, since $\ps{\alpha_0-\gamma h_2+\alpha_1+\alpha_\infty-2Q,\omega_1}=\ps{\bm s,\omega_1}+\gamma$. 
		As a consequence we have proved that this extra term is indeed absolutely convergent in the $\bm c$ variable. By construction of the analytic continuation of the correlation functions and more precisely the reasoning conducted before Lemma~\ref{lemma:analycity} we see that this term will be compensated by other terms in $\mathfrak R_{\alpha_0,-\gamma h_1}^1(0;\bm c)$. More precisely, in a similar fashion as above we can expand (with the notations from Lemma~\ref{lemma:analycity})
		\begin{align*}
			E_{\bm\alpha}^{2}(\bm c)=E_{\alpha_0-\gamma h_1}^{2}&(\bm c)+\norm{z}^{\gamma\ps{Q-\alpha_0,e_1}}B^{(1)}_\gamma(\alpha_0,\gamma)E_{\alpha-\gamma h_2}^{2}(\bm c)+Cz+\bar C\bar z + l.o.t.\\
			-\sum_{w:\{1,\infty\}\to\{Id,s_1\}}&\frac{e^{-\ps{\bm s'(w),\omega_1}\ps{\bm c,e_1}}}{\ps{\bm s'(w),\omega_1}}\E\Big[\left(e^{-\mu_{2}e^{\gamma\ps{\bm c,e_{2}}}M^{\gamma e_{2}}(\C)}-\mathds 1_{\ps{\bm c,e_2}<0}\mathcal{R}_{Id,s_2}\right)\\
			& R_{s_1}(\alpha_0-\gamma h_1)V_{\hat s_1(\alpha_0-\gamma h_1)}(0)\mathcal{R}_{Id,s_1}R_{w(1)}(\alpha_1)V_{\hat w(1)\alpha_1}(1)R_{w(\infty)}(\alpha_\infty)V_{\hat w(\infty)\alpha_\infty}(\infty)\Big]\\
			-\norm{z}^{\gamma\ps{Q-\alpha_0,e_1}}&B^{(1)}_\gamma(\alpha_0,\gamma)\mathds 1_{\ps{\bm c,e_1}<0}\sum_{w:\{1,\infty\}\to\{Id,s_1\}} \frac{e^{-\ps{\tilde{\bm s}(w),\omega_1}\ps{\bm c,e_1}}}{\ps{\tilde{\bm s}(w),\omega_1}}\times\\
			&\E\Big[\left(e^{-\mu_{2}e^{\gamma\ps{\bm c,e_{2}}}M^{\gamma e_{2}}(\C)}-\mathds 1_{\ps{\bm c,e_2}<0}\mathcal{R}_{Id,s_2}\right)V_{\alpha_0-\gamma h_2}(0)\mathcal{R}_{Id,s_1}V_{\alpha_1}(1)V_{\alpha_\infty}(\infty)\Big]
		\end{align*}
		where $\bm s'(w)=\hat s_1(\alpha_0-\gamma h_1)+\hat w(1)\alpha_1+\hat w(\infty)\alpha_\infty-2Q$ and $\tilde{\bm s}(w)=\alpha_0-\gamma h_2+\hat w(1)\alpha_1+\hat w(\infty)\alpha_\infty-2Q$. Explicit computations show that the terms on the last three lines are given by the integral with respect to $\ps{\bm c,e_1}$ of the extra terms from $\expect{\bm{\mathrm{R}}_{\bm\alpha}^{2}(\bm c)}$, so that both terms will actually compensate and vanish in the expansion. 
		
		We can proceed in the same way for the other terms.
		For $\expect{\bm{\mathrm{R}}_{\bm\alpha}^{1}(\bm c)}$ we have 
		\begin{align*}
			&e^{-2\ps{Q,\bm c}}\expect{\mathcal{R}_{Id,s_1}\mathcal{R}_{Id,s_2}V_{\alpha_0}(0)V_{-\gamma h_1}(z)V_{\alpha_1}(1)V_{\alpha_\infty}(\infty)}\\
			&=\mathds 1_{\max_i\ps{\bm c,e_i}<0}\expect{\left(V_{\alpha_0-\gamma h_1}(0)+R_{s_2}(\alpha_0)V_{\hat s_2\alpha_0-\gamma h_1}(0)\right)\mathcal{R}_{Id,s_1,s_2,s_1s_2}V_{\alpha_1}(1)V_{\alpha_\infty}(\infty)}\\
			&+\mathds 1_{\max_i\ps{\bm c,e_i}<0}\norm{z}^{\gamma\ps{Q-\alpha_0,e_1}}\expect{\left(R_{s_1}(\alpha_0)V_{\hat s_1\alpha_0-\gamma h_1}(0)\right)\mathcal{R}_{Id, s_1,s_2,s_1s_2}V_{\alpha_1}(1)V_{\alpha_\infty}(\infty)}\\
			&+\mathds 1_{\max_i\ps{\bm c,e_i}<0}\norm{z}^{\gamma\ps{Q-\alpha_0,\rho}}\expect{R_{s_1s_2}(\alpha_0)V_{\hat s_1\hat s_2\alpha_0-\gamma h_1}(0)\mathcal{R}_{Id, s_1,s_2,s_1s_2}V_{\alpha_1}(1)V_{\alpha_\infty}(\infty)}\\
			&+Cz+\bar C\bar z + l.o.t.
		\end{align*}   
		while by the same reasoning as the one developed for $\Phi_{\alpha_0,\alpha}(z;\bm c)$, we have the equality
		\begin{align*}
			&\expect{\mathcal{R}_{Id,s_2}V_{\alpha_0}(0)V_{-\gamma h_1}(z)V_{\alpha_1}(1)V_{\alpha_\infty}(\infty)e^{-\mu_{1}e^{\gamma\ps{\bm c,e_{1}}}M^{\gamma e_{1}}(\C)}}\\
			=&\expect{V_{\alpha_0-\gamma h_1}(0)\mathcal{R}_{Id,s_2}V_{\alpha_1}(1)V_{\alpha_\infty}(\infty)e^{-\mu_{1}e^{\gamma\ps{\bm c,e_{1}}}M^{\gamma e_{1}}(\C)}}\\
			&+\norm{z}^{\gamma\ps{Q-\alpha_0,e_1}}B^{(1)}_\gamma(\alpha_0,\gamma)\expect{V_{\alpha_0-\gamma h_2}(0)\mathcal{R}_{Id,s_2}V_{\alpha_1}(1)V_{\alpha_\infty}(\infty)e^{-\mu_{1}e^{\gamma\ps{\bm c,e_{1}}}M^{\gamma e_{1}}(\C)}}\\
			&+\expect{R_{s_2}(\alpha_0)V_{\hat s_2\alpha_0-\gamma h_1}(0)\mathcal{R}_{Id,s_2}V_{\alpha_1}(1)V_{\alpha_\infty}(\infty)e^{-\mu_{1}e^{\gamma\ps{\bm c,e_{1}}}M^{\gamma e_{1}}(\C)}}\\
			&+\norm{z}^{\gamma\ps{Q-\hat s_2\alpha_0,e_1}}B^{(1)}_\gamma(\hat s_2\alpha_0,\gamma)\expect{R_{s_2}(\alpha_0)V_{\hat s_2\alpha_0-\gamma h_2}(0)\mathcal{R}_{Id,s_2}V_{\alpha_1}(1)V_{\alpha_\infty}(\infty)e^{-\mu_{1}e^{\gamma\ps{\bm c,e_{1}}}M^{\gamma e_{1}}(\C)}}\\
			&+Cz+\bar C\bar z+o\left(\norm{z}^{\gamma\ps{Q-\alpha_0,\rho}}\right).
		\end{align*}
		Since $\ps{\hat s_2\alpha_0-Q,e_1}=\ps{\alpha_0-Q,\rho}$ we infer from the above reasoning that
		\begin{align*}
			&\expect{\mathcal{R}_{Id,s_2}V_{\alpha_0-\gamma h_1}(0)V_{\alpha_1}(1)V_{\alpha_\infty}(\infty)e^{-\mu_{1}e^{\gamma\ps{\bm c,e_{1}}}M^{\gamma e_{1}}(\C)}}\\
			&+\norm{z}^{\gamma\ps{Q-\alpha_0,e_1}}B^{(1)}_\gamma(\alpha_0,\gamma)\expect{V_{\alpha_0-\gamma h_2}(0)\mathcal{R}_{Id,s_2}V_{\alpha_1}(1)V_{\alpha_\infty}(\infty)e^{-\mu_{1}e^{\gamma\ps{\bm c,e_{1}}}M^{\gamma e_{1}}(\C)}}\\
			&+\norm{z}^{\gamma\ps{Q- \alpha_0,\rho}}B^{(2)}_\gamma(\alpha_0,\gamma)\expect{R_{s_2}(\alpha_0-\gamma h_3)V_{\hat s_2(\alpha_0-\gamma h_3)}(0)\mathcal{R}_{Id,s_2}V_{\alpha_1}(1)V_{\alpha_\infty}(\infty)e^{-\mu_{1}e^{\gamma\ps{\bm c,e_{1}}}M^{\gamma e_{1}}(\C)}}\\
			&+Cz+\bar C\bar z+o\left(\norm{z}^{\gamma\ps{Q-\alpha_0,\rho}}\right).
		\end{align*}
		Put differently we can write that
		\begin{align*}
			\expect{\bm{\mathrm{R}}_{\bm\alpha}^{1}(\bm c)}=&\expect{\bm{\mathrm{R}}_{\alpha_0-\gamma h_1}^{1}(\bm c)}+\norm{z}^{\gamma\ps{Q-\alpha_0,e_1}}B^{(1)}_\gamma(\alpha_0,\gamma)\expect{(\bm{\mathrm{R}}_{\alpha_0-\gamma h_2}^{1}(\bm c)}\\
			+&\norm{z}^{\gamma\ps{Q- \alpha_0,\rho}}B^{(2)}_\gamma(\alpha_0,\gamma)R_{s_2}(\alpha_0-\gamma h_3)\expect{\bm{\mathrm{R}}_{\hat s_2(\alpha_0-\gamma h_3)}^{1}(\bm c)}+Cz+\bar C\bar z + l.o.t.\\
			+\mathds 1_{\max_i\ps{\bm c,e_i}<0} \E\Big[&\mathcal{R}_{s_1,s_1s_2}V_{\alpha_0-\gamma h_1}(0)\mathcal{R}_{Id,s_1}V_{\alpha_1}(1)V_{\alpha_\infty}(\infty)\Big]\\
			-\norm{z}^{\gamma\ps{Q-\alpha_0,e_1}}&B^{(1)}_\gamma(\alpha_0,\gamma)\mathds 1_{\ps{\bm c,e_2}<0}\times\\
			\Big(\E\Big[&e^{-\mu_{1}e^{\gamma\ps{\bm c,e_{1}}}M^{\gamma e_{1}}(\C)}\mathcal{R}_{s_2}V_{\alpha_0-\gamma h_2}(0)\mathcal{R}_{Id,s_2}V_{\alpha_1}(1)V_{\alpha_\infty}(\infty)\Big]\\
			&-\mathds 1_{\ps{\bm c,e_1}<0}\E\Big[\mathcal{R}_{Id,s_2,s_1s_2}V_{\alpha_0-\gamma h_2}(0)\mathcal{R}_{Id,s_1}V_{\alpha_1}(1)V_{\alpha_\infty}(\infty)\Big]\Big)\\
			-\norm{z}^{\gamma\ps{Q- \alpha_0,\rho}}& B^{(2)}_\gamma(\alpha_0,\gamma)\mathds 1_{\ps{\bm c,e_2}<0}\times\\
			\Big(\E\Big[&\left(e^{-\mu_{1}e^{\gamma\ps{\bm c,e_{1}}}M^{\gamma e_{1}}(\C)}-\mathds 1_{\ps{\bm c,e_1}<0}\mathcal{R}_{Id,s_1}\right)V_{\alpha_0-\gamma h_3}(0)\mathcal{R}_{Id,s_2}V_{\alpha_1}(1)V_{\alpha_\infty}(\infty)\Big]\\
			&-\mathds{1}_{\max _i\ps{\bm c,e_i}<0}\E\Big[\mathcal{R}_{Id,s_1,s_2}V_{\alpha_0-\gamma h_3}(0)\mathcal{R}_{Id,s_2}V_{\alpha_1}(1)V_{\alpha_\infty}(\infty)\Big]\Big).
		\end{align*}
		In analogy with $\expect{\bm{\mathrm{R}}_{\bm\alpha}^{1}(\bm c)}$ we see that the remainder terms are integrable with respect to $\bm c$ and will be compensated by another remainder term from $\mathfrak{R}_{\bm\alpha}^{1}(z;\bm c)$ and  $\mathfrak{R}_{\bm\alpha}^{1,2}(z;\bm c)$. Namely the first remainder term is integrable in $\ps{\bm c,e_1}$ (since $\ps{\hat s_1(\alpha_0-\gamma h_1)+\alpha_1+\alpha_\infty-2Q,\omega_1}>0$) and in $\ps{\bm c,e_2}$ as soon as $\ps{\bm s,\omega_2}>0$ (which we assumed to hold) and will be compensated by a term in $\mathfrak{R}_{\bm\alpha}^{1,2}(z;\bm c)$.
		The second remainder term is integrable with respect to $\ps{\bm c,e_2}$ and $\ps{\bm c,e_1}$ too and will be compensated by terms in $\mathfrak{R}_{\bm\alpha}^{1}(z;\bm c)$ and $\mathfrak{R}_{\bm\alpha}^{1,2}(z;\bm c)$, while the penultimate one is integrable with respect to $\ps{\bm c,e_2}$ and will correspond to a term in $\mathfrak{R}_{\bm\alpha}^{1}(z;\bm c)$. Eventually the last term is also integrable over $\R^2$ and will be annihilated by a term in $\mathfrak{R}_{\bm\alpha}^{1,2}(z;\bm c)$.

		Finally we can expand $\expect{\bm{\mathrm{R}}_{\bm\alpha}^{1,2}(\bm c)}$ in the same way, which yields
		\begin{align*}
			\expect{\bm{\mathrm{R}}_{\bm\alpha}^{1,2}(\bm c)}=&\expect{\bm{\mathrm{R}}_{\alpha_0-\gamma h_1}^{1,2}(\bm c)}+\norm{z}^{\gamma\ps{Q-\alpha_0,e_1}}B^{(1)}_\gamma(\alpha_0,\gamma)\expect{(\bm{\mathrm{R}}_{\alpha_0-\gamma h_2}^{1,2}(\bm c)}\\
			+&\norm{z}^{\gamma\ps{Q- \alpha_0,\rho}}B^{(2)}_\gamma(\alpha_0,\gamma)R_{s_2}(\alpha_0-\gamma h_3)\expect{\bm{\mathrm{R}}_{\hat s_2(\alpha_0-\gamma h_3)}^{1,2}(\bm c)}+Cz+\bar C\bar z + l.o.t.\\
			+\mathds 1_{\max_i\ps{\bm c,e_i}<0} \E\Big[&\mathcal{R}_{s_1,s_1s_2,s_2s_1,s_1s_2s_1}V_{\alpha_0-\gamma h_1}(0)\mathcal{R}_{Id,s_1}V_{\alpha_1}(1)V_{\alpha_\infty}(\infty)\Big]\\
			+\norm{z}^{\gamma\ps{Q-\alpha_0,e_1}}&B^{(1)}_\gamma(\alpha_0,\gamma)\mathds 1_{\max_i\ps{\bm c,e_i}<0} \E\Big[\mathcal{R}_{Id,s_2,s_1s_2,s_1s_2s_1}V_{\alpha_0-\gamma h_2}(0)\mathcal{R}_{Id,s_2}V_{\alpha_1}(1)V_{\alpha_\infty}(\infty)\Big]\\
			+\norm{z}^{\gamma\ps{Q- \alpha_0,\rho}}& B^{(2)}_\gamma(\alpha_0,\gamma)\mathds 1_{\max_i\ps{\bm c,e_i}<0}\times\\
			&R_{s_2}(\alpha_0-\gamma h_3)\E\Big[\mathcal{R}_{Id,s_2,s_1s_2,s_1s_2s_1}V_{\hat s_2(\alpha_0-\gamma h_3)}(0)\mathcal{R}_{Id,s_1}V_{\alpha_1}(1)V_{\alpha_\infty}(\infty)\Big].
		\end{align*}
		
		Putting everything together shows that
		\begin{align*}
			\mathfrak R_{\alpha_0,\alpha}(z;\bm c)&=\mathfrak R_{\alpha_0-\gamma h_1}(\bm c)+\norm{z}^{\gamma\ps{Q-\alpha_0,e_1}}B^{(1)}_\gamma(\alpha_0,\gamma)\mathfrak R_{\alpha_0-\gamma h_2}(\bm c)\\
			&+\norm{z}^{\gamma\ps{Q-\alpha_0,\rho}}B^{(2)}_\gamma(\alpha_0,\gamma)R_{s_2}(\alpha_0-\gamma h_3)\mathfrak R_{\hat s_2(\alpha_0-\gamma h_3)}(\bm c)+Cz+\bar C\bar z+o\left(\norm{z}^{\gamma\ps{Q-\alpha_0,\rho}}\right)\\
			&+R(z;\bm c),
		\end{align*}
		where $R(z;\bm c)$ is not a lower order term but will be compensated by terms appearing in the expansions of  $\mathfrak R_{\alpha_0,\alpha}^{i}(z;\bm c)$, $i=1,2$, and $\mathfrak R_{\alpha_0,\alpha}^{1,2}(z;\bm c)$.
		
		\subsubsection{Finishing up the proof of Lemma~\ref{lemma:OPE1}}
		All in all we see that we can write an expansion of $\mathcal H(z)$ which takes the form
		\begin{align*}
			&\mathcal{H}(z)=\int_{\R^2}\left(\Phi_{\alpha_0-\gamma h_1}(\bm c)-\mathfrak R_{\alpha_0-\gamma h_1}(\bm c)\right)d\bm c+\sum_{i=1}^2\int_\R \left(  \Phi^1_{\alpha_0-\gamma h_1}(c_i)-\mathfrak R_{\alpha_0-\gamma h_1}^i(c_i)\right)dc_i+\Phi^{1,2}_{\alpha_0-\gamma h_1}\\
			&+\norm{z}^{\gamma\ps{Q-\alpha_0,e_1}}B^{(1)}_\gamma(\alpha_0,\gamma)\times\\
			&\int_{\R^2}\left(\Phi_{\alpha_0-\gamma h_2}(\bm c)-\mathfrak R_{\alpha_0-\gamma h_2}(\bm c)\right)d\bm c+\sum_{i=1}^2\int_\R \left(\Phi^1_{\alpha_0-\gamma h_2}(c_i)-\mathfrak R_{\alpha_0-\gamma h_2}^i(c_i)\right)dc_i+\Phi^{1,2}_{\alpha_0-\gamma h_2}\\
			&+\norm{z}^{\gamma\ps{Q-\alpha_0,\rho}}B^{(2)}_\gamma(\alpha_0,\gamma)R_{s_2}(\alpha_0-\gamma h_3)\times\\
			&\int_{\R^2}\left(\Phi_{\hat s_2(\alpha_0-\gamma h_3)}(\bm c)-\mathfrak R_{\hat s_2(\alpha_0-\gamma h_3)}(\bm c)\right)d\bm c+\sum_{i=1}^2\int_\R \left(  \Phi^1_{\hat s_2(\alpha_0-\gamma h_3)}(c_i)-\mathfrak R_{\hat s_2(\alpha_0-\gamma h_3)}^i(c_i)\right)dc_i+\Phi^{1,2}_{\hat s_2(\alpha_0-\gamma h_3)}\\
			&+Cz+\bar C\bar z+o\left(\norm{z}^{\gamma\ps{Q-\alpha_0,\rho}}\right)
		\end{align*}
		as soon as the lower order terms which appear in the expansions of 
		$\Phi_{\alpha_0,-\gamma h_1}(z;\bm c)$ and the $\mathfrak R_{\alpha_0,-\gamma h_1}(z;\bm c)$ are integrable in $\bm c$, uniformly in $z$ in a neighborhood of the origin. 
		
		To see why this is indeed the case, let us for instance consider the remainder term $\mathfrak R(z)$ arising in the expansion of $\Phi_{\alpha_0,-\gamma h_1}(z;\bm c)$:
		\begin{align*}
			\mathfrak{R}(z)&=e^{\ps{\bm s+\gamma e_1,\bm c}}\int_\C \frac{\norm{x-z}^{\gamma^2}-\norm{x}^{\gamma^2}}{\norm{x}^{\gamma\ps{\alpha_0,e_1}}}\E\Big[e^{-e^{\gamma\ps{\bm c,e_2}}J_{\alpha_0-\gamma h_1,B_r}^2(x)}e^{-\sum_{i=1}^2e^{\gamma\ps{\bm c,e_i}}I^i_{\alpha_0-\gamma h_2}}\\
			&\left(F_1(x)\int_0^1e^{-e^{\gamma\ps{\bm c,e_{1}}}\left(t\delta J(z,x)+\delta J_{\alpha_0-\gamma h_1}^i(x)\right)}dte^{-e^{\gamma\ps{\bm c,e_2}}\delta J_{\alpha_0-\gamma h_1,\C_r}^2(x)}-1\right)\Big]d^2x.
		\end{align*}
		Then it is readily seen that the integral remains uniformly bounded when $\bm c$ ranges over $\mathcal C_-$. As a consequence $e^{-\ps{\bm s+\gamma e_1,\bm c}}\mathfrak R(z)$ remains uniformly bounded over $\mathcal C_-$: since we have assumed that $\ps{\bm s,\omega_1}>-\gamma$ and $\ps{\bm s,\omega_2}>0$ we infer that $\mathfrak R(z)$, viewed as a function of $\bm c$, is integrable in $\mathcal C_-$. Likewise when $\ps{\bm c,e_1}\to+\infty$, thanks to the existence of negative moments at all order of the GMC measures considered we can infer that \textit{e.g.} $e^{\ps{\bm s+\gamma e_1,\bm c}}\expect{\exp\left(-e^{\ps{\bm c,\gamma e_1}I^i_{\alpha_0-\gamma h_2}}\right)}$ is integrable for $\bm c$ in the region $\ps{\bm c,e_1}>0$ (see for instance the proof of~\cite[Theorem 3.1]{Toda_construction} for more details). As a consequence this remainder term is indeed integrable. We can proceed in the same ways for the other remainder terms that arise in the expansion of $\Phi_{\alpha_0,-\gamma h_1}(z;\bm c)$ and of $\mathfrak R_{\alpha_0,-\gamma h_1}^1(z;\bm c)$ too. As for $\mathfrak R_{\alpha_0,-\gamma h_1}(z;\bm c)$ the dependence in $\bm c$ of the remainder terms is completely explicit and readily seen to be uniformly integrable too.
		
		From this we infer that the desired asymptotic expansion for $\mathcal{H}$ does indeed hold provided that $\alpha_0$ is taken so that $\ps{\alpha_0,e_1}>\frac2\gamma-\eps$ and $\ps{\alpha_0,e_2}>q-\eps$, where $\eps$ is positive and small enough. As a consequence we can identify the coefficients that appear in this expansion with the ones stemming from the statement of Theorem~\ref{thm:BPZ} in this very case. This shows that we have the equalities
		\begin{equation*}
			\begin{split}
				&\frac{C_\gamma(\alpha_0-\gamma h_2,\alpha_1,\alpha_\infty)}{A_\gamma^{(1)}(-\gamma h_1,\alpha_0,\alpha_1,\alpha_\infty)}=\frac{C_\gamma(\alpha_0-\gamma h_1,\alpha_1,\alpha_\infty)}{B_\gamma^{(1)}(\alpha_0,\gamma)}\quad\text{and}\\
				&\frac{R_{s_2}(\hat s_2(\alpha_0-\gamma h_3))C_\gamma(\hat s_2(\alpha_0-\gamma h_3),\alpha_1,\alpha_\infty)}{A_\gamma^{(2)}(-\gamma h_1,\alpha_0,\alpha_1,\alpha_\infty)}=\frac{C_\gamma(\alpha_0-\gamma h_1,\alpha_1,\alpha_\infty)}{B^{(2)}_\gamma(\alpha_0,\gamma)}\cdot
			\end{split}
		\end{equation*}
		All the quantities that appear in the above expression are analytic in the weight $\alpha_0$ as soon as the probabilistic representation makes sense. Uniqueness of the analytic continuation allows to extend the validity of this equality to the whole range of values prescribed by Lemma~\ref{lemma:OPE1}.

		\subsubsection{Evaluation of some integrals}
		To finish up with the proof of Lemma~\ref{lemma:OPE1} it only remains to check that the expression of the coefficients $B^{(i)}_\gamma$ does coincide with that given by the integrals encountered above: 
		\begin{lemma}\label{lemma:evaluate_B}
			Let $a<2$ and $b>-2$ be two real numbers and set
			\[
			r_{a,b}(x)\coloneqq \norm{x}^b\left(\mathds 1_{a-b<2}+\mathds 1_{a-b<1}\frac b2\left(\frac{1}{x}+\frac1{\bar x}\right)\right).
			\]
			Then as soon as $a-b\in (0,+\infty)\setminus \{1,2\}$:
			\begin{equation}\label{eq:integrals_l}
				\begin{split}
					&\int_\C \frac{\norm{x-1}^b-r_{a,b}(x)}{\norm{x}^{a}}d^2x=\pi\frac{ l(-1+\frac{a-b}2)}{l(-\frac b2)l(\frac{a}{2})}\cdot
				\end{split}
			\end{equation}
		\end{lemma}
		\begin{proof}
			The integral~\eqref{eq:integrals_l} is the analytic continuation in the $a,b$ variables of the integral
			\[
			\int_\C\frac{\norm{x-1}^b}{\norm{x}^{a}}d^2x,
			\]
			from the region $\mathfrak{Re}(a-b)>2$ to the region $2>\mathfrak{Re}(a-b)>1$, the above integral being given by the expected result $\pi\frac{ l(-1+\frac{a-b}2)}{l(-\frac b2)l(\frac{a}{2})}$ (see \cite[p. 504]{Importance}). To see why, set
			\[
			F(a,b)\coloneqq \int_\C \frac{\norm{x-1}^b-\norm{x}_+^b}{\norm{x}^{a}}d^2x-\frac{2\pi}{2+b-a}\cdot
			\]
			Then $F$ is analytic in the domain $\mathfrak{Re}(a,-b)<2$, $\mathfrak{Re}(a-b)>1$, $2+b-a\neq0$. Furthermore over the subdomain where $\mathfrak{Re}(a-b)>2$, it is equal to 
			\[
			\int_\C\norm{x}^{a-b}\norm{x-1}^bd^2x=\pi\frac{ l(-1+\frac{a-b}2)}{l(-\frac b2)l(\frac{a}{2})}\cdot
			\]
			As a consequence by uniqueness of the analytic continuation $F(a,b)$ is also equal to $\pi\frac{ l(-1+\frac{a-b}2)}{l(-\frac b2)l(\frac{a}{2})}$ in the subdomain $2>\mathfrak{Re}(a-b)>1$. But in that case $F(a,b)$ is found to be equal to
			\[
			\int_\C \frac{\norm{x-1}^b-\norm{x}^b}{\norm{x}^{a}}d^2x.
			\]
			
			More generally the same argument shows that 
			\[
			\int_\C \frac{\norm{x-1}^b-r_{a,b}(x)}{\norm{x}^{a}}d^2x
			\]
			is the analytic continuation of \[
			\int_\C \frac{\norm{x-1}^b}{\norm{x}^{a}}d^2x,
			\]
			which allows to conclude that Equation~\eqref{eq:integrals_l} does indeed hold.
		\end{proof}
		With all these different pieces now put together, we are in position to wrap up the proof of Lemma~\ref{lemma:OPE1}. 
	\end{proof}
	
	\subsection{The case where $\chi=\gamma$, $\ps{\alpha_0,e_1}>\frac2\gamma$ and $\ps{\alpha_0,e_2}<\frac2\gamma$}\label{subsec:OPE_2}
	The case we consider next arises when $\ps{\alpha_0,e_1}>\frac2\gamma$ and $\ps{\alpha_0,e_2}<\frac2\gamma$, but still with $\chi=\gamma$. Based on a similar reasoning as the one developed above we can obtain an expansion that closely resembles that of Lemma~\ref{lemma:OPE1}:
	\begin{lemma}\label{lemma:OPE2}
		Under the assumptions that $q-\eps<\ps{\alpha_0,e_1}<q$ and $\frac{2}{\gamma}-\eps<\ps{\alpha_0,e_2}<\frac2\gamma$:
		\begin{equation}\label{eq:OPE_wrefl21}
			\begin{split}
				\mathcal{H}(z)=C_\gamma(&\alpha_0-\gamma h_1,\alpha_1,\alpha_\infty)\norm{\mathcal{H}_0(z)}^2\\
				&+B^{(1)}_\gamma(\alpha_0,\gamma)R_{s_1}(\hat s_1(\alpha_0-\gamma h_2))C_\gamma(\hat s_1(\alpha_0-\gamma h_2),\alpha_1,\alpha_\infty)\norm{\mathcal{H}_1(z)}^2\\
				&+B^{(2)}_\gamma(\alpha_0,\gamma)C_\gamma(\alpha_0-\gamma h_3,\alpha_1,\alpha_\infty)\norm{\mathcal{H}_2(z)}^2
			\end{split}
		\end{equation}
		provided that $(\alpha_0,-\gamma h_1,\alpha_1,\alpha_\infty)\in\mathcal{A}_4$ with $\ps{\bm s,\omega_2}>0$.
	\end{lemma}
	\begin{proof}
		The arguments developed to prove this statement are similar to the ones presented along the proof of Lemma~\ref{lemma:OPE1} but are slightly more subtle and require additional care. 
		Namely the main difference with the previous case is that the random variable $\delta I(z)$ introduced in Equation~\eqref{eq:deltaI} no longer has a finite $L^1$ moment because of the singularity at $x=0$ in the integrand. In order to understand what happens around this singular point this time we write
		\begin{align*}
			\delta I(z)=\delta I_r(z)+\delta J_r(z),\quad\text{with}\quad\delta J_r(z)\coloneqq\int_{B_r}\frac{\norm{x-z}^{\gamma^2}-\norm{x}^{\gamma^2}}{\norm{x}^{\gamma\ps{\alpha_0,e_1}}}F_1(x)M^{\gamma e_1}(d^2x_1)
		\end{align*}
		where $B_r=B(0,e^{-r})$ for some $r=r(z)$ such that $\norm{z}e^{r}\to+\infty$ to be fixed later on. 
		
		The asymptotic of the expectation term is governed by $\delta J_r(z)$ rather that $\delta I_r(z)$. To see why let us write that
		\begin{align*}
			&\Phi_{\alpha_0,-\gamma h_1}(z;\bm c)=\Phi_{\alpha_0-\gamma h_1}(\bm c)+\sum_{i=1}^3\Phi^i(z;\bm c)\quad\text{where}\\
			&\Phi^1(z;\bm c)\coloneqq e^{\ps{\bm s,\bm c}}\expect{\left(\exp\left(-e^{\gamma\ps{\bm c,e_1}}\delta J_r(z)\right)-1\right)e^{-\sum_{i=1}^2e^{\gamma\ps{\bm c,e_i}}I^i_{\alpha_0-\frac2\gamma h_1}}}\\
			&\Phi^2(z;\bm c)\coloneqq e^{\ps{\bm s,\bm c}}\expect{\left(\exp\left(-e^{\gamma\ps{\bm c,e_1}}\delta I_r(z)\right)-1\right)e^{-\sum_{i=1}^2e^{\gamma\ps{\bm c,e_i}}I^i_{\alpha_0-\frac2\gamma h_1}}}\\
			&\Phi^3(z;\bm c)\coloneqq e^{\ps{\bm s,\bm c}}\expect{\left(\exp\left(-e^{\gamma\ps{\bm c,e_1}}\delta I_r(z)\right)-1\right)\left(\exp\left(-e^{\gamma\ps{\bm c,e_1}}\delta J_r(z)\right)-1\right)e^{-\sum_{i=1}^2e^{\gamma\ps{\bm c,e_i}}I^i_{\alpha_0-\frac2\gamma h_1}}}
		\end{align*}
		and study the behavior of these three terms.
		
		Let us start with $\Phi^1$. This term can be treated in the same fashion as the last order term that arises in the asymptotic of Lemma~\ref{lemma:OPE1}. Namely we can use the radial-angular decomposition~\eqref{eq:rad_ang_dec} for the one-dimensional Brownian motion with drift $\nu=\ps{\alpha_0-Q,e_1}$ $\ps{\X_{t+r}(0)-\X_r(0),e_1}+\nu t$ to write that
		\begin{align*}
			\delta J_r(z)&=e^{\gamma\left(\lambda_r'+\M_1\right)}\mathrm J_r(z; -\M_1),\quad\text{with}\\
			\mathrm J_r(z; -\M_1)&\coloneqq \int_{0}^{+\infty}e^{\gamma \mathcal B^{\nu}_t}\int_0^{2\pi}\left(\norm{\frac{e^{-t-r+i\theta}}z-1}^{\gamma^2}-\norm{\frac{e^{-t-r}}z}^{\gamma^2}\right)F_1(e^{-t-r+i\theta})M^{\gamma e_1}_Y(dt+r,d\theta).
		\end{align*}
		Here $\lambda_r'\coloneqq\ps{\X_r(0)+(\alpha_0-Q)r+\gamma h_1\ln\norm{z},e_1}$ and $(\mathcal B^\nu_t)_{t\geq 0}$ is independent of $(\X(x))_{\norm{x}>e^{-r}}$ and started from $-\M_1$, sampled according to its marginal law. Proceeding along the same lines as above we see that, as soon as $\frac{e^{-r}}z\to 0$ as $z\to 0$ (which we assumed to hold)
		\[
		\Phi^1(z;\bm c)=\norm{z}^{\gamma\ps{Q-\alpha_0,e_1}}R_{s_1}(\alpha_0)\Phi_{\hat s_1\alpha_0-\gamma h_1}(\bm c) + l.o.t.
		\]
		Now from their expressions we know that $R_{s_1}(\alpha_0)=B^{(1)}_\gamma(\alpha_0,\gamma)R_{s_1}(\hat s_1(\alpha_0-\gamma h_2))$ and that $\hat s_1\alpha_0-\gamma h_1=\hat s_1(\alpha_0-\gamma h_2)$. As a consequence we end up with the asymptotic:
		\[
		\Phi^1(z;\bm c)=\norm{z}^{\gamma\ps{Q-\alpha_0,e_1}}R_{s_1}(\hat s_1(\alpha_0-\gamma h_2))\Phi_{\hat s_1(\alpha_0-\gamma h_2)}(\bm c) + \mathfrak{R}^1(z)
		\]
		where $\mathfrak{R}^1(z)$ is a $o\left(\norm{z}^{\gamma\ps{Q-\alpha_0,e_1}}\right)$. We will provide a more precise description of $\mathfrak{R}^1$ below.
		
		Let us now turn to $\Phi^2$. Since the integral now avoids the singular point $x=0$ we can proceed in the same way as for the first term in the expansion of Lemma~\ref{lemma:OPE1} but by replacing $\delta I(z)$ by $\delta I_r(z)$:
		\begin{align*}
			\E\Big[\left(\exp\left(-e^{\gamma\ps{\bm c,e_1}}\delta I_r(z)\right)-1\right)&e^{-\sum_{i=1}^2e^{\gamma\ps{\bm c,e_i}}I^i_{\alpha_0-\frac2\gamma h_1,0}}\Big]\\
			=-e^{\ps{\bm s+\gamma e_1,\bm c}}\int_{\C\setminus B_r} \frac{\norm{x-z}^{\gamma^2}-\norm{x}^{\gamma^2}}{\norm{x}^{\gamma\ps{\alpha_0,e_1}}}&F_1(x)\times\\
			&\E\left[\int_0^1e^{-te^{\gamma\ps{\bm c,e_{1}}}\delta J(z,x)}dt\prod_{i=1}^2e^{-e^{\gamma\ps{\bm c,e_i}}J_{\alpha_0-\gamma h_1}^i(x)}\right]d^2x.
		\end{align*}
		The next step is to write the expectation term as
		\begin{align*}
			&\E\left[\int_0^1e^{-te^{\gamma\ps{\bm c,e_{1}}}\delta J(z,x)}dt e^{-e^{\gamma\ps{\bm c,e_1}}J_{\alpha_0-\gamma h_1}^1(x)}e^{-e^{\gamma\ps{\bm c,e_2}}I_{\alpha_0-\gamma h_2}^2}\right]d^2x\\
			&-e^{\gamma\ps{\bm c,e_2}}\E\left[\delta J_{\alpha_0-\gamma h_1}^2(x)\int_0^1e^{-te^{\gamma\ps{\bm c,e_{1}}}\delta J(z,x)}dt e^{-e^{\gamma\ps{\bm c,e_1}}J_{\alpha_0-\gamma h_1}^1(x)}e^{-e^{\gamma\ps{\bm c,e_2}}I_{\alpha_0-\gamma h_2}^2}\right]d^2x+l.o.t.
		\end{align*}
		Along the same lines as in the proof of Lemma~\ref{lemma:OPE1} we infer that as $z\to0$, the term that appears on the second line is such that
		\begin{align*}
			e^{\ps{\bm s+\gamma \rho,\bm c}}\int_{\C\setminus B_r}& \frac{\norm{x-z}^{\gamma^2}-\norm{x}^{\gamma^2}}{\norm{x}^{\gamma\ps{\alpha_0,e_1}}}F_1(x)\times\\
			&\E\left[\delta J_{\alpha_0-\gamma h_1}^2(x)\int_0^1e^{-te^{\gamma\ps{\bm c,e_{1}}}\delta J(z,x)}dt e^{-e^{\gamma\ps{\bm c,e_1}}J_{\alpha_0-\gamma h_1}^1(x)}e^{-e^{\gamma\ps{\bm c,e_2}}I_{\alpha_0-\gamma h_2}^2}\right]d^2x\\
			&\sim \norm{z}^{\gamma\ps{Q-\alpha_0,\rho}}\Phi_{\alpha_0-\gamma h_3}(\bm c)\int_{\C\setminus B_r} \frac{\norm{x-1}^{\gamma^2}-\norm{x}^{\gamma^2}}{\norm{x}^{\gamma\left(\ps{\alpha_0,e_1+e_2}-q\right)}}d^2x\int_{\C} \frac{\norm{y-1}^{\gamma^2}-\norm{y}^{\gamma^2}}{\norm{y}^{\gamma\ps{\alpha_0,e_2}}}d^2y.
		\end{align*}
		The last integrals are well-defined as $r\to+\infty$, and can be evaluated using Lemma~\ref{lemma:evaluate_B} above, thanks to which we can deduce that this term is asymptotically equivalent to \[
		\norm{z}^{\gamma\ps{Q-\alpha_0,\rho}}B^{(2)}_\gamma(\alpha_0,\gamma)\Phi_{\alpha_0-\gamma h_3}(\bm c)
		\] as soon as $r\to+\infty$ when $z\to 0$. 
		Therefore we can write that 
		\[
		\Phi^2(z;\bm c)=\norm{z}^{\gamma\ps{Q-\alpha_0,\rho}}B^{(2)}_\gamma(\alpha_0,\gamma)\Phi_{\alpha_0-\gamma h_3}(\bm c)+\mathfrak{R}_2(z),
		\]
		where the remainder term is given by 
		\begin{align*}
			\mathfrak{R}^2(z)=-e^{\ps{\bm s+\gamma e_1,\bm c}}\int_{\C\setminus B_r}& \frac{\norm{x-z}^{\gamma^2}-\norm{x}^{\gamma^2}}{\norm{x}^{\gamma\ps{\alpha_0,e_1}}}F_1(x)\times\\
			&\E\left[\int_0^1e^{-te^{\gamma\ps{\bm c,e_{1}}}\delta J(z,x)}dt e^{-e^{\gamma\ps{\bm c,e_1}}J_{\alpha_0-\gamma h_1}^1(x)}e^{-e^{\gamma\ps{\bm c,e_2}}I_{\alpha_0-\gamma h_2}^2}\right]d^2x
		\end{align*}
		up to a $o\left(\norm{z}^{\gamma\ps{Q-\alpha_0,\rho}}\right)$.

		\subsubsection{The reflection terms}
		The reflection terms required to make sense of the four-point correlation functions can be processed in the very same way as in the proof of Lemma~\ref{lemma:OPE1} so we omit the computations of the asymptotic of these terms. We stress that conducting these explicit but tedious computations show that the condition for the reflection terms involved to be integrable is the same as in Lemma~\ref{lemma:OPE1}, that is $\ps{\bm s,\omega_2}>0$. 
		
		\subsubsection{The remainder terms}\label{sub:remainder}
		It remains to consider the remainder term $\mathfrak{R}(z)\coloneqq \mathfrak{R}^1(z)+\mathfrak{R}^2(z)+\Phi^3(z)$ that arises in the above expansions. For this we start by writing that 
		\begin{align*}
			\mathfrak{R}(z)=\Phi_{\alpha_0,-\gamma h_1}(z;\bm c)&-\norm{z}^{\gamma\ps{Q-\alpha_0,e_1}}R_{s_1}(\hat s_1(\alpha_0-\gamma h_2))\Phi_{\hat s_1(\alpha_0-\gamma h_2)}(\bm c)\\
			&-\norm{z}^{\gamma\ps{Q-\alpha_0,\rho}}B^{(2)}_\gamma(\alpha_0,\gamma)\Phi_{\alpha_0-\gamma h_3}(\bm c).
		\end{align*}
		In order to explain why this term does not contribute to the identification of the coefficients in the expansion of $\mathcal H$ from Theorem~\ref{thm:OPE} we make the simplifying assumption that $\ps{\bm s,\omega_i}>0$ for all $i$. If this is not the case then we can reproduce the argument developed in Subsection~\ref{sub:refl1} to show that the reasoning remains valid in that case too.
		
		To start with we already know thanks to Theorem~\ref{thm:BPZ} that the \emph{a priori} expansion of $\int_{\R^2}\Phi_{\alpha_0,-\gamma h_1}(z;\bm c)d\bm c$ only features terms which are (infinite) polynomials in $z,\bar z$ multiplied by $\norm{z}^{\gamma\ps{Q-\alpha_0,h_1-h_i}}$ for $i=1,2,3$ as soon as $\mathcal{H}$ makes sense. Combining this with the above description of $\mathfrak R(z)$ allows to write that this is also the case for $\mathfrak R(z)$ in that
		\begin{equation}\label{eq:JPP}
			\int_{\R^2}\mathfrak{R}(z)d\bm c=\sum_{i=0}^2 \norm{z}^{\gamma\ps{Q-\alpha_0,\sum_{j=1}^ie_j}}\mathrm{P}_i(z,\bar z;\alpha_0)
		\end{equation}
		for some power series $\mathrm{P}_i(z,\bar z;\alpha_0)$ in $z,\bar z$, provided that $\mathcal{H}$ is well-defined and under the assumption that $\ps{\alpha_0,e_2}<\frac2\gamma$ (and where we have not added the reflection terms that may appear above not to overload the discussion). Our goal is to show that $\mathrm{P}_i(0,0;\alpha_0)=0$, and for this we will now carry a study of the behavior of $\mathfrak{R}(z)$ as $z\to0$.
		
		To start with the explicit expression of $\mathfrak{R}(z)$ shows, in agreement with Theorem~\ref{thm:analycity}, that the dependence in $\ps{\alpha_0,e_2}$ of $\mathfrak{R}(z)$ is actually analytic in a complex neighborhood of $(\frac2\gamma-\eps,q)$. For instance let us consider the first remainder term $\mathfrak{R}^1(z)$, defined by
		\[
		\mathfrak{R}^1(z)=\Phi^1(\bm c)-\norm{z}^{\gamma\ps{Q-\alpha_0,e_1}}B^{(1)}_\gamma(\alpha_0,\gamma)R_{s_1}(\hat s_1(\alpha_0-\gamma h_2))\Phi_{\hat s_1(\alpha_0-\gamma h_2)}(\bm c).
		\]
		Then the probabilistic representation of $\Phi_{\hat s_1(\alpha_0-\gamma h_2)}$ makes sense for $\ps{\hat s_1(\alpha_0-\gamma h_2)-Q,e_i}<0$ for $i=1,2$, and in particular makes sense for $\alpha_0$ close to $Q$. Likewise the term
		\[
		\Phi^1(\bm c)=e^{\ps{\bm s,\bm c}}\expect{\left(\exp\left(-e^{\gamma\ps{\bm c,e_1}}\delta J_r(z)\right)-1\right)e^{-\sum_{i=1}^2e^{\gamma\ps{\bm c,e_i}}I^i_{\alpha_0-\gamma h_1}}}
		\]
		is perfectly well-defined when $\alpha_0$ is taken close to $Q$. The reasoning developed along the proof of Theorem~\ref{thm:analycity} shows that in that case both terms are analytic in $\ps{\alpha_0,e_2}$. We can proceed in the same way for the other remainder terms, which are probabilistically speaking perfectly well-defined for $\alpha_0$ close to $Q$. This shows that $\mathfrak{R}$ depends analytically in $\ps{\alpha_0,e_2}$ in a complex neighborhood of $(\frac2\gamma-\eps,q)$.
		In particular this allows to extend the validity of Equation~\eqref{eq:JPP} for $\ps{\alpha_0,e_2}$ in $(\frac2\gamma-\eps,q)$, the corresponding integral in the $\bm c$ variable being still absolutely convergent as argued in the proof of Theorem~\ref{thm:analycity}.
		
		Now we can use the fusion asymptotics~\eqref{eq:fusion} to infer that
		\begin{equation}\label{eq:JPP1}
			\mathfrak{R}^2(z)=o\left(\norm{z}^{\gamma\ps{Q-\alpha_0,e_1}+\frac{\left(\ps{\alpha_0,e_1}-\frac2\gamma\right)^2}{4}-\eps.}\right),
		\end{equation}
		where the $o(\cdot)$ is uniformly integrable with respect to $\bm c$.
		To see why we can reproduce the argument developed in the proof of Lemma~\ref{lemma:OPE1}:
		\begin{align*}
			\int_{\C\setminus B_r}& \frac{\norm{x-z}^{\gamma^2}-\norm{x}^{\gamma^2}}{\norm{x}^{\gamma\ps{\alpha_0,e_1}}}F_1(x)\times\\
			&\E\left[\int_0^1e^{-te^{\gamma\ps{\bm c,e_{1}}}\delta J(z,x)}dt e^{-e^{\gamma\ps{\bm c,e_1}}J_{\alpha_0-\gamma h_1}^1(x)}e^{-e^{\gamma\ps{\bm c,e_2}}I_{\alpha_0-\gamma h_2}^2}\right]d^2x\\
			\sim \norm{z}^{\gamma^2}e^{-r(2-\gamma\ps{\alpha_0,e_1})} \int_{\C\setminus B_0}& \norm{x}^{-\gamma\ps{\alpha_0,e_1}}F_1(xe^{-r})\times\\
			&\E\left[\int_0^1e^{-te^{\gamma\ps{\bm c,e_{1}}}\delta J(z,e^{-r}x)}dt e^{-e^{\gamma\ps{\bm c,e_1}}J_{\alpha_0-\gamma h_1}^1(e^{-r}x)}e^{-e^{\gamma\ps{\bm c,e_2}}I_{\alpha_0-\gamma h_2}^2}\right]d^2x
		\end{align*}
		where we can use Equation~\eqref{eq:fusion} to see that the expectation term is at most a $o\left(e^{-r\left(\frac{(\ps{\alpha_0,e_1}-\frac2\gamma)^2}{4}-\eps\right)}\right)$. Since $r$ is chosen so that $\frac{e^{-r}}{z}\to 0$ as $z\to 0$ we recover Equation~\eqref{eq:JPP1}. This shows that $\mathfrak{R}^2(z)$ is a $o\left(\norm{z}^{\gamma\ps{Q-\alpha_0,\rho}}\right)$ as soon as $\ps{\alpha_0,e_2}$ is taken close enough to $q$ and $\ps{\alpha_0,e_1}$ different from $\frac2\gamma$. Moreover this term is seen to be integrable with respect to $\bm c$.

		The same applies to the other terms $\mathfrak{R}^1(z)$ and $\Phi^3(z;\bm c)$. This is readily seen for $\Phi^3$ which is a lower order term compared to $\Phi^2$, while for $\mathfrak{R}^1$ we can write that
		\begin{align*}
			&\expect{\left(\exp\left(-e^{\gamma\ps{\bm c,e_1}}\delta J_r(z)\right)-1\right)e^{-\sum_{i=1}^2e^{\gamma\ps{\bm c,e_i}}I^i_{\alpha_0-\gamma h_1}}}\\
			&=\int_0^{+\infty}(-\nu)e^{\nu\M_1}\expect{\left(\exp\left(-e^{\gamma\left(\ps{\bm c,e_1}+\lambda_r'+\M_1\right)}\mathrm J_r(z; -\M_1)\right)-1\right)e^{-\sum_{i=1}^2e^{\gamma\ps{\bm c,e_i}}I^i_{\alpha_0-\gamma h_1}}}d\M_1\\
			&=\norm{z}^{\gamma\ps{Q-\alpha_0,e_1}}e^{\ps{Q-\alpha_0,e_1}{\ps{\bm c,e_1}}}\times\\
			&\expect{e^{-\nu\X_r(0)-\nu^2r}\int_{\ps{\bm c,e_1}+\lambda_r'}^{+\infty}(-\nu)e^{\nu\M_1}\left(e^{-e^{\gamma\M_1}\mathrm J_r(z;\ps{\bm c,e_1}+\lambda_r' -\M_1)}-1\right)e^{-\sum_{i=1}^2e^{\gamma\ps{\bm c,e_i}}I^i_{\alpha_0-\gamma h_1}}}d\M_1.
		\end{align*}
		In the latter note that $I^i_{\alpha_0-\frac2\gamma h_1}$ does depend on $\M_1$ and $\mathrm{J}$ since it involves the sigma-algebra generated by the $(\X(x))_{\norm{x}<e^{-r}}$. However we can control the integral over $B_r$ that enters $I^i_{\alpha_0-\frac2\gamma h_1}$ in the same fashion as in the proof of Lemma~\ref{lemma:OPE1}. As a consequence and in order to keep the proof concise we can assume that $I^i_{\alpha_0-\gamma h_1}$ is independent of $\M_1$ and $\mathrm{J}$. 
		Now the first exponential term entering the expectation is a Girsanov transform, thanks to which we end up with
		\begin{align*}
			&=\norm{z}^{\gamma\ps{Q-\alpha_0,e_1}}e^{\ps{Q-\alpha_0,e_1}{\ps{\bm c,e_1}}}\times\\
			&\expect{\int_{\ps{\bm c,e_1}+\hat{\lambda}_r}^{+\infty}(-\nu)e^{\nu\M_1}\left(e^{-e^{\gamma\M_1}\mathrm J_r(z;\ps{\bm c,e_1}+\hat{\lambda}_r -\M_1)}-1\right)e^{-\sum_{i=1}^2e^{\gamma\ps{\bm c,e_i}}I^i_{\hat s_1\alpha_0-\gamma h_1}}}d\M_1.
		\end{align*}
		where $\hat\lambda_r=\ps{\X_r(0)+(Q-\alpha_0)r+\gamma h_1\ln\norm{z},e_1}$. Therefore we can write that
		\begin{align*}
			\mathfrak{R}^1(z)&=\norm{z}^{\gamma\ps{Q-\alpha_0,e_1}}e^{\ps{Q-\alpha_0,e_1}{\ps{\bm c,e_1}}}\times\\
			&\expect{\int_{-\infty}^{\ps{\bm c,e_1}+\hat{\lambda}_r}(-\nu)e^{\nu\M_1}\left(e^{-e^{\gamma\M_1}\mathrm J(-\infty)}-1\right)e^{-\sum_{i=1}^2e^{\gamma\ps{\bm c,e_i}}I^i_{\hat s_1\alpha_0-\gamma h_1}}}d\M_1\\
			+&\expect{\int_{\ps{\bm c,e_1}+\hat{\lambda}_r}^{+\infty}(-\nu)e^{\nu\M_1}\left(e^{-e^{\gamma\M_1}\mathrm J_r(z;\ps{\bm c,e_1}+\hat{\lambda}_r -\M_1)}-e^{-e^{\gamma\M_1}\mathrm J(-\infty)}\right)e^{-\sum_{i=1}^2e^{\gamma\ps{\bm c,e_i}}I^i_{\hat s_1\alpha_0-\gamma h_1}}}d\M_1
		\end{align*}
		up to lower order terms, where $\mathrm{J}(-\infty)\coloneqq \lim\limits_{z\to 0}\mathrm J_r(z;\ps{\bm c,e_1}+\hat{\lambda}_r -\M_1)$. The first expectation term is asymptotically equivalent to 
		\begin{align*}
			\expect{\mathrm J(-\infty)\int_{-\infty}^{\ps{\bm c,e_1}+\hat{\lambda}_r}\nu e^{(\nu+\gamma)\M_1}e^{-\sum_{i=1}^2e^{\gamma\ps{\bm c,e_i}}I^i_{\hat s_1\alpha_0-\gamma h_1}}}d\M_1
		\end{align*} which is at most a $\mathcal{O}\left(\norm{z}^{\gamma^2}\right)$ while the second one can also be seen to be a $o\left(\norm{z}^{\eps}\right)$ for some positive $\eps$ that only depends on $\ps{\alpha_0,e_1}$. Indeed one can use the Markov property of the process $\mathcal{B}$ in the same fashion as in the proof of~\cite[Proposition 4.10]{Toda_correl1} to see that  the second expectation term is governed by a term of same order as
		\begin{align*}
			\expect{e^{\gamma(\ps{\bm c,e_1}+\hat{\lambda}_r)}\int_{\ps{\bm c,e_1}+\hat{\lambda}_r}^{+\infty}e^{(\nu+\gamma)\M_1}e^{-e^{\gamma\M_1}\mathrm J(-\infty)}e^{-\sum_{i=1}^2e^{\gamma\ps{\bm c,e_i}}I^i_{\hat s_1\alpha_0-\gamma h_1}}}d\M_1\\
			&
		\end{align*}
		which is as desired.
		In particular for $\ps{Q-\alpha_0,e_2}$ small enough we see that the corresponding term in the expansion of $\mathfrak{R}$ is a $o\left(\norm{z}^{\gamma\ps{Q-\alpha_0,\rho}}\right)$.
		
		Recollecting terms allows to claim that for $\ps{\alpha_0,e_2}$ close to $q$ 
		\[
		\mathfrak{R}(z)=o\left(\norm{z}^{\gamma\ps{Q-\alpha_0,\rho}}\right)
		\]
		where the $o(\cdot)$ term in absolutely integrable with respect to $\bm c$. From this we can infer that
		\[
		\int_{\R^2}\mathfrak{R}(z)d\bm c=o\left(\norm{z}^{\gamma\ps{Q-\alpha_0,\rho}}\right).
		\]
		As a consequence we get that the polynomials $\mathrm{P}_i(z,\bar z;\alpha_0)$ are such that $\mathrm{P}_i(0,0;\alpha_0)=0$ for $i=1,2,3$ under this assumption.
		Now this coefficient is known to depend analytically in $\ps{\alpha_0,e_2}$, so that we can deduce that $\mathrm{P}_i(0,0;\alpha_0)=0$ as soon as $\mathfrak{R}$ makes sense. This allows to discard the remainder term $\mathfrak{R}(z)$ in the identification of the coefficients arising in the expansion of $\mathcal{H}$ from either Theorem~\ref{thm:BPZ} and Lemma~\ref{lemma:OPE2}.
		
		To sum things up, we have proved that $\mathcal H$ has the following expansion:
		\begin{align*}
			\mathcal H(z)&=C_\gamma(\alpha_0-\gamma h_1,\alpha_1,\alpha_\infty) + Az+\bar A\bar z\\
			&+\norm{z}^{\gamma\ps{Q-\alpha_0,e_1}}B^{(1)}_\gamma(\alpha_0,\gamma)R_{s_1}(\hat s_1(\alpha_0-\gamma h_2))C_\gamma(\hat s_1(\alpha_0-\gamma h_2),\alpha_1,\alpha_\infty)\left(Bz+\bar B\bar z\right)\\
			& + \norm{z}^{\gamma\ps{Q-\alpha_0,\rho}}B^{(2)}_\gamma(\alpha_0,\gamma)C_\gamma(\alpha_0-\gamma h_3,\alpha_1,\alpha_\infty)+l.o.t.
		\end{align*}

		We can conclude for the proof of Lemma~\ref{lemma:OPE2} in the same fashion as we did for Lemma~\ref{lemma:OPE1} by identification with the coefficients arising in the expansion of Theorem~\ref{thm:BPZ}.
	\end{proof}
	
	
	\subsection{The case where $\chi=\frac2\gamma$}\label{subsec:OPE_3}
	Eventually we treat the case where $\chi$ is equal to $\frac2\gamma$. Under this assumption the expansion features two reflection terms, as the following statement discloses:
	\begin{lemma}\label{lemma:OPE3}
		Assume that $\chi=\frac2\gamma$. Then 
		\begin{equation}\label{eq:OPE_wrefl32}
			\begin{split}
				\mathcal{H}(z)=&C_\gamma(\alpha_0-\frac2\gamma\gamma h_1,\alpha_1,\alpha_\infty)\norm{\mathcal{H}_0(z)}^2\\
				&+B^{(1)}_\gamma(\alpha_0,\frac2\gamma)R_{s_1}(\hat s_1(\alpha_0-\frac2\gamma h_2))C_\gamma(\hat s_1(\alpha_0-\frac2\gamma h_2),\alpha_1,\alpha_\infty)\norm{\mathcal{H}_1(z)}^2\\
				&+B^{(2)}_\gamma(\alpha_0,\frac 2\gamma)R_{s_1s_2}(\hat s_1\hat s_2(\alpha_0-\frac2\gamma h_3))C_\gamma(\hat s_1\hat s_2(\alpha_0-\frac2\gamma h_3),\alpha_1,\alpha_\infty)\norm{\mathcal{H}_2(z)}^2
			\end{split}
		\end{equation}
		as soon as $(\alpha_0,-\frac2\gamma h_1,\alpha_1,\alpha_\infty)\in\mathcal{A}_4$.
	\end{lemma}
	\begin{proof}
		The terms associated to the reflection coefficients can be dealt with in the same way as in the proof of Lemma~\ref{lemma:OPE1} so that we will omit them in what follows and focus on the expansion of $\Phi_{\alpha_0,-\frac2\gamma h_1}(z;\bm c)$. To start with we make the assumption that $\alpha_0$ is close to $Q$, for a meaning of close that will be made precise during the proof. 
		
		We start by picking $r=r(z)>0$ such that $\norm{z}e^r\to +\infty$ as $z\to0$, to be fixed later on, and split the integrals involved as
		\begin{align*}
			I^i_{\alpha_0,-\frac2\gamma h_1}(z)=I^i_r(z)+J^i_r(z),\quad\text{with}\quad J^i_{r}(z)\coloneqq \int_{B_r}\frac{\norm{x_i-z}^{\frac{2}{\gamma}\ps{h_1,e_i}}}{\norm{x_i}^{\gamma\ps{\alpha_0,e_i}}}F_i(x)M^{\gamma e_i}(d^2x_i)
		\end{align*}
		where $B_r=B(0,e^{-r})$. By doing so we can write that
		\begin{align*}
			\Phi_{\alpha_0,-\frac2\gamma h_1}(z;\bm c)=&\Phi_{\alpha_0-\frac2\gamma h_1}(\bm c) + \Phi^1(z;\bm c)+\Phi^2(z;\bm c)+\Phi^3(z;\bm c),\quad\text{where}\\
			\Phi^1(z;\bm c)\coloneqq& e^{\ps{\bm s,\bm c}}\E\left[e^{-\sum_{i=1}^2e^{\gamma\ps{\bm c,e_i}}I^i_r(0)}\left(e^{-e^{\gamma\ps{\bm c,e_1}}J^1_r(z)}-1\right)\right] \\
			\Phi^2(z;\bm c)\coloneqq &e^{\ps{\bm s,\bm c}}\E\left[e^{-\sum_{i=1}^2e^{\gamma\ps{\bm c,e_i}}I^i_r(0)}\left(e^{-e^{\gamma\ps{\bm c,e_1}}J^1_r(z)}-1\right)\left(e^{-e^{\gamma\ps{\bm c,e_2}}J^2_r}-1\right)\right]\\
			\Phi^3(z;\bm c)\coloneqq& e^{\ps{\bm s,\bm c}}\E\left[e^{-e^{\gamma\ps{\bm c,e_1}}I^1_{\alpha_0,-\frac2\gamma h_1}(z)}e^{-e^{\gamma\ps{\bm c,e_2}}I^2_{\alpha_0,-\frac{2}\gamma h_1}}\left(1-e^{-e^{\gamma\ps{\bm c,e_i}}(I^1_r(0)-I^1_r(z))}\right)\right]\\
			+ &e^{\ps{\bm s,\bm c}}\E\left[e^{-e^{\gamma\ps{\bm c,e_1}}I^1_r(0)}e^{-e^{\gamma\ps{\bm c,e_2}}I^2_{\alpha_0,-\frac{2}\gamma h_1}}\left(1-e^{-e^{\gamma\ps{\bm c,e_i}}J^1_r(0)}\right)\right].
		\end{align*}
		As we will see, the expansion of $\Phi^1$ allows to infer the equality
		\[
		R_{s_1}(\alpha_0)C_\gamma\left(\hat s_1\alpha_0-\frac2\gamma h_1,\alpha_1,\alpha_\infty\right)=A_\gamma^{(1)}\left(-\frac2\gamma h_1,\alpha_0,\alpha_1,\alpha_\infty\right)C_\gamma\left(\alpha_0-\frac2\gamma h_1,\alpha_1,\alpha_\infty\right)
		\] 
		while the expansion of $\Phi^2$ yields 
		\[
		R_{s_1s_2}(\alpha_0)C_\gamma\left(\hat s_1\hat s_2\alpha_0-\frac2\gamma h_1,\alpha_1,\alpha_\infty\right)=A_\gamma^{(2)}\left(-\frac2\gamma h_1,\alpha_0,\alpha_1,\alpha_\infty\right)C_\gamma\left(\alpha_0-\frac2\gamma h_1,\alpha_1,\alpha_\infty\right)
		\] 
		Since the coefficients that appear in these expansions satisfy the properties that
		\[
		R_{s_1}(\alpha_0)=B^{(1)}_\gamma(\alpha_0,\frac2\gamma)R_{s_1}\left(\alpha_0-\frac2\gamma h_2\right),\quad R_{s_1s_2}(\alpha_0)=B^{(2)}_\gamma(\alpha_0,\frac2\gamma)R_{s_1s_2}\left(\alpha_0-\frac2\gamma h_3\right),
		\]
		the statement of Lemma~\ref{lemma:OPE3} follows from these equalities along the same lines as in the proof of Lemma~\ref{lemma:OPE1}.
		
		\subsubsection{The leading terms in the expansion}
		We start by treating the first term in the expansion of $\Phi_{\alpha_0,-\frac2\gamma h_1}(z;\bm c)$, given by $\Phi^1$. The very same reasoning as the one conducted along the proof of Lemmas~\ref{lemma:OPE1} and~\ref{lemma:OPE2} still applies and shows that this term is asymptotically equivalent to 
		\[
		\norm{z}^{\frac2\gamma\ps{Q-\alpha_0,e_1}}R_{s_1}(\alpha_0)\Phi_{\hat s_1\alpha_0-\frac2\gamma h_1}(\bm c).
		\] 
		To see why note that the only change to make compared to the proof of Lemma~\ref{lemma:OPE2} is to consider $\lambda_r+\frac{2}{\gamma}h_1\ln\norm{z}$ instead of $\lambda_r+\gamma h_1\ln\norm{z}$ hereafter. With this new notation at hand we again use the radial-angular decomposition~\eqref{eq:rad_ang_dec} to write that
		\begin{align*}
			J^1_r(z)=& e^{\gamma(\lambda_r+\M_1)} \int_{0}^{+\infty}e^{\gamma\ps{\B^\nu_t,e_1}}\int_0^{2\pi }\norm{\frac{e^{-t-r+i\theta}}{z}-1}^{2}F_1(e^{-t-r+i\theta})M^{\gamma e_1}_{\Y}(dt+r,d\theta).
		\end{align*}
		The asymptotic is thus described in the same fashion as in the proof of Lemma~\ref{lemma:OPE1}, the only difference being the value of $\lambda_r$ considered there. 
		Moreover and similarly to the $\chi=\gamma$ case we can use the remarkable property of the coefficients $B^{(1)}_\gamma$ and $R_{s_1}$
		\[
		R_{s_1}(\alpha_0)=B^{(1)}_\gamma(\alpha_0,\frac2\gamma)R_{s_1}(\hat s_1(\alpha_0-\frac2\gamma h_2))
		\] 
		to infer the following expansion for $\Phi^1$:
		\[
		\Phi^1(z;\bm c)=\norm{z}^{\frac2\gamma\ps{Q-\alpha_0,e_1}}B^{(1)}_\gamma(\alpha_0,\frac2\gamma)R_{s_1}(\hat s_1(\alpha_0-\frac2\gamma h_2))\Phi_{\hat s_1(\alpha_0-\gamma h_2)}(\bm c)+\mathfrak{R}_1(z).
		\]
		Here $\mathfrak{R}^1(z)$ is a lower order term which can be processed like before using fusion asymptotics. This reasoning shows that it is a $o\left(\norm{z}^{\frac2\gamma\ps{Q-\alpha_0,e_1}+\eps}\right)$ where $\eps$ is positive and is independent of $\ps{\alpha_0,e_2}$. In particular for $\ps{Q-\alpha_0,e_2}$ close enough to $0$ we see that $\mathfrak{R}^1=o\left(\norm{z}^{\frac2\gamma\ps{Q-\alpha_0,\rho}}\right)$.
		
		We now turn our attention to $\Phi^2(z;\bm c)$. Then we can use the two-dimensional path decomposition for the planar, drifted Brownian motion  $t\mapsto\X_{t+r}(0)-\X_r(0)+(\alpha_0-Q)t$ to write that, using the notations from Subsection~\ref{subsec:brown_cond_neg} 
		\begin{align*}
			&\E\left[e^{-\sum_{i=1}^2e^{\gamma\ps{\bm c,e_i}}I^i_r(0)}\left(e^{-e^{\gamma\ps{\bm c,e_1}}J^1_r(z)}-1\right)\left(e^{-e^{\gamma\ps{\bm c,e_2}}J^2_r}-1\right)\right]\\
			&=\sum_{s\in W_{1,2}}\lambda_s\int_{\mathcal{C}}e^{\ps{\alpha_0-\hat s\alpha_0,\M}}\E\Big[\prod_{i=1}^2e^{-e^{\gamma\ps{\bm c,e_i}}I^i_r(0)}\left(e^{-e^{\gamma\ps{\Lambda_r+\M,e_1}}\mathrm J_r^i(z;-\M)}-1\right)\Big]d\M,
		\end{align*}
		where $\Lambda_r\coloneqq\X_r(0)+(\alpha_0-Q)r+\frac2\gamma h_1\ln\norm{z}+\bm c$, and with the process from Subsection~\ref{subsec:brown_cond_neg} that enters the definition of \[
		\mathrm J_r^i(z; -\M)\coloneqq \int_{0}^{+\infty}e^{\gamma \ps{\B^{\nu}_t,e_i}}\int_0^{2\pi}\norm{\frac{e^{-t-r+i\theta}}z-1}^{2\ps{h_1,e_i}}F_i(e^{-t-r+i\theta})M^{\gamma e_i}_Y(dt+r,d\theta)
		\] having drift $\nu=\alpha_0-Q$ and being started from $-\M$. 
		
		For fixed $s\in W_{1,2}$, we make the change of variable $\M\leftrightarrow \Lambda_r+\M$ to end up with
		\begin{align*}
			\E\Big[e^{\ps{\hat s\alpha_0-\alpha_0,\Lambda_r}}&\int_{\mathcal{C}+\Lambda_r}e^{\ps{\alpha_0-\hat s\alpha_0,\M}}\prod_{i=1}^2e^{-e^{\gamma\ps{\bm c,e_i}}I^i_r(0)}\left(e^{-e^{\gamma\ps{\M,e_i}}\mathrm J^i_r(z;\Lambda_r-\M)}-1\right)\Big]d\M.
		\end{align*}
		As usual the first exponential term is a Girsanov transform whose effect is to shift the law of $(\X(x))_{\norm{x}>e^{-r}}$ by $\left(\hat s\alpha_0-\alpha_0\right) G(0,\cdot)$, and to shift the law of $\Lambda_r$ to that of $$\Lambda^{s}_r\coloneqq \bm c+\X_r(0)+(\hat s\alpha_0-Q)r +\frac2\gamma h_1\ln\norm{z}.$$ 
		This allows to rewrite the latter as
		\begin{align*}
			\norm{z}^{\frac2\gamma\ps{\hat s\alpha_0-\alpha_0,h_1}}e^{\ps{\hat s\alpha_0-\alpha_0,\bm c}}\E\Big[&\int_{\mathcal{C}+\Lambda_r^{s}}e^{\ps{\alpha_0-\hat s\alpha_0,\M}}\prod_{i=1}^2e^{-e^{\gamma\ps{\bm c,e_i}}\hat I^{i}_{s,r}(0)}\left(e^{-e^{\gamma\ps{\M,e_i}}\mathrm J^i_r(z;\Lambda_r-\M)}-1\right)\Big]d\M
		\end{align*}
		where the notation $\hat I^{i}_{s,r}(0)$ means that $\alpha_0$ is replaced by $\hat s\alpha_0$ in the definition of $I^i_r(0)$.
		
		Now ne easily checks that $\hat s\alpha_0-\frac2\gamma h_1-Q$ belongs to the Weyl chamber $\mathcal{C}_-$ as soon as $\ps{Q-\alpha_0,\rho}<\frac2\gamma$, which we assumed to hold. Moreover if $s=s_1s_2$ we see that \[
		\ps{\hat s_1\hat s_2\alpha_0-\hat s_2\hat s_1\alpha_0,\hat s_1\hat s_2\alpha_0-\frac2\gamma h_1-Q}=\ps{\alpha_0-Q,e_1}^2 +\ps{\alpha_0-Q,e_2}\ps{\alpha_0-Q,\rho}+\frac2\gamma\ps{Q-\alpha_0,e_2}
		\]
		so that $\ps{\hat s_1\hat s_2\alpha_0-\hat s_2\hat s_1\alpha_0,\bm c+\Lambda^{s_1s_2}_r}\to+\infty$ almost surely. As a consequence the reasoning conducted in the proof of~\cite[Theorem 4.4]{Toda_correl1} shows that in this asymptotic the term corresponding to $s=s_1s_2$ in the expectation term will be given by
		\[
		\norm{z}^{\frac2\gamma\ps{Q-\alpha_0,\rho}}R_{s_1s_2}(\alpha_0)\Phi_{\hat s_1\hat s_2\alpha_0-\chi h_1}(\bm c)+l.o.t.
		\]
		Therefore it remains to consider the terms that correspond to $s=s_2s_1$ and $s=s_1s_2s_1$. 
		
		For $s=s_2s_1$ the issue is that the limit $\lim\limits_{r\to+\infty}\hat I^{2}_{s,r}(0)$ does not make sense, and therefore the asymptotic is governed by the fusion asymptotics~\eqref{eq:fusion}. They yield the estimate
		\begin{align*}
			\E\Big[&\int_{\mathcal{C}+\Lambda_r^{s}}e^{\ps{\alpha_0-\hat s\alpha_0,\M}}\prod_{i=1}^2e^{-e^{\gamma\ps{\bm c,e_i}}\hat I^{i}_{s,r}(0)}\left(e^{-e^{\gamma\ps{\M,e_i}}\mathrm J^i_r(z;\Lambda_r-\M)}-1\right)\Big]d\M=o\left(e^{-r\left(\frac{\ps{\alpha_0-Q,e_2}^2}{4}-\eta\right)}\right)
		\end{align*}
		for any positive $\eta$. As a consequence by choosing $r$ large enough compared to $-\ln\norm{z}$ we see that this expectation term is a $o\left(\norm{z}^{\frac2\gamma\ps{Q-\alpha_0,e_2}}\right)$. The corresponding term in the expansion of $\Phi^2$ is thus seen to be a $o\left(\norm{z}^{\frac2\gamma\ps{Q-\alpha_0,\rho}}\right)$ as desired.
		
		As for the case where $s=s_1s_2s_1$, the leading term is $\norm{z}^{\frac2\gamma\ps{Q-\alpha_0,\rho}}$ while the expectation term has the same feature in terms of fusion asymptotics. As a consequence it is easily seen to be a lower order term too.
		
		Eventually this shows that for $\alpha_0$ close to $Q$,
		\begin{equation}\label{eq:last_one_i_hope}
			\begin{split}
				\Phi_{\alpha_0,-\frac2\gamma h_1}(z;\bm c)=&\Phi_{\alpha_0-\frac2\gamma h_1}(\bm c)+\norm{z}^{\frac2\gamma\ps{Q-\alpha_0,e_1}}R_{s_1}(\alpha_0)\Phi_{\hat s_1(\alpha_0-\frac2\gamma h_2)}(\bm c)\\
				&+\norm{z}^{\frac2\gamma\ps{Q-\alpha_0,\rho}}R_{s_1s_2}(\alpha_0)\Phi_{\hat s_1\hat s_2(\alpha_0-\frac2\gamma h_3)}(\bm c)+\mathfrak{R}(z)
			\end{split}
		\end{equation}
		where $\mathfrak{R}$ is a lower order term in $z$.
		
		
		\subsubsection{Extending the validity of Equation~\eqref{eq:last_one_i_hope}}
		We have just provided an expansion of $\Phi_{\alpha_0,-\frac2\gamma h_1}(z;\bm c)$ under the assumption that $\alpha_0$ is close to $Q$. Using the \emph{a priori} form of the expansion given by Theorem~\ref{thm:BPZ} we can actually extend the range of values for which this expansion is valid.
		
		Indeed we can use the same type of arguments that we have used in Subsection~\ref{sub:remainder}. Namely along the same lines we now that based on the description of $\mathcal H$ given by Theorem~\ref{thm:BPZ}, the remainder term $\mathfrak{R}(z)$ is such that
		\begin{equation}\label{eq:JPP2}
			\int_{\R^2}\mathfrak{R}(z)d\bm c=\sum_{i=0}^2 \norm{z}^{\frac2\gamma\ps{Q-\alpha_0,\sum_{j=1}^ie_j}}\mathrm{P}_i(z,\bar z;\alpha_0)
		\end{equation}
		with $\mathrm{P}_i(z,\bar z;\alpha_0)$ a power series in $z,\bar z$, as soon as $\mathcal{H}$ makes sense. 
		We will see below that $\mathcal{H}$ can be defined provided that $\gamma>1$ by choosing $\alpha_0$ close to $q\omega_1+\frac2\gamma\omega_2$ with $\ps{\alpha_0,e_2}<\frac2\gamma$. Therefore the expansion~\eqref{eq:JPP2} is valid provided that $\alpha_0$ satisfies such assumptions.
		
		Now the left-hand side as well as the three first terms in the right-hand side in Equation~\eqref{eq:last_one_i_hope} depend analytically on $\ps{\alpha_0,e_2}$ for the whole range of values for which $\ps{\alpha_0,e_2}<q$. As a consequence the power series $\mathrm{P}_i(z,\bar z;\alpha_0)$ are actually analytic in $\ps{\alpha_0,e_2}$ (in a complex neighborhood of) for $\ps{\alpha_0,e_2}<q$, which allows to extend the validity of Equation~\eqref{eq:JPP2} for $\ps{\alpha_0,e_2}$ close to $q$.
		But in that case we have proved that $\mathfrak{R}$ is such that 
		\[
		\mathfrak{R}(z)=o\left(\norm{z}^{\frac2\gamma\ps{Q-\alpha_0,\rho}}\right),
		\]
		so that $\mathrm{P}_i(0,0;\alpha_0)=0$ for $\alpha_0$ close to $Q$. By analycity this equality extends for $\ps{\alpha_0,e_2}<\frac2\gamma$ too, which is the framework where $\mathcal{H}$ makes sense. Therefore the remainder term $\mathfrak{R}$ does not contribute to the identification of the coefficients coming from the distinct expansions of $\mathcal{H}$ given by Theorem~\ref{thm:BPZ} and Lemma~\ref{lemma:OPE3}.

		All in all, we can conclude that Lemma~\ref{lemma:OPE3} does indeed hold.
	\end{proof}
	

	
	\subsection{Conclusion of the proof of Theorem~\ref{thm:OPE}}
	
	\subsubsection{Analycity of the extension}
	To start with let us consider $\alpha_0$ as in Lemma~\ref{lemma:OPE1} in such a way that the set of weights $\bm\alpha=(\alpha_0,-\gamma h_1,\alpha_1,\alpha_\infty)\in\mathcal{A}_4$ with $\alpha_1=\kappa\omega_2$ is non-empty. Note that this is possible for all values of $\gamma$: indeed if $\alpha_0=\frac2\gamma\omega_1+q\omega_2$, $\alpha_\infty=\frac2\gamma\omega_2+q\omega_1$ and $\alpha_1=q\omega_2$ then we have $\ps{\bm s,\omega_1}+\gamma=\frac{1}{3}(\frac{2}{\gamma}-\gamma)$, which is positive since $\gamma\in(0,\sqrt 2)$. Likewise $\ps{\bm s,\omega_2}=\frac23\left(\frac2\gamma -\gamma\right)$ is positive too.
	
	Then for such an $\alpha_0$, we have provided (under the assumptions of Lemma~\ref{lemma:OPE1}) an expansion of the four-point correlation function $\ps{V_{-\gamma h_1}(z)V_{\alpha_0}(0)V_{\alpha_1}(1)V_{\alpha_\infty}(\infty)}$ similar to that of Theorem~\ref{thm:BPZ}. In particular by linear independence of the hypergeometric functions the coefficients that appear in this expansion can be identified, and therefore we see that we have the equalities
	\begin{equation*}
		\begin{split}
			&\frac{C_\gamma(\alpha_0-\gamma h_2,\alpha_1,\alpha_\infty)}{A_\gamma^{(1)}(-\gamma h_1,\alpha_0,\alpha_1,\alpha_\infty)}=\frac{C_\gamma(\alpha_0-\gamma h_1,\alpha_1,\alpha_\infty)}{B_\gamma^{(1)}(\alpha_0,\gamma)}\quad\text{and}\\
			&\frac{R_{s_2}(\hat s_2(\alpha_0-\gamma h_3))C_\gamma(\hat s_2(\alpha_0-\gamma h_3),\alpha_1,\alpha_\infty)}{A_\gamma^{(2)}(-\gamma h_1,\alpha_0,\alpha_1,\alpha_\infty)}=\frac{C_\gamma(\alpha_0-\gamma h_1,\alpha_1,\alpha_\infty)}{B^{(2)}_\gamma(\alpha_0,\gamma)}\cdot
		\end{split}
	\end{equation*}
	Using analycity of three-point correlation functions as well as analycity of $A_\gamma$ and $B$ we can extend the first equality to the whole range of values of $\alpha_0$ such that $(\alpha_0-\gamma h_i,\alpha_1,\alpha_\infty)\in\mathcal{A}_3$ for $i=1,2$ and the second equality for $(\hat s_2(\alpha_0-\gamma h_3),\alpha_1,\alpha_\infty)$ and $(\alpha_0-\gamma h_1,\alpha_1,\alpha_\infty)$ in $\mathcal{A}_3$. 
	
	We can proceed in a similar way by taking $\alpha_0$ as considered in the statements of Lemma~\ref{lemma:OPE2}. Like above the set of $(\alpha_0,-\gamma h_1,\alpha_1,\alpha_\infty)$ satisfying the assumptions of Lemma~\ref{lemma:OPE2} is non-empty for any value of $\gamma\in(0,\sqrt 2)$, by taking $\alpha_0$ close to $q\omega_1+\frac2\gamma\omega_2$, while  $\alpha_\infty=\frac2\gamma\omega_1+q\omega_2$ and $\alpha_1=q\omega_2$. This shows that
	\begin{align*}
		&\frac{R_{s_1}(\hat s_1(\alpha_0-\gamma h_2))C_\gamma(\hat s_1(\alpha_0-\gamma h_2),\alpha_1,\alpha_\infty)}{A_\gamma^{(1)}(-\gamma h_1,\alpha_0,\alpha_1,\alpha_\infty)}=\frac{C_\gamma(\alpha_0-\gamma h_1,\alpha_1,\alpha_\infty)}{B^{(1)}_\gamma(\alpha_0,\gamma)}\quad\text{and}\\
		&\frac{C_\gamma(\alpha_0-\gamma h_3,\alpha_1,\alpha_\infty)}{A_\gamma^{(2)}(-\gamma h_1,\alpha_0,\alpha_1,\alpha_\infty)}=\frac{C_\gamma(\alpha_0-\gamma h_1,\alpha_1,\alpha_\infty)}{B^{(2)}_\gamma(\alpha_0,\gamma)}
	\end{align*}
	as soon as $\alpha_0$ is such that $(\hat s(\alpha_0-\gamma h_i),\alpha_1,\alpha_\infty)$ for $i=2$ or $i=3$ and $(\alpha_0-\chi h_1,\alpha_1,\alpha_\infty)$ belong to $\mathcal{A}_3$.
	
	Now let us consider the map \[
	F:\beta\mapsto C_\gamma(\beta-\gamma e_1,\alpha_1,\alpha_\infty)\frac{A_\gamma^{(1)}(-\gamma h_1,\beta+\gamma h_2,\alpha_1,\alpha_\infty)}{B^{(1)}(\beta+\gamma h_2)}\cdot
	\]
	Then this map is analytic in a complex neighborhood of the set of $\beta$ such that $(\beta-\gamma e_1,\alpha_1,\alpha_\infty)\in\mathcal A_3$. Moreover we have seen above that it coincides with $C_\gamma(\beta,\alpha_1,\alpha_\infty)$ provided that $\beta$ is in addition such that $(\beta,\alpha_1,\alpha_\infty)$ belongs to $\mathcal{A}_3$. Likewise we have proved that it is equal to $R_{s_1}(\hat s_1(\alpha_0-\gamma h_2))C_\gamma(\hat s_1(\alpha_0-\gamma h_2),\alpha_1,\alpha_\infty)$ as soon as $(\hat s_1\beta,\alpha_1,\alpha_\infty)\in\mathcal A_3$.
	This shows that the map defined by setting $R_s(\hat s\beta)C_\gamma(\hat s\beta,\alpha_1,\alpha_\infty)$ where $s\beta\in Q+\mathcal{C}_-$ is actually equal to $F$, and is therefore analytic over a complex neighborhood of the subset of $\beta$ such that $(\beta,\alpha_1,\alpha_\infty)$ or $(\hat s_1\beta,\alpha_1,\alpha_\infty)$ belongs to $\mathcal{A}_3$.
	
	The same reasoning applies when $s=s_2$. Namely we have obtained the equality
	\[
	R_s(\hat s\beta)C_\gamma(\hat s\beta,\alpha_1,\alpha_\infty)=G(\beta),\quad G(\beta)\coloneqq C_\gamma(\beta-\gamma \rho,\alpha_1,\alpha_\infty)\frac{A_\gamma^{(2)}(-\gamma h_1,\beta+\gamma h_3,\alpha_1,\alpha_\infty)}{B^{(2)}(\beta+\gamma h_3)}
	\]
	for $s\in\{Id,s_2\}$ such that $(\hat s\beta,\alpha_1,\alpha_\infty)\in\mathcal{A}_3$, depending on the values of $\alpha_0$ around $\frac2\gamma\omega_1+q\omega_2$. 
	Like above this shows that the map $R_s(\hat s\beta)C(\hat s\beta,\alpha_1,\alpha_\infty)$ is analytic in a complex neighborhood of the subset of $\beta$ such that $(\beta,\alpha_1,\alpha_\infty)$ or $(\hat s_2\beta,\alpha_1,\alpha_\infty)$ belongs to $\mathcal{A}_3$. 
	
	Let us now look at what happens for other elements of the Weyl group $W$. It is readily seen that for any fixed $s\in W$, the map
	\[
	\beta\mapsto R_{s}(\beta)C_\gamma(\hat s\beta,\alpha_1,\alpha_\infty)
	\]
	is analytic in a complex neighborhood of $\mathcal{U}_s(\alpha_1,\alpha_\infty)\coloneqq\{\beta\in\R^2, (\hat s\beta,\alpha_1,\alpha_\infty)\in\mathcal{A}_3\}$. Moreover we have proved above that for $i=1,2$ the map defined by setting 
	\[
	\beta\mapsto R_{s_i}(\hat s\beta)C_\gamma(\hat s_i\hat s\beta,\alpha_1,\alpha_\infty)
	\]
	was analytic in a complex neighborhood of $\mathcal{U}_s(\alpha_1,\alpha_\infty)\cup\mathcal{U}_{s_is}(\alpha_1,\alpha_\infty)$. Because the reflection coefficients enjoy the property that $R_{s_i}(\hat s \beta)R_{s}(\beta)=R_{s_is}(\beta)$ we therefore see that the map from Equation~\eqref{eq:ext_3ptsbis} is actually analytic in a complex neighborhood of $\mathcal{U}_s(\alpha_1,\alpha_\infty)\cup\mathcal{U}_{s_is}(\alpha_1,\alpha_\infty)$. As a consequence this map is seen to be analytic in a complex neighbourhood of
	\[
	\bigcup_{w\in <s_1,s_2>} \mathcal{U}_w(\alpha_1,\alpha_\infty)
	\]
	where $<s_1,s_2>$ is the group generated by $s_1$ and $s_2$. The latter being nothing but $W$, we infer that the map is analytic in a complex neighborhood of $\mathcal{U}(\alpha_1,\alpha_\infty)$. This shows that as desired, the extension of the three-points correlation functions from Equation~\eqref{eq:ext_3ptsbis} is analytic in a
	complex neighborhood of $\mathcal{U}(\alpha_1,\alpha_\infty)$.
	
	\subsubsection{Operator Product Expansions and shift equations}
	Having proved that the extension $\alpha\mapsto R_s(\alpha)C_\gamma(\hat s\alpha,\alpha_1,\alpha_\infty)$ is analytic in a complex neighborhood of $\mathcal{U}(\alpha_1,\alpha_\infty)$, we can denote by $C_\gamma(\alpha,\alpha_1,\alpha_\infty)$ this extension.
	
	Then as explained above, we know that the equalities
	\begin{equation}\label{eq:shift_proba1}
		\frac{C_\gamma(\alpha_0-\gamma h_2,\alpha_1,\alpha_\infty)}{A_\gamma^{(1)}(-\gamma h_1,\alpha_0,\alpha_1,\alpha_\infty)}=\frac{C_\gamma(\alpha_0-\gamma h_1,\alpha_1,\alpha_\infty)}{B^{(1)}_\gamma(\alpha_0,\gamma)}\quad\text{and}
	\end{equation}
	\begin{equation}\label{eq:shift_proba2}
		\frac{C_\gamma(\alpha_0-\gamma h_3,\alpha_1,\alpha_\infty)}{A_\gamma^{(2)}(-\gamma h_1,\alpha_0,\alpha_1,\alpha_\infty)}=\frac{C_\gamma(\alpha_0-\gamma h_1,\alpha_1,\alpha_\infty)}{B^{(2)}_\gamma(\alpha_0,\gamma)}
	\end{equation}
	hold true in some open subset of $\mathcal{U}(\alpha_1,\alpha_\infty)$. By analycity of the left and right-hand sides this equality extends to the whole range of values for which it makes sense. Put differently we recover Equation~\eqref{eq:OPE} in the case where $\chi=\gamma$.
	
	When $\chi=\frac2\gamma$, the same reasoning remains valid. Namely thanks to Lemma~\ref{lemma:OPE3} we know that as soon as the set of $(\alpha_0,-\frac2\gamma h_1,\alpha_1,\alpha_\infty)\in\mathcal{A}_4$ with $\alpha_0$ as in the statement of Lemma~\ref{lemma:OPE3} is non-empty we have the equalities
	\begin{equation}\label{eq:shift_proba3}
		\frac{C_\gamma(\alpha_0-\frac2\gamma h_2,\alpha_1,\alpha_\infty)}{A_\gamma^{(1)}(-\frac2\gamma h_1,\alpha_0,\alpha_1,\alpha_\infty)}=\frac{C_\gamma(\alpha_0-\frac2\gamma h_1,\alpha_1,\alpha_\infty)}{B^{(1)}_\gamma(\alpha_0,\frac2\gamma)}\quad\text{and}
	\end{equation}
	\begin{equation}\label{eq:shift_proba4}
		\frac{C_\gamma(\alpha_0-\frac2\gamma h_3,\alpha_1,\alpha_\infty)}{A_\gamma^{(2)}(-\frac2\gamma h_1,\alpha_0,\alpha_1,\alpha_\infty)}=\frac{C_\gamma(\alpha_0-\frac2\gamma h_1,\alpha_1,\alpha_\infty)}{B^{(2)}_\gamma(\alpha_0,\frac2\gamma)}\cdot
	\end{equation}
	Therefore to conclude for the proof of Theorem~\ref{thm:OPE} it suffices to check that for any fixed value of $\gamma\in(0,\sqrt 2)$, if the set of  $(\alpha_0,-\frac2\gamma h_1,\alpha_1,\alpha_\infty)\in\mathcal{A}_4$ is non-empty then we can find $(\alpha_0,-\frac2\gamma h_1,\alpha_1,\alpha_\infty)$ that meets the requirements of Lemma~\ref{lemma:OPE3}. Now one can check that as soon as $(\alpha_0,-\frac2\gamma h_1,\alpha_1,\alpha_\infty)\in\mathcal{A}_4$, we have $\ps{\bm s,\omega_1}<-\frac{1}{3}\left(\gamma+\frac2\gamma\right)$ so that for $\gamma\leq1$ (which implies that $-\frac{1}{3}\left(\gamma+\frac2\gamma\right)\leq-\gamma$) this set is empty. Conversely if $\gamma>1$ by choosing $\alpha_0=\alpha_\infty=q\omega_1+\frac2\gamma\omega_2$ while $\alpha_1=q\omega_2$ then $\ps{\bm s,\omega_1}=-\frac{1}{3}\left(\gamma+\frac2\gamma\right)>-\gamma$ so that this set is non-empty, and by choosing the weights close to the above choice the assumptions of Lemma~\ref{lemma:OPE3} are fulfilled. This wraps up the proof of Theorem~\ref{thm:OPE}.

	
	
	
	\section{Shift equations and computation of the three-point correlation functions}\label{sec:toda_end}
	This concluding section brings together the building blocks unveiled in the previous sections to provide a proof of the main statement of the present document, Theorem~\ref{thm:main_result}. To do so we first provide some background on the special function $\Upsilon$ and then bring the proof of Theorem~\ref{thm:main_result} to its end.
	
	
	
	\subsection{On the $\Upsilon$ function}
	The expression~\eqref{eq:fali} proposed for the three-point correlation functions involves the special function $\Upsilon$, which is ubiquitous in Liouville theory and more generally in Toda CFTs. One of the reasons why it is so is that it enjoys remarkable shift equations. These shift equations take the form\footnote{Note that the special function considered here differs from the standard expression because of the convention on the length of the simple roots. One recovers the usual expression via the correspondence $\Upsilon(z)=\Upsilon_{\frac{\gamma}{\sqrt 2}}(\frac1{\sqrt 2}z)$}:
	\begin{equation}\label{eq:shift_upsilon}
		\begin{split}
			\Upsilon\left(z+\chi\right)=l\left(\frac\chi2 z\right)\left(\frac\chi{\sqrt 2}\right)^{1-\chi z}\Upsilon\left(z\right),
		\end{split}
	\end{equation}
	valid for $z\in\C$ and $\chi\in\{\gamma,\frac 2\gamma\}$.
	When $0<\mathfrak{R}(z)<q=\gamma+\frac2\gamma$, this special function admits the integral representation
	\begin{equation}
		\ln\Upsilon(z)=\int_0^{+\infty}\left(\left(\frac {q}{2}-z\right)^2\frac{e^{-t}}2-\frac{\left(\text{sinh}\left(\left(\frac {q}2-z\right)\frac t{2\sqrt 2}\right)\right)^2}{\text{sinh}\left(\frac {t\gamma}{2\sqrt 2}\right)\text{sinh}\left(\frac {2\sqrt2t}\gamma\right)}\right)\frac{dt}{t},
	\end{equation}
	while the shift equations~\eqref{eq:shift_upsilon} allow to continue it to an analytic function over $\C$, with no poles and zeros given by the $(-\gamma\N-\frac2\gamma\N)\cup(q+\gamma\N+\frac2\gamma\N)$, and which satisfies $\Upsilon(q-z)=\Upsilon(z)$.

	The formula proposed in~\cite{FaLi1} for $\mathfrak{sl}_n$ Toda three-point correlation functions closely resembles that of the DOZZ formula~\cite{DO94, ZZ96, KRV_DOZZ} for Liouville theory. And actually the DOZZ formula can be recovered from the expression of the $\mathfrak{sl}_3$ Toda three-point correlation functions as the following statement discloses:
	\begin{lemma}\label{lemma:res_correl}
		Assume that $\ps{\bm s,\omega_1}\to0$ while $\ps{\bm s,\omega_2}$ remains positive. Then
		\begin{equation}
			\mathfrak{C}_\gamma(\alpha_0,\alpha_1,\alpha_\infty)\sim \frac1{\sqrt 2\ps{\bm s,\omega_1}}C^{\text{DOZZ}}_{\sqrt 2\gamma}\left(\frac{\ps{ \alpha_0,e_2}}{\sqrt 2},\frac{\ps{ \alpha_1,e_2}}{\sqrt 2},\frac{\ps{ \alpha_\infty,e_2}}{\sqrt 2}\right)\cdot
		\end{equation}
	\end{lemma}
	\begin{proof}
		As $\ps{\bm s,\omega_1}\to0$, the prefactor converges to $\left(\pi\mu l\left(\frac{\gamma^2}{2}\right)\left(\frac{\gamma^2}{2}\right)^{2-\gamma^2}\right)^{\frac{\ps{2Q-\bar\alpha,e_2}}{\gamma}}$ where we have used that $\ps{\bm s,\omega_1}=0$. Likewise by looking at each factor appearing in the product $\prod_{1\leq j,k\leq 3}\displaystyle \Upsilon\left(\frac\kappa 3 + \ps{\alpha_0-Q,h_j})+\ps{\alpha_\infty-Q,h_k}\right)$ one can check that
		\begin{align*}
			\mathfrak{C}_\gamma(\alpha_0,\alpha_1,\alpha_\infty)&\sim\frac{1}{\ps{\bm s,\omega_1}}\left(\pi\mu l\left(\frac{\gamma^2}{2}\right)\left(\frac{\gamma^2}{2}\right)^{2-\gamma^2}\right)^{\frac{\ps{2Q-\bar\alpha,e_2}}{\gamma}}\\
			&\frac{\Upsilon'(0)\Upsilon(\kappa)\Upsilon(\ps{Q-\alpha_0,e_2})\Upsilon(\ps{Q-\alpha_\infty,e_2})}{\Upsilon\left(\frac{\ps{\bar\alpha-2Q,e_2}}2\right)\Upsilon\left(\frac{\ps{\bar\alpha,e_2}}2-\ps{\alpha_0,e_2}\right)\Upsilon\left(\frac{\ps{\bar\alpha,e_2}}2-\ps{\alpha_1,e_2}\right)\Upsilon\left(\frac{\ps{\bar\alpha,e_2}}2-\ps{\alpha_\infty,e_2}\right)}\cdot
		\end{align*}
		Using the fact that $\Upsilon(\ps{Q-\alpha_\infty,e_2})=\Upsilon(\ps{\alpha_\infty,e_2})$, via our convention on the Upsilon function the latter is nothing but the DOZZ formula, up to the normalization factor $\sqrt 2$, which shows that $\mathfrak{C}_\gamma(\alpha_0,\alpha_1,\alpha_\infty)\sim \frac1{\ps{\bm s,\omega_1}}\frac1{\sqrt 2}C^{\text{DOZZ}}_{\sqrt 2\gamma}\left(\frac{\ps{ \alpha_0,e_2}}{\sqrt 2},\frac{\ps{ \alpha_1,e_2}}{\sqrt 2},\frac{\ps{ \alpha_\infty,e_2}}{\sqrt 2}\right)$.
	\end{proof}
	A counterpart statement also holds for the probabilistically defined correlation functions:
	\begin{lemma}\label{lemma:res_correlbis}
		Assume that $\alpha_0,\alpha_\infty \in Q+\mathcal{C}_-$ are close enough to $Q$ so that $\kappa\coloneqq\ps{2Q-\alpha_0-\alpha_\infty,3\omega_1}$ is such that $\kappa<q$. 
		Then, as $\eps\to0$ with $\eps>0$,
		\begin{equation}
			C_\gamma(\alpha_0,(\kappa+\eps)\omega_2,\alpha_\infty)\sim \frac{3}{\sqrt 2\eps}C^{\text{DOZZ}}_{\sqrt 2\gamma}\left(\frac{\ps{ \alpha_0,e_2}}{\sqrt 2},\frac{\ps{ \alpha_1,e_2}}{\sqrt 2},\frac{\ps{ \alpha_\infty,e_2}}{\sqrt 2}\right)\cdot
		\end{equation}
	\end{lemma}
	\begin{proof}
		Under the assumptions made on the weights, we see that the three-point correlation functions $C_\gamma(\alpha_0,(\kappa+\eps)\omega_2,\alpha_\infty)$ admits the probabilistic representation
		\[
		\left ( \prod_{i=1}^{2} \frac{\Gamma \left(\frac{\ps{\bm s,\omega_i}}{\gamma}\right)\mu_i^{-\frac{\ps{\bm s,\omega_i}}{\gamma}}}\gamma   \right )\E \left [      \prod_{i=1}^{2}\left(\int_\C \frac{\norm{1-x_i}^{-\gamma\ps{\alpha_1,e_i}}}{\norm{x_i}^{\gamma\ps{\alpha_0,e_i}}}\norm{x_i}_+^{\gamma\ps{\alpha_0+\alpha_1+\alpha_\infty,e_i}}M^{\gamma e_i}(d^2x_i)\right)^{-\frac{\ps{\bm s,\omega_i}}{\gamma}} \right ],
		\]
		where $\ps{\bm s,\omega_1}=\frac\eps3$ while $ \ps{\bm s,\omega_2}$ converges towards $\ps{2Q-\alpha_0-\alpha_\infty,e_1}$. Therefore as $\eps\to 0$ we see that $C_\gamma(\alpha_0,(\kappa+\eps)\omega_2,\alpha_\infty)$ is asymptotically equivalent to
		\[
		\frac{3}{\eps} \frac{\Gamma \left(\frac{\ps{\bm s,\omega_2}}{\gamma}\right)\mu_i^{-\frac{\ps{\bm s,\omega_2}}{\gamma}}}\gamma \E \left [      \left(\int_\C \frac{\norm{1-x_2}^{-\gamma\kappa }}{\norm{x_2}^{\gamma\ps{\alpha_0,e_2}}}\norm{x_2}_+^{\gamma\ps{\alpha_0+\alpha_1+\alpha_\infty,e_2}}M^{\gamma e_2}(d^2x_2)\right)^{-\frac{\ps{\bm s,\omega_2}}{\gamma}} \right].
		\]
		To check that the expectation term does indeed coincide with the probabilistic representation of the DOZZ formula proved in~\cite{KRV_DOZZ}, it remains to ensure that $2\ps{\bm s,\omega_2}=\ps{\alpha_0+\alpha_1+\alpha_\infty-2Q,e_2}$ where the factor $2$ comes from the fact that the simple root $e_2$ has norm $\sqrt 2$. Using the fact that $2\omega_2-e_2=\omega_1$ the above equality boils down to $\ps{\alpha_0+\alpha_1+\alpha_\infty-2Q,\omega_1}=0$, which was our assumption.
	\end{proof}
	The second statement about the proposed expression for the three-point correlation function is concerned with shift equations, which take the form:
	\begin{equation}\label{eq:shift_fali}
		\frac{\mathfrak{C}_\gamma(\alpha_0-\chi h_{i+1},\alpha_1,\alpha_\infty)}{\mathfrak{C}_\gamma(\alpha_0-\chi h_1,\alpha_1,\alpha_\infty)}=\frac{A_\gamma^{(i)}(-\chi h_1,\alpha_0,\alpha_1,\alpha_\infty)}{B^{(i)}(\alpha_0)}\quad\text{for $i=1,2$ and $\chi\in\{\gamma,\frac2\gamma\}$}.
	\end{equation}
	Equation~\eqref{eq:shift_fali} follows from the shift equation~\eqref{eq:shift_upsilon} of the $\Upsilon$ function after some elementary but lengthy computations.
	

	
	\subsection{From shift equations to three-point correlation functions}
	We are now in position to conclude for the proof of our main statement, Theorem~\ref{thm:main_result}. Indeed under the assumption that $\gamma>1$, we have already seen along the proof of Theorem~\ref{thm:OPE} that the set of weights $\bm\alpha=(\alpha,\alpha_0,\alpha_1,\alpha_\infty)\in\mathcal{A}_4$ such that $\alpha=-\chi h_1$ with $\chi\in\{\gamma,\frac2\gamma\}$ and $\alpha_1=\kappa\omega_2$ for $\kappa<q$ is non-empty and open in $\mathcal{A}_4$.
	
	Now we have seen in Theorem~\ref{thm:BPZ} that for such weights, the associated four-point correlation functions can be expressed in terms of hypergeometric functions as follows:
	\begin{equation}
		\begin{split}
			&\ps{V_{\alpha}(z)V_{\alpha_0}(0)V_{\alpha_1}(1)V_{\alpha_\infty}(\infty)}=\norm{z}^{\chi\ps{h_1,\alpha_0}}\norm{z-1}^{\frac{\chi\kappa}3}\mathcal H(z),\quad\text{where}\\
			&\mathcal H(z)=C_\gamma(\alpha+\alpha_0,\alpha_1,\alpha_\infty)\left(\norm{\mathcal H_0(z)}^2+\sum_{i=1}^2A_\gamma^{(i)}(\alpha,\alpha_0,\alpha_1,\alpha_\infty)\norm{\mathcal H_i(z)}^2\right).
		\end{split}
	\end{equation}
	
	On the other hand we proved in Theorem~\ref{thm:OPE} that under the same assumptions on the weights, this function $\mathcal{H}$ could be expressed using different coefficients that involve three-point correlation functions: 
	\begin{equation*}
		\begin{split}
			&\ps{V_{\alpha}(z)V_{\alpha_0}(0)V_{\alpha_1}(1)V_{\alpha_\infty}(\infty)}=\norm{z}^{\chi\ps{h_1,\alpha_0}}\norm{z-1}^{\frac{\chi\kappa}3}\mathcal H(z),\quad\text{with}\\
			&\mathcal H(z)=\sum_{i=1}^3B_\gamma^{(i-1)}(\alpha_0,\chi)C_\gamma(\alpha_0-\chi h_i,\alpha_1,\alpha_\infty)\norm{\mathcal H_{i-1}(z)}^2\end{split}
	\end{equation*}
	and where $C_\gamma(\alpha_0-\chi h_i,\alpha_1,\alpha_\infty)$ denotes the extension of $C_\gamma$ defined by setting $C_\gamma(\alpha,\alpha_1,\alpha_\infty)\coloneqq R_s(\alpha)C_\gamma(\hat s\alpha,\alpha_1,\alpha_\infty)$ where $s\in W$ is such that $\hat s\alpha\in Q+\mathcal{C}_-$. This extension is analytic in virtue of Theorem~\ref{thm:OPE}.
	
	Combining these two equalities and because the hypergeometric functions are linearly independent, we get for $i=1,2$ and $\chi\in\{\gamma,\frac2\gamma\}$, the following shift equations:
	\begin{equation}
		\begin{split}
			&\frac{C_\gamma(\alpha_0-\chi h_{i+1},\alpha_1,\alpha_\infty)}{C_\gamma(\alpha_0-\chi h_1,\alpha_1,\alpha_\infty)}=\frac{A_\gamma^{(i)}(-\chi h_1,\alpha_0,\alpha_1,\alpha_\infty)}{B^{(i)}(\alpha_0)}\cdot
		\end{split}
	\end{equation}
	In particular this shift equation allows to extend the map $\alpha\mapsto C_\gamma(\alpha,\alpha_1,\alpha_\infty)$ to an open complex neighborhood of $\R^2$ of the form $\R^2\times (-\delta,\delta)^2$ on which it is analytic.
	
	Now we have seen before that the expression $\FaLi$ proposed for the three-point correlation functions satisfies the very same set of shift equations. As a consequence the map defined by setting $\alpha\mapsto\frac{\mathfrak{C}_\gamma(\alpha,\alpha_1,\alpha_\infty)}{C_\gamma(\alpha,\alpha_1,\alpha_\infty)}$ is analytic over $\R^2\times (-\delta,\delta)^2$ and is periodic with periods $\chi e_i$ for $\chi\in\{\gamma,\frac2\gamma\}$ and $i=1,2$. 
	Now as soon as $\frac{\gamma^2}{2}$ is not a rational number, it is readily seen that the set $\gamma\Z e_1+\gamma\Z e_2+\frac2\gamma\Z e_1+\frac2\gamma\Z e_2$ is dense in $\R^2$. Therefore under this assumption that $\frac{\gamma^2}{2}\not\in\Q$, we get that $C_\gamma(\alpha_0,\alpha_1,\alpha_\infty) = a_\gamma(\alpha_1,\alpha_\infty)\mathfrak{C}_\gamma(\alpha_0,\alpha_1,\alpha_\infty)$ for some constant $a_\gamma(\alpha_1,\alpha_\infty)$ independent of $\alpha_0$. By symmetry in the weights $\alpha$ we see that this constant is actually independent of these three variables. It can be evaluated using Lemmas~\ref{lemma:res_correl} and~\ref{lemma:res_correlbis}, thanks to which it is found to be equal to $1$.
	Therefore we have proved that $C_\gamma(\alpha_0,\alpha_1,\alpha_\infty) = \mathfrak{C}_\gamma(\alpha_0,\alpha_1,\alpha_\infty)$ for $\frac{\gamma^2}{2}\in(1,\sqrt 2)\setminus\Q$. Because both quantities are continuous in the variable $\gamma$ this equality extends to all values of $\gamma\in[1,\sqrt2)$. This concludes for the proof of our main statement, Theorem~\ref{thm:main_result}.


	
	\bibliography{biblio}
	\bibliographystyle{plain}

\end{document}